\documentclass[11pt]{amsart}
\usepackage{amsmath}
\usepackage{latexsym}
\usepackage{amssymb}
\usepackage{amscd}

\usepackage{amsfonts}
\usepackage[all]{xy}
\usepackage{}

\newtheorem{theorem}{Theorem}

\newtheorem{theo}{Theorem}[section]
\newtheorem{lemma}[theo]{Lemma}
\newtheorem{defi}[theo]{Definition}
\newtheorem{prop}[theo]{Proposition}

\newtheorem{cor}[theo]{Corollary}
\newtheorem{remark}[theo]{Remark}

\newtheorem{example}[theo]{Example}
\numberwithin{equation}{section}

\def\coh{\operatorname{coh}}
\def\Qcoh{\operatorname{Qcoh}}
\def\lto{\longrightarrow}

\def\E{{\mathcal E}}
\def\F{{\mathcal F}}

\def\O{{\mathcal O}}

\def\D{{\mathcal{D}}}

\def\T{{\mathcal{T}}}

\def\pre-tr{\operatorname{pre-tr}}

\def\Hom{\operatorname{Hom}}

\newcommand{\bL}{{\mathbf L}}
\newcommand{\bR}{{\mathbf R}}
\newcommand{\bbZ}{{\mathbb Z}}
\newcommand{\bbN}{{\mathbb N}}
\newcommand{\bbP}{{\mathbb P}}
\newcommand{\bp}{\mathbf{p}}

\newcommand{\cJ}{{\mathcal J}}

\newcommand{\cF}{{\mathcal F}}
\newcommand{\cG}{{\mathcal G}}
\newcommand{\cO}{{\mathcal O}}
\newcommand{\cP}{{\mathcal P}}
\newcommand{\cL}{{\mathcal L}}

\newcommand{\cN}{{\mathcal N}}
\newcommand{\cD}{{\mathcal D}}
\newcommand{\cV}{{\mathcal V}}
\newcommand{\cA}{{\mathcal A}}
\newcommand{\cB}{{\mathcal B}}
\newcommand{\cI}{{\mathcal I}}
\newcommand{\cC}{{\mathcal C}}
\newcommand{\cE}{{\mathcal E}}
\newcommand{\cW}{{\mathcal W}}
\newcommand{\cU}{{\mathcal U}}

\newcommand{\cS}{{\mathcal S}}
\newcommand{\cT}{{\mathcal T}}
\newcommand{\cK}{{\mathcal K}}

\newcommand{\Perf}{\operatorname{Perf}}

\newcommand{\Ker}{\operatorname{Ker}}
\newcommand{\coker}{\operatorname{Coker}}
\newcommand{\im}{\operatorname{Im}}

\newcommand{\Ext}{\operatorname{Ext}}

\newcommand{\colim}{\underrightarrow{\operatorname{colim}}\,}
\newcommand{\holim}{\underleftarrow{\operatorname{holim}}\,}

\newcommand{\tors}{\operatorname{tors}}

\newcommand{\Ho}{{H^0}}
\newcommand{\SF}{\mathcal{S}\mathcal{F}}
\newcommand{\id}{\operatorname{id}}

\newcommand{\tr}{\operatorname{tr}}

\newcommand{\Ac}{\mathcal{A}c}
\newcommand{\Hog}{H^{*}}
\newcommand{\Mod}{\mathcal{M}od\operatorname{-}}
\newcommand{\ModU}{\mathcal{M}od_{\mathbb{U}}\operatorname{-}}
\newcommand{\ModUs}{\mathcal{M}od^{\str}_{\mathbb{U}}\operatorname{-}}

\newcommand\RHOM{\mathcal{R}\mathcal{H}om}
\newcommand\EXT{\mathcal{E}xt}
\newcommand{\cd}{\mathcal{D}}

\newcommand{\Ob}{\operatorname{Ob}}

\newcommand{\DGcat}{\operatorname{DGcat}}
\newcommand{\Hqe}{\operatorname{Hqe}}
\newcommand{\rep}{\operatorname{rep}}
\newcommand{\can}{\operatorname{can}}

\newcommand{\hocolim}{\underrightarrow{\operatorname{hocolim}}\,}
\newcommand{\dgcat}{\mathbf{DGcat}}
\newcommand{\prfdg}{\mathcal{P}erf}

\newcommand{\umod}{\operatorname{Mod}\operatorname{-}}

\newcommand{\QMod}{\operatorname{QMod}}

\newcommand{\Gr}{\operatorname{Gr}}
\newcommand{\Tors}{\operatorname{Tors}}
\newcommand{\QGr}{\operatorname{QGr}}

\newcommand{\Proj}{\mathbb{P}\operatorname{roj}}

\newcommand{\op}{\operatorname{op}}
\newcommand{\str}{\operatorname{str}}

\newcommand{\fA}{\mathfrak{A}}
\newcommand{\fE}{\mathfrak{E}}
\newcommand{\Loc}{\operatorname{Loc}}

\title[]
{Uniqueness of enhancement for triangulated categories}

\author[]{Valery A.~Lunts}
\address{Department of Mathematics, Indiana University,
Bloomington, IN 47405, USA} \email{vlunts@indiana.edu}
\author[]{Dmitri O.~Orlov}
\address{Steklov Mathematical Institute, 8 Gubkina St. Moscow, Russia}
\email{orlov@mi.ras.ru}

\thanks{The first named author was partially supported by the NSA
grant H98230-05-1-0050. The second named author was partially
supported by grant RFFI 08-01-00297 and grant NSh-1987.2008.1}

\begin{document}

\begin{abstract}
The paper contains general results on the uniqueness of a DG enhancement
for triangulated categories. As a consequence we obtain such uniqueness for
the unbounded categories of quasi-coherent sheaves, for the triangulated categories of perfect complexes,
and for the bounded derived categories of coherent sheaves on quasi-projective schemes. If a scheme is projective
then we also prove a strong uniqueness for the triangulated category of perfect complexes and
 for the bounded derived categories of coherent sheaves. These results directly imply that fully faithful functors
from the bounded derived categories of coherent sheaves and  the triangulated categories of perfect complexes on
 projective schemes can be represented by objects on the product.
\end{abstract}


\maketitle

\tableofcontents

\section*{Introduction}

Triangulated categories were invented about 50 years ago as a
convenient tool to do homological algebra. Yet it has been known for
some time now that the notion of a triangulated category is not
satisfactory: morphisms between objects in such a category are
usually given by cohomology groups of certain complexes, and you
forget too much by passing to cohomology. The problem is that the
cone of a morphism is not functorial in triangulated categories.
Let us give a couple of "frustrating" examples.

Given a triangulated category $\cT$ we can consider the category $\cT
^{\vee}$ of all cohomological functors from $\cT$ to the category of
abelian groups. We ``know'' that $\cT ^{\vee}$ should also be a triangulated
category, which one might call the dual of $\cT.$ However one cannot
prove that $\cT ^{\vee}$ is indeed triangulated.

Another example is the operation of the tensor product of two
triangulated categories (which should also be a triangulated
category) that cannot be performed without an extra data \cite{BLL}.

Thus we like to consider a triangulated category $\cT$ together with
an {\it enhancement} $\cB,$ which has the same objects as $\cT$ and
the set of morphisms between two objects in $\cB$ is a complex. One
recovers morphisms in $\cT$ by taking the cohomology $H^0$ of the
corresponding morphism complex in $\cB.$ Thus $\cB$ is a DG category and
$\cT$ is the its {\it homotopy} category $\Ho (\cB).$ The notion
of a triangulated category lifts to the DG world \cite{BK}: one has
{\it pretriangulated} DG categories, in which the cone of a morphism
is functorial!

Let $\cT$ be a triangulated category. An {\it
enhancement} of $\cT$ is a pair $(\cB, \varepsilon),$ where $\cB$ is a
pretriangulated DG category and $\varepsilon:\Ho(\cB)\stackrel{\sim}{\lto} \cT$
is an equivalence of  triangulated categories.
There are questions of existence and uniqueness of enhancement for a
given triangulated category.

The category $\cT$ has a  unique enhancement if it
has one and for any two enhancements $(\cB, \varepsilon)$ and $(\cB', \varepsilon')$ of
$\cT$ the DG categories $\cB$ and $\cB'$ are quasi-equivalent, i.e.
there exists a quasi-functor $\phi: \cB \lto \cB'$ which induces an
equivalence $\Ho(\phi):\Ho(\cB)\stackrel{\sim}{\lto} \Ho(\cB').$ In
this case the enhancements $(\cB, \varepsilon)$ and $(\cB', \varepsilon')$ are called
{\it equivalent}.

Enhancements $(\cB, \varepsilon)$ and $(\cB', \varepsilon')$ of $\cT$ are called {\it strongly
equivalent} if there exists a quasi-functor $\phi:\cB \to \cB'$ such
that the functors $\varepsilon' \cdot\Ho(\phi)$ and $\varepsilon$ are isomorphic.

It is important to know that an enhancement exists and is unique
for a given triangulated category $\cT,$ because then its choice is not considered
as an extra data. For example, in string theory categories of D-branes
arise as  DG categories (actually $A_{\infty}$\!-categories),
homotopy categories of which are equivalent to derived categories of coherent sheaves
on some projective varieties. It is very useful to know that these equivalences can be lifted to the ``DG level'' as well, i.e.
DG categories of D-branes are quasi-equivalent
to  natural enhancements of the derived categories of coherent sheaves.

We fix a field $k$ and all our categories are $k$\!-linear.  Our main
results are the following.

\begin{theorem}\label{imain}{\rm (=Theorem \ref{main}).}
Let $\cA$ be a small category which we consider as a DG category
 and $L\subset D(\cA)$ be a
localizing subcategory with the quotient functor $\pi :D(\cA)\to
D(\cA)/L$ that has a right adjoint (Bousfield localization) $\mu.$
Assume that the following conditions hold
\begin{enumerate}
\item[a)]  for every $Y\in \cA$ the object  $\pi (h^Y)\in D(\cA)/L$ is
compact;
\item[b)] for every $Y,Z\in \cA$ we have
$\Hom (\pi (h^Y),\pi (h^Z)[i])=0\quad \text{when}\quad i<0.$
\end{enumerate}
Then the triangulated category $D(\cA)/L$ has a unique enhancement.
\end{theorem}

\begin{theorem}
\label{imainperf} {\rm (=Theorem \ref{mainperf}).}
Let $\cA$ be a small category which we consider as a DG category
and $L\subset D(\cA)$ be a
localizing subcategory that is generated by compact objects $L^c=L\cap D(\cA)^c.$
Assume that for the quotient functor $\pi :D(\cA)\to
D(\cA)/L$ the following condition holds
\begin{enumerate}
\item[] for every $Y,Z\in \cA$ we have
$\Hom (\pi (h^Y),\pi (h^Z)[i])=0\quad \text{when}\quad i<0.$
\end{enumerate}
Then the triangulated
subcategory  of compact objects $(D(\cA)/L)^c$ has a
unique enhancement.
\end{theorem}

Our main tool in the proof of Theorems \ref{imain} and \ref{imainperf} is the Drinfeld
construction of a DG quotient of a DG category with its universal
property \cite{Dr}.

For convenience, in Section 2 we collect all our results together.
Sections 3-6 are devoted to proofs of two main Theorems \ref{imain} and \ref{imainperf}.
First, we give some preliminary lemmas and present  the main technical tool for the next sections, which is Proposition \ref{mainprop}.
Secondly, in Section \ref{cons} we construct a quasi-functor $\widetilde\rho$ (formula (\ref{rho})) which is a central object for all our considerations.
After that in Sections 5 and 6 we prove the main theorems; we show how to  apply Drinfeld Theorem \ref{Drin} to the quasi-functor
$\widetilde\rho$ and argue that the induced quasi-functor $\rho$ is actually a quasi-equivalence between different enhancements.
In the end of Sections \ref{Thmain} and \ref{Thmainperf} we give more advanced and precise versions of Theorems \ref{imain} and \ref{imainperf}
(see  Theorems \ref{fufabig} and \ref{corperf}).

In Section \ref{geom} as a consequence of Theorem \ref{imain}
we deduce the uniqueness of an enhancement for unbounded derived category of an abelian Grothendieck category $\cC$ under mild
additional conditions on it (see Theorem \ref{Grothcat}). More precisely, we prove that
the derived category $D(\cC)$ has a unique enhancement, if  the Grothendieck category $\cC$  has a set of small generators
which are compact objects in the derived category $D(\cC).$

This result can be applied to the category of quasi-coherent sheaves on a quasi-compact and quasi-separated scheme.
We say that a quasi-compact and quasi-separated scheme $X$ {\it has enough  locally free sheaves},
if  for any finitely presented sheaf $\F$ there is an epimorphism
$\E\twoheadrightarrow\F$ with a locally free
sheaf $\E$ of finite type. Theorem \ref{Grothcat} immediately implies that
the derived category of quasi-coherent sheaves $D(\Qcoh X)$  has a unique enhancement
if the quasi-compact and separated scheme  $X$  has enough locally free sheaves (Theorem \ref{qcqse}).
In particular, this statement can be applied for any quasi-projective scheme (Corollary \ref{qcqproj}).

In Section \ref{geom} we  show how to apply Theorem \ref{imainperf}
 to the subcategories of perfect complexes $\Perf(X).$
We proved that for  any quasi-projective scheme $X$ over $k$  the
triangulated category of perfect complexes $\Perf(X)$ has a  unique
enhancement (Theorem \ref{perfc}).

In Section \ref{cohbound} we introduce a notion of compactly approximated objects in a triangulated category
and we  prove that the triangulated category of compactly
approximated objects $(D(\cA)/L)^{ca}$ has a unique
enhancement (Theorem \ref{unca}). This result allows us to deduce the uniqueness of an enhancement for
the bounded derived category
of coherent sheaves $D^b(\coh X)$ on a quasi-projective scheme $X$ (Theorem \ref{bdcmain}).

In the case of projective varieties using results of \cite{Or, Or2} we can prove  stronger results.
If $X$ is a projective scheme over $k,$ then
the bounded derived category $D^b(\coh X)$ and the triangulated category of perfect complexes $\Perf(X)$ have
{\it strongly} unique enhancements (Theorem \ref{strun}). This result is a consequence of a general statement
about the bounded derived category of an exact category possessing  an ample sequence of objects
(Theorem \ref{streq}).

As corollaries of these results we obtain a representation of fully faithful functors from
categories of perfect complexes and bounded derived categories of coherent sheaves.
Any complex of quasi-coherent sheaves $\cE^{\cdot}$
on the product $X\times Y$ determines a functor
$$
\Phi_{\cE^{\cdot}}(-)=\bR p_{2*}(\cE^{\cdot}\stackrel{\bL}{\otimes}  p^{*}_{1}(-)): D(\Qcoh X)\to D(\Qcoh Y).
$$

We show that our theorems on uniqueness of enhancements imply
that if there is   a fully faithful functor $K: \Perf(X)\to D(\Qcoh(Y))$
for a quasi-projective schemes $X$ and a quasi-compact and separated scheme $Y$
then we can find an object $\cE^{\cdot}\in D(\Qcoh(X\times Y))$ such that
the restriction of the functor $\Phi_{\cE^{\cdot}}: D(\Qcoh X)\to D(\Qcoh Y)$ on $\Perf(X)$ is fully faithful too and
$\Phi_{\cE^{\cdot}}(P^{\cdot})\cong K(P^{\cdot})$ for every $P^{\cdot}\in \Perf(X).$
If, in addition, the functor $K$ sends $\Perf(X)$ to $\Perf(Y),$ then the functor $\Phi_{\cE^{\cdot}}$ is fully faithful,
sends  $\Perf(X)$ to $\Perf(Y),$ and $\cE^{\cdot}$ is isomorphic to an object of $D^b(\coh(X\times Y)).$
Finally, if  $X$ is a projective projective such that the maximal torsion subsheaf $T_0(\O_X)\subset\O_X$ of dimension 0 is trivial,
we show that the functor $\Phi_{\cE^{\cdot}}|_{\Perf(X)}$ is isomorphic
to $K$ (Corollary \ref{repr}). For a projective scheme $X$ with $T_0(\O_X)=0$ and the bounded derived category of coherent sheaves
$D^b(\coh X)$ we also proved that a fully faithful functor $K$ from $D^b(\coh X)$
to $D(\Qcoh Y)$ has the form $\Phi_{\cE^{\cdot}}$ if $K$ commutes with homotopy limits
(see Corollaries \ref{dbcoh} and \ref{lastcor}).

The main results of this paper were reported by the first author in
December 2008 at the conference on triangulated categories at
Swansea University. Recently some results on representability of fully faithful functors
between categories of perfect complexes were independently obtained in \cite{Ba}.

The second author is grateful to C.~Lazaroiu and A.~Kuznetsov for very useful discussions.
We thank the anonymous referee for careful reading of the text, making
several useful suggestions and for finding a few minor errors in the
original version.

\section{DG categories, quasi-functors, and quotients of DG
categories}

Our main reference for DG categories is \cite{Ke,Dr}.
Here we only recall a few points and introduce notation.
Let $k$ be an arbitrary field. We will write $\otimes$ for the tensor product over $k.$
All categories, DG categories, functors, DG
functors and etc. are assumed to be $k$\!-linear.

A {\it DG category} is a $k$\!-linear category $\cA$ whose morphism spaces $\Hom (X,Y)$
are provided
with a structure of a $\bbZ$\!-graded $k$\!-module and a differential
$d:\Hom(X,Y)\to \Hom (X,Y)$ of degree 1, so that for every $X, Y, Z\in
\Ob\cA$ the composition $\Hom (Y, Z)\otimes \Hom (X, Y)\to
\Hom (X, Z)$ is the morphism of DG $k$\!-modules. The identity morphism $1_X\in \Hom (X, X)$ is closed of
degree zero.

Using the supercommutativity isomorphism $S\otimes T\simeq T\otimes
S$ in the category of DG $k$\!-modules one defines for every DG
category $\cA$ the {\it opposite DG category} $\cA^{op}$ with $\Ob\cA
^{op}=\Ob\cA$ and $\Hom_{\cA^{op}}(X, Y)=\Hom_{\cA}(Y, X).$

For a DG category $\cA$ we denote by $\Ho(\cA)$ and
$\Hog(\cA)$ its homotopy and graded homotopy categories,
respectively.
The {\it homotopy category} $\Ho(\cA)$ has the same objects as the DG category $\cA$ and its
morphisms are defined by taking the $0$\!-th cohomology
$H^0(\Hom_{\cA} (X,Y))$
of the complex $\Hom_{\cA} (X,Y).$
The {\it graded homotopy category}
$\Hog(\cA)$ is defined  by replacing
each $\Hom$ complex in $\cA$ by the direct sum of its cohomology groups.

As usual a {\it DG functor}
$\F:\cA\to\cA'$ is given by a map $\F:\Ob(\cA)\to\Ob(\cA')$ and
by morphisms of DG $k$\!-modules
$$
\F(X,Y) : \Hom_{\cA}(X,Y) \to \Hom_{\cA}(\F X,\F Y)\quad X,Y\in\Ob(\cA)
$$
compatible with the composition and the units.

A DG functor $\F: \cA\to\cB$ is called a {\it quasi-equivalence} if
$\F(X,Y)$ is a quasi-isomorphism for all objects $X, Y$ of $\cA$
and the induced functor $H^0(\F): H^0(\cA)\to H^0(\cB)$ is an
equivalence. DG categories $\cA$ and $\cB$ are called {\it quasi-equivalent} if there exist DG
categories $\cC_1,\dots, \cC_n$ and a chain of quasi-equivalences
$\cA \leftarrow \cC_1 \rightarrow \cdots \leftarrow \cC_n \rightarrow \cB.$

Given a small DG category $\cA$ we define a {\it right DG $\cA$\!-module} as a DG functor
$M: \cA^{op}\to \Mod k,$ where $\Mod k$ is the DG category of DG $k$\!-modules. We denote by $\Mod\cA$ the DG
category of right DG $\cA$\!-modules.

Denote by $\Ac(\cA)$ the full
DG subcategory of $\Mod\cA$ consisting of all acyclic DG modules.
It is well-known that the
homotopy category of DG modules $\Ho(\Mod\cA)$ has a natural structure of a triangulated category
and the homotopy category of acyclic complexes $\Ho (\Ac(\cA))$ forms a full triangulated subcategory in it.
The {\it derived
category} $D(\cA)$ is the Verdier quotient of $\Ho(\Mod\cA)$
by the subcategory $\Ho (\Ac(\cA)).$

For each object $Y$ of $\cA$ we have the right module
represented by $Y$ $$h^Y(-):=\Hom_{\cA}(-,Y)$$
which is called a {\it representable} DG module. This
gives the Yoneda DG functor
$h^\bullet :\cA \to
\Mod\cA$ that is full and faithful.

The DG $\cA$\!-module is called {\it free} if it is isomorphic to a direct sum of  DG modules of the form
$h^Y[n],$ where $Y\in\cA,\; n\in\bbZ.$
A DG $\cA$\!-module
$P$ is called {\it semi-free} if it has a filtration
$0=\Phi_0\subset \Phi_1\subset ...=P$
such that each quotient  $\Phi_{i+1}/\Phi_i$ is free. The full
DG subcategory of semi-free DG modules is denoted by $\SF(\cA).$
We denote by $\SF_{fg}(\cA)\subset \SF(\cA)$ the full DG subcategory of finitely generated semi-free
DG modules, i.e. $\Phi_n=P$ for some $n$ and $\Phi_{i+1}/\Phi_i$ is a finite direct sum of DG modules of the form
$h^Y[n].$
We also denote by   $\prfdg(\cA)$ the DG category of perfect DG modules, i.e. the full DG
subcategory of $\SF(\cA)$ consisting of all DG modules which are homotopy
equivalent to a direct summand of a finitely generated semi-free DG module.

It is also natural to consider the category of h-projective DG modules.
We call a  DG $\cA$\!-module $P$ {\it h-projective (homotopically projective)} if
$$\Hom_{\Ho(\Mod\cA)}(P, N)=0$$
for every acyclic DG module $N$ (by duality, we can define {\it h-injective} DG modules). Let  $\cP(\cA)\subset \Mod\cA$ denote the full
subcategory of h-projective objects. It can be easily checked that a semi-free
DG-module is h-projective.
For every DG
$\cA$\!-module $M$ there is a quasi-isomorphism $\bp M\to M$ such that $\bp M$ is a semi-free
DG $\cA$\!-module.
Thus we obtain that the canonical DG functors $SF(\cA)\hookrightarrow\cP(\cA)\hookrightarrow\Mod\cA$ induce equivalences
$\Ho(SF(\cA))\stackrel{\sim}{\to} \Ho(\cP(\cA))\stackrel{\sim}{\to} D(\cA)$ of the triangulated categories (see \cite{Ke} 3.1, \cite{Hi} 2.2, \cite{Dr} 13.2).

Let $\F:\cA \to \cB$ be a DG functor between DG categories.
It induces the DG functors of restriction and extension of scalars
$$
\F_*:\Mod\cB\to \Mod\cA,\quad \F^*:\Mod\cA
\to \Mod\cB.
$$
The DG functor $\F^*$ is an extension of $\F$ on the category of DG modules, i.e the following diagram
commutes
$$
\begin{CD}
\cA @>h>> \Mod\cA\\
@V{\F}VV @VV{\F^*}V\\
\cB @>h>> \Mod\cB
\end{CD}
$$
where the horizontal arrows are the Yoneda embeddings.
The DG functors $(\F^*, \F_{*})$ are adjoint: for $M\in\Mod\cA$ and $N\in\Mod\cB$ there are functorial
isomorphisms
$$
\Hom_{\Mod\cB}(\F^*(M), N)\cong\Hom_{\Mod\cA}(M, \F_*(N)).
$$
The DG functor $\F^*$ preserves semi-free DG modules and $\F^* : \SF(\cA)\to \SF(\cB)$
is a quasi-equivalence if $\F$ is such. The DG functor $\F_{*}$ preserves acyclic
DG modules.

The DG functors $\F_*$ and $\F^*$ induce
the corresponding derived functors
$$F_*:D(\cB)\to D(\cA),\quad \bL F^*:D(\cA)\to D(\cB).$$

There is a third DG functor
$\F^!:\Mod\cA
\to \Mod\cB,
$
which is a right adjoint to $\cF_{*}$ and preserves h-injectives (see \cite{ELO1}). It is defined by the following formula
$$
\F^!(M)(X):=\Hom_{\Mod\cA}(\cF_{*}(h^X), M),\quad\text{where}\quad X\in\cB \quad\text{and}\quad M\in\Mod\cA.
$$
The DG functor $\F^!$ induces the derived functor $\bR F^!: D(\cA)\to D(\cB)$ which is right adjoint to $F_*.$

Let $\cA$ and $\cB$ be two small DG categories. Let $X$ be an
$\cA-\cB$\!-bimodule, i.e. a DG $\cA^{op}\otimes\cB$\!-module $X.$
For each DG $\cA$\!-module $M,$ we obtain a DG $\cB$\!-module
$M\otimes_{\cA} X.$
The DG functor $(-)\otimes_{\cA} X: \Mod\cA \to \Mod\cB$ admits a right adjoint
$\Hom_{\cB} (X, -).$
These functors do not respect
quasi-isomorphisms in general, but they form a Quillen adjunction
and the derived functors
$(-)\stackrel{\bL}{\otimes}_{\cA}X$ and
$\bR \Hom_{\cB} (X, -)$
form an adjoint pair of functors between derived categories $D(\cA)$ and
$D(\cB).$

Let
$\DGcat_k$ be the category of small DG $k$\!-linear categories.
It is known \cite{Ta} that it admits a structure of cofibrantly
generated model category whose weak equivalences are the
quasi-equivalences.
This shows in particular that the localization $\Hqe$ of
$\DGcat_k$ with respect to the quasi-equivalences has small
$\Hom$\!-sets. This also gives that a morphism from $\cA$ to $\cB$
in the localization can be represented as $\cA\leftarrow \cA_{cof}\to\cB,$ where $\cA\leftarrow \cA_{cof}$ is a cofibrant replacement.

The morphism sets in the localization are much better
described in term of quasi-functors. Consider two DG categories $\cA$ and $\cB$.
Denote by  $\rep(\cA,\cB)$ the full subcategory of the derived
category $D(\cA^{op}\otimes\cB)$ of $\cA-\cB$\!-bimodules formed by
all bimodules $X$ such that the tensor functor
$$
(-)\stackrel{\bL}{\otimes}_\cA X: D(\cA) \to D(\cB)
$$
takes every representable $\cA$\!-module to an object which is
isomorphic to a representable $\cB$\!-module. We call such a bimodule
a {\it quasi-functor} from $\cA$ to $\cB.$
In
other words a quasi-functor is represented by a DG functor $\cA\to \Mod\cB$ whose essential image consists of
quasi-representable DG $\cB$\!-modules (``quasi-representable'' means quasi-isomorphic to a representable DG module).
Since the category of quasi-representable DG
$\cB$\!-modules is equivalent to $\Ho(\cB)$ a quasi-functor
$\F \in \rep(\cA,\cB)$ defines a functor $\Ho(\F):\Ho(\cA)\to \Ho(\cB).$

It is known (see \cite{To})  that the morphisms from $\cA$ to $\cB$
in the localization of $\DGcat_k$ with respect to the
quasi-equivalences are in natural bijection with the isomorphism
classes of $\rep(\cA,\cB).$
Denote by $\dgcat$ a 2-category of  DG
categories with objects being small DG categories, 1-morphisms --
quasi-functors, 2-morphisms -- morphisms of quasi-functors, i.e. morphisms in
$D(\cA^{op}\otimes\cB).$

\begin{example}\label{first} {\rm Let $\cA$ and $\cB$ be two DG categories and $\varPhi:\cA \to
\Mod\cB$ be a DG functor. Then $\varPhi$ induces a
quasi-functor from $\cA$ to $\SF(\cB).$ Indeed, every $Y\in \cA$
defines a DG $\SF(\cB)$\!-module by the formula:
$$
P\mapsto \Hom_{\Mod\cB}(P,\F(Y)), \quad\text{for any}\quad P\in \SF(\cA).
$$
If $Q\in \SF(\cB)$ is quasi-isomorphic to $\F(Y),$ then this DG
$\SF(\cB)$\!-module is quasi-isomorphic to $h^Q\in
\Mod\SF(\cB).$ We will denote this quasi-functor  by
$\phi:\cA \to \SF(\cB).$}
\end{example}

For any DG category $\cA$ there exist a DG category $\cA^{\text{pre-tr}}$ that is called {\it pretriangulated hull}
and canonical fully faithful DG
functor $\cA\hookrightarrow\cA^{\text{pre-tr}}.$
The idea of the definition of $\cA^{\text{pre-tr}}$ is to formally add to $\cA$
all shifts, all cones, cones of morphisms between cones and etc. The objects of this DG category
are `one-sided twisted complexes' (see \cite{BK}).
There is a canonical fully faithful DG functor
(the Yoneda embedding) $\cA^{\text{pre-tr}}\to \Mod\cA,$ and under this embedding
$\cA^{\text{pre-tr}}$ is DG-equivalent to DG category of finitely generated semi-free DG modules, which we denote
by $\SF_{fg}(\cA).$

\begin{defi} We
say that $\cA$ is {\it pretriangulated}  if for every
objects $X\in\cA$ the object $X[n]\in \cA^{\text{pre-tr}}
$ is homotopy equivalent to an object of $\cA$
and for every closed morphism $f$ in $\cA$ of degree 0 the cone $\operatorname{Cone}(f)\in \cA^{\text{pre-tr}}$
is homotopy equivalent to an object of $\cA.$
In other words $\cA$ is pretriangulated if and only if the DG functor $\cA\to\cA^{\text{pre-tr}}$
is a quasi-equivalence.
\end{defi}
Thus if $\cA$ is pretriangulated the homotopy
category $\Ho(\cA)$ is triangulated. The DG category $\cA^{\text{pre-tr}}$ is always  pretriangulated,
so $\Ho(\cA^{\text{pre-tr}})$ is a triangulated category. We denote $\cA
^{\tr}:=\Ho(\cA^{\text{pre-tr}}).$

Notice that a quasi-functor $\F \in \rep(\cA, \cB)$
defines a functor $\cA^{\text{tr}}\to \cB^{\text{tr}}.$
Let us recall the main theorem in \cite{Dr}.

\begin{theo}{\rm (\cite{Dr})}\label{Drin} Let $\cA$ be a small   DG category
and $\cB \subset \cA$ be a full DG subcategory. For all pairs $(\cC
,\xi ),$ where $\cC$ is a DG category and $\xi \in \rep(\cA ,\cC)$ the
following properties are equivalent:
\begin{enumerate}
\item[(i)] the functor $\Ho(\xi):\Ho(\cA)\to \Ho(\cC)$ is essentially
surjective, and the functor $\cA^{\tr} \to \cC^{\tr}$
corresponding to $\xi$ induces an equivalence  $\cA^{\tr}/\cB
^{\tr}\to \cC^{\tr}$;
\item[(ii)]
for every DG category $\cK$ the functor $\rep(\cC ,\cK)\to \rep(\cA
,\cK)$ corresponding to $\xi$ is fully faithful and $\Phi \in \rep(\cA
,\cK)$ belongs to its essential image if and only if the image of
$\Phi$ in $\rep(\cB ,\cK)$ is zero.
\end{enumerate}
A pair $(\cC ,\xi )$ satisfying (i),(ii) exists and is unique in the
sense of {\bf DGcat}: given another such pair $(\cC^\prime ,\xi
^\prime )$ there exists a quasi-functor $\delta \in \rep(\cC ,\cC
^\prime)$ inducing an equivalence $\Ho(\delta):\Ho(\cC )\stackrel{\sim}{\lto} \Ho(\cC
^\prime)$ and such that the quasi-functors $\xi^\prime $ and
$\delta \cdot \xi$ are isomorphic.
\end{theo}

\begin{remark} \label{quot} {\rm  Although the above theorem is stated in the language
of quasi-functors it is important for us that one can choose a pair
$(\cC ,\xi)$ where $\xi :\cA \to \cC$ is a DG functor and not just a
quasi-functor. Indeed, recall Drinfeld's construction of the DG
quotient $\cA /\cB$: it is obtained from $\cA$ by adding for every
object $U\in \cB$ a morphism $\epsilon_U :U\to U$ of degree  $-1$
such that $d(\epsilon_U)=\id_U$ (no new objects and no new
relations between morphisms are added). Thus in particular $\cA$ is
a DG subcategory of $\cA /\cB$ and $\xi :\cA \to \cA /\cB$ is the
inclusion DG functor.}
\end{remark}

\begin{lemma}
Let $\cA$ be a small DG category and let $\cB\subset\cA$ be a full DG subcategory.
Assume that $\cA$ is pretriangulated. Then the Drinfeld DG quotient $\cA/\cB$ is also pretriangulated.
\end{lemma}
\begin{proof}
The canonical embedding $H^0(\cA/\cB)\to (\cA/\cB)^{\tr}$ is full and faithful. We need to prove
that it is essentially surjective. The quotient DG functor $\xi: \cA\to \cA/\cB$ induces the DG functor
$\xi^{\text{pre-tr}}:\cA^{\text{pre-tr}}\to (\cA/\cB)^{\text{pre-tr}}$ so that the natural diagram
$$
\xymatrix{
\cA \ar@{^{(}->}[r]\ar[d]_{\xi} & \cA^{\text{pre-tr}}\ar@<-2ex>[d]^{\xi^{\text{pre-tr}}}\\
\cA/\cB \ar@{^{(}->}[r]  & (\cA/\cB)^{\text{pre-tr}}
}
$$
commutes. It induces the
commutative diagram
$$
\xymatrix{
H^0(\cA) \ar[r]\ar[d]_{H^0(\xi)} & \cA^{\tr}\ar@<-0.5ex>[d]^{\xi^{\tr}}\\
H^0(\cA/\cB) \ar[r]  & (\cA/\cB)^{\tr},
}
$$
where the upper horizontal arrow is an equivalence by our assumption and the right
vertical arrow identifies $(\cA/\cB)^{\tr}$ with the Verdier quotient of $\cA^{\tr}$ by $\cB^{\tr}$
according to
Theorem \ref{Drin}. Hence the composition of these two functors is essentially
surjective which proves the lemma.
\end{proof}



\begin{defi} Let $\T$ be a category.
 An object $Y\in \T$ is called {\it compact} (in $\T$) if $\Hom
_{\T}(Y,-)$ commutes with arbitrary (existing in $\T$) direct sums, i.e. for each
family of objects $\{ X_i\}\subset \T$ such that
$\bigoplus_i X_i$ exists the canonical map
$$\bigoplus_i\Hom (Y, X_i){\longrightarrow}
\Hom (Y,\bigoplus_i X_i)$$ is an isomorphism.
\end{defi}

Denote by $\T^c\subset \T$ the full subcategory consisting of all compact
objects in $\T.$

\begin{defi}\label{van}
 Let $\cT$ be a triangulated category that admits arbitrary direct sums. A set $S \subset \cT^c$ is
called a set of compact generators if
any object $X\in \cT$ such that $\Hom(Y, X[n])=0$ for all $Y\in S$ and  all $n\in \bbZ$ is a zero object.
\end{defi}
\begin{remark}
{\rm  Since $\cT$ admits arbitrary direct sums it can be proved that the property that $S \subset \T^c$ is
 a set of compact generators is  equivalent to the following property:
the category $\cT$ coincides with the smallest full
triangulated subcategory containing $S$ and  closed under
direct sums.}
\end{remark}

\begin{example}\label{EX}{\rm Let $\cA$ be a small DG category. The set $\{h^Y\}_{Y\in \cA}$ is a set of
 compact generators of $D(\cA)$ and the subcategory of compact objects $D(\cA)^c$ coincides with the subcategory of perfect
 DG modules $\Perf(\cA).$}
\end{example}

There is another notion of a set of generators. It is called a set of classical generators.
\begin{defi}
Let $\T$ be a triangulated category. We say that a set $S\subset \T$
is a set of classical generators for $\T$ if the category $\T$ coincides with the smallest triangulated subcategory of
$\T$ which contains $S$ and is closed under direct summands.
\end{defi}
\begin{remark}\label{twodef}{\rm
These two definitions of a set of generators are closely related to each other.
Assume that a triangulated category $\T,$ which admits arbitrary direct sums, is compactly generated by the set
of compact objects $\T^c.$ In this situation a set $S\subset \T^c$ is a set of compact generators of $\T$ if and only if
this set $S$ is a set of classical generators of the subcategory of compact objects $\T^c.$ This is proved in \cite{Ne}.
Thus the set $\{h^Y\}_{Y\in \cA}$ from Example \ref{EX} is a set of
 classical generators of the subcategory of compact objects $D(\cA)^c\cong \Perf(\cA).$}
\end{remark}

We will recall the definition of homotopy colimits in triangulated categories.
Namely, let $\T$ be a triangulated category and
$$X_0\stackrel{i_0}{\lto} X_1\stackrel{i_1}{\lto}\cdots{\lto} X_n\stackrel{i_n}{\lto} X_{n+1}{\lto}\cdots$$
be a sequence of morphisms in $\T.$

\begin{defi}\label{hocolim} Assume that the direct sum $\bigoplus_n X_n$ exists in $\T.$
The homotopy limit of this sequence ${\hocolim}X_n\in \T$ is by definition given, up to non-canonical isomorphism,
by a triangle
$$
\begin{CD}
\bigoplus_n X_n @>{(1, -i_n)}>> \bigoplus_n X_n @>>> {\hocolim}X_n @>>> \left(\bigoplus_n X_n\right)[1]
\end{CD}
$$
\end{defi}

If $\cA$ is a DG category and maps $i_n: X_n\to X_{n+1}$ are closed morphisms of degree zero
in $\Mod\cA,$ then one has the usual ${\colim} X_n$ in the
category $Z^0(\Mod\cA)$ which is isomorphic to
${\hocolim} X_n$ in $D(\cA).$ The category $Z^0(\Mod\cA)$ has the same object as $\Mod\cA$ and morphisms  are
closed morphism of degree $0$ (it is an abelian category).

\begin{defi} Let $\T$ be a triangulated category with arbitrary
direct sums. A strictly full triangulated
subcategory $\cS\subset \T$ is called  localizing if it is
closed under arbitrary direct sums.
\end{defi}

\begin{remark}\label{locfac}{\rm  If a subcategory $\cS$ is localizing it is known
 that the quotient triangulated category $\T/\cS$
has arbitrary direct sums and the quotient functor $\T\to
\T/\cS$ preserves direct sums (see \cite{NeBook} Cor. 3.2.11).
For example, the triangulated category $\Ho(\Mod\cA)$ has
arbitrary direct sums and $\Ho (\Ac(\cA))$ is a localizing
subcategory in $\Ho(\Mod\cA).$ Hence, the derived category $D(\cA)$ also has
arbitrary direct sums.}
\end{remark}

The following propositions (and their proofs) are essentially equal to
Lemma 4.2 in \cite{Ke}.

\begin{prop}\label{Keller2} Let $\F:\cB\hookrightarrow \cC$ be a full embedding of DG categories.
and let $\F^{*}:\SF(\cB)\to\SF(\cC)$ be the extension DG functor. Then the induced
derived functor $\bL F^*=H^0(\F^*): D(\cB)\to D(\cC)$ is fully faithful.
If, in addition, the category $H^0(\cC)$ is classically generated by $\Ob\cB$ then
$\bL F^*$ is an equivalence.
\end{prop}
\begin{proof}
The functor $\bL F^*$ commutes with direct sums and on $\cB\subset\D(\cB)$ coincides with the
functor $H^0(\F)$  that is fully faithful.

Let $Y\in \cB .$ The objects $X\in D(\cB)$ for which the map
$$
\Hom_{D(\cB)}(Y,X[n])\lto \Hom
_{D(\cC)}(\bL F^*(Y),\bL F^*(X)[n])
$$
is bijective for all $n$ form a
triangulated subcategory of $D(\cB)$ which contains the generating
set $\Ob\cB$ and is closed under direct sums (since $Y$ and
$\bL F^*(Y)$ are compact and $\bL F^*$ commutes with direct
sums). So this subcategory coincides with the whole $D(\cB).$
The same
argument shows that for a fixed $X\in D(\cB)$ the map
$$
\Hom_{D(\cB)}(X^\prime,X)\lto \Hom
_{D(\cC)}(\bL F^*(X^\prime),\bL F^*(X))
$$
is bijective for all
$X^\prime \in D(\cB).$
Hence $\bL F^*$ is fully faithful.

If now  the category $\Ho(\cC)$ is classically generated by $\Ob\cB$ then it is contained in the essential image
of $\bL F^*.$ On the other hand, the set $\Ob \cC$ compactly generates $D(\cC).$
Hence the whole category
$D(\cC)$ is in the essential image of $\bL F^*,$ because  $D(\cB)$ contains arbitrary direct sums and $\bL F^*$
commutes with direct sums.
\end{proof}

Let $\cC$ be a pretriangulated DG category and $\cB
\subset \cC$ be a full DG subcategory. Consider the canonical DG functor
\begin{equation}\label{phi}
\varPhi :\cC \to \Mod\cB, \quad\text{where}\quad \varPhi (X)(B)=\Hom_{\cC}(B,X) \quad\text{for}\quad B\in\cB, \; X\in\cC
\end{equation}
If we denote by $\cJ$ the full embedding of $\cB$ to $\cC$ then the DG functor $\varPhi$ is the composition
of Yoneda DG functor $h^{\bullet}:\cC\to \Mod\cC$ and the restriction functor $\cJ_{*}:\Mod\cC\to\Mod\cB.$

The DG functor $\varPhi$ induces a quasi-functor $\phi:\cC\to\SF(\cB)$ (see Remark \ref{first}) and the functor
$\Ho(\varPhi):\Ho(\cC)\to \Ho(\Mod\cB).$ The homotopy functor $H^0(\phi):\Ho(\cC)\to D(\cB)$
is the composition of $\Ho(\varPhi)$ with the localization functor
$Q: \Ho(\Mod\cB)\to D(\cB).$

\begin{prop}\label{Keller1} Let $\cB\subset\cC$ and $\phi:\cC\to\SF(\cB)$ be as above. Assume that
the triangulated category $\Ho(\cC)$ is idempotent complete and is classically generated by the set $\Ob\cB.$
Then the functor $H^0(\phi)$ induces an equivalence between $\Ho(\cC)$ and  the subcategory of compact objects
$D(\cB)^c\cong\Perf(\cB).$
\end{prop}
\begin{proof}
Indeed, if $Y,Z\in \cB\subset\cC$ then
$$
\Hom_{\Ho(\Mod\cB)}(\Ho(\varPhi)(Y),\Ho(\varPhi )(Z)[n])=\Hom
_{\Ho(\cC)}(Y,Z[n])$$ by Yoneda. Now, since $\varPhi(Y)=h^Y\in \Ho(
\Mod\cB)$ is h-projective we have
$$
\Hom_{\Ho(\Mod\cB)}(\Ho(\varPhi)(Y),\Ho(\varPhi )(Z)[n])\cong\Hom
_{D(\cB)}(H^0(\phi)(Y), H^0(\phi)(Z)[n]).
$$
Since the category $\Ho(\cC)$ is classically generated by $\Ob\cB$ we get that the functor
$H^0(\phi)$ is fully faithful. Since $\Ho(\cC)$ is idempotent complete and $D(\cB)^c$ is classically generated by
$\Ob\cB$ we obtain that the subcategory $D(\cB)^c$ in the essential image of the functor $H^0(\phi).$
\end{proof}

\begin{prop}\label{Keller} Let $\cB\subset\cC,$ and $\phi:\cC\to\SF(\cB)$ be as above.
 Assume that $\Ho(\cC)$
contains arbitrary direct sums and $\Ob\cB$ forms a set of compact
generators of $\Ho(\cC).$ Then the functor $H^0(\phi): \Ho(\cC)\stackrel{\sim}{\to} D(\cB)$ is
an equivalence of
triangulated categories.
\end{prop}

\begin{proof}
First notice that the functor $H^0(\phi)$ commutes with direct sums.
Indeed,  since $Q$ commutes with direct
sums (see Remark \ref{locfac}) it suffices to prove that $\Ho(\varPhi)$ does.
Fix $Y\in \cB$ and a set $\{X_i\}\subset \cC.$ Since $Y$ is compact
in $\Ho(\cC)$ we have the following isomorphisms which are
functorial in $Y$:
$$
\Ho(\varPhi)\left(\bigoplus X_i\right)(Y)=H^0\Hom_{\cC}(Y,\bigoplus X_i)=\bigoplus
H^0\Hom_{\cC}(Y,X_i)=\bigoplus \Ho(\varPhi)(X_i)(Y).
$$
Hence
$\Ho(\varPhi)(\bigoplus X_i)=\bigoplus \Ho(\varPhi)(X_i).$

Note that the exact functor $H^0(\phi)$ maps the set $\Ob\cB$
of compact generators of $\Ho(\cC)$ to the set $\{h^Y\}_{Y\in \cB}$
of compact generators of $D(\cB).$ By Remark \ref{twodef} and Proposition \ref{Keller1} it induces
an equivalence between subcategories of compact objects $\Ho(\cC)^{c}$ and
$D(\cB)^c.$

The same argument as in the proof of Proposition \ref{Keller2}
gives us that the functor $H^0(\phi)$ is fully faithful.
Since, in addition, $\Ho(\cC)$ contains arbitrary direct sums and $H^0(\phi)$
commutes with direct sums it follows that $H^0(\phi)$ is
essentially surjective.
\end{proof}

Let $\cT$ be a triangulated category with small Hom-sets, i.e Hom between any two objects should be a set.
Assume that $\cT$ admits arbitrary direct sums and let
$\cS\subset\cT$ be a localizing triangulated subcategory.
We can consider the Verdier quotient $\cT/\cS$ with a natural localization map $\pi : \cT\to \cT/\cS.$
Notice however that Hom-sets in $\cT/\cS$ need not be small.
It is known (Remark \ref{locfac}) that the category $\cT/\cS$ also has arbitrary direct sums and, moreover,
the functor $\pi$ preserves direct sums.

Assume that the Verdier quotient $\cT/\cS$ is a
category with small Hom-sets. If the triangulated category $\cT$ has a set of compact generators
then the Brown representability theorem holds for
$\cT$ and the quotient functor $\pi : \cT\to \cT/\cS$ has a right adjoint $\mu : \cT/\cS\to \cT$ (see \cite{NeBook} Ex.8.4.5).
This adjoint is
called the Bousfield localization functor.

\begin{defi} Let $\cT$ be a triangulated category with small Hom-sets.
Let $\cS$ be a thick subcategory. We say that a Bousfield localisation
functor exists for the pair $\cS\subset\cT$ when there is a right adjoint to the natural
functor
$\pi : \cT \to \cT/\cS.$
We will call the adjoint the Bousfield localisation functor, and denote it
$\mu : \cT/\cS \to \cT.$
\end{defi}

Let us summarize the facts  about Bousfield localization we will need in the following proposition.

\begin{prop}\label{adjnew}
Let $\cT$ be a compactly generated triangulated category with small Hom-sets that admits arbitrary direct sums. Let
$\cS\subset\cT$ be a localizing triangulated subcategory and
$\pi : \cT \to \cT/\cS$ be  the quotient functor. Assume that the quotient
$\cT/\cS$ is a category with small Hom-sets. Then

\begin{itemize}
\item[a)] there is a right adjoint functor $\mu :\cT/\cS\to \cT;$
\item[b)] the functor $\mu$ is full and faithful;
\item[c)] if for every compact object $Y\in \cT$ the object $\pi (Y)$ is compact in
$\cT/\cS,$ then the functor $\mu$ preserves direct sums;
\item[d)] if for every compact object $Y\in \cT$ the object $\pi (Y)$ is compact in
$\cT/\cS$ and $\cT$ is compactly generated by a set $R\subset\cT^c,$ then $\cT/\cS$ is
also compactly generated by $\pi(R).$
\end{itemize}
\end{prop}
\begin{proof}
a) It is consequence of the Brown representability theorem (\cite{NeBook} Th. 8.4.4).

b) This is also a general statement which says that if the right adjoint functor to a localization exists
then it is fully faithful (\cite{NeBook} Lemma 9.1.7).

c)
If the object $\pi(Y)$ is compact, then
\begin{multline*}
\Hom(Y,\mu(\bigoplus X_i))\cong \Hom
_{}(\pi(Y),\bigoplus X_i)
\cong \bigoplus \Hom_{}(\pi(Y),X_i)
\cong  \bigoplus \Hom_{}(Y,\mu(X_i))\cong\\\cong \Hom_{}(Y,\bigoplus\mu(X_i))
\end{multline*}
for  every $Y\in
\cT^c.$ Since $\cT$ is compactly generated  we obtain that
$\bigoplus \mu(X_i)\to \mu(\bigoplus X_i)$ is an isomorphism.

d) If $\Hom(\pi(Y), X[n])=0$ for all $Y\in R$ and all $n\in \bbZ$ then $\Hom(Y, \mu(X)[n])=0.$
Since $R\subset\cT^c$ is a set of compact generators
then $\mu(X)=0.$ And hence $X=0,$ because $\mu$ is fully faithful.
Therefore $\pi(R)$ is a set of compact generators for $\cT/\cS.$
\end{proof}

\begin{remark}\label{remar} {\rm
Let $\cA$ be a DG category. The objects $\{h^Y\}_{Y\in\cA}$ form
a set of compact generators of $D(\cA).$
Let $L\subset D(\cA)$ be a localizing subcategory, and
$\pi :D(\cA)\to D(\cA)/L$ be the quotient functor. Assume that the Verdier quotient
$D(\cA)/L$ is a category with small Hom-sets. Then we can apply Proposition \ref{adjnew}
and get a Bousfield localisation functor $\mu: D(\cA)/L\to D(\cA),$ which is fully faithful.
 If for every $Y\in \cA$ the object $\pi (h^Y)$ is compact in
$D(\cA)/L,$ then the objects $\{\pi (h^Y)\}_{Y\in\cA}$ compactly generate the category
$D(\cA)/L$ and the functor $\mu$ preserves direct sums.
}
\end{remark}

If a localizing subcategory $\cS\subset \cT$ is compactly generated by the set of objects
$\cS\cap\cT^c$ then the quotient category $\cT/\cS$ has small Hom-sets (\cite{NeBook}, Cor.4.4.3) and
we can say even more (\cite{Ne} Th.2.1, \cite{NeBook} Th.4.4.9).

\begin{theo}{\rm \cite{Ne}}\label{Neeman} Let $\cT$ be a compactly generated triangulated
category admitting arbitrary direct sums and let $\cS$ be a
localizing triangulated subcategory which is generated
by a subset $S\subset\cT^c$ of compact objects. Then

\begin{enumerate}
\item $\cT/\cS$ has small Hom-sets and is compactly generated;
\item $\T^c$ maps to $(\cT/\cS)^c$ under the quotient functor;
\item the induced functor $\T^c/\cS^c\to (\cT/\cS)^c$ is fully faithful;
\item $(\cT/\cS)^c$ is the idempotent completion of $\cT^c/\cS^c.$
\end{enumerate}
\end{theo}

\begin{remark}\label{Usmall}
{\rm
One of our main tools will be Theorem \ref{Drin} above which assumes that we work with small categories.
More precisely, we will need to apply the DG localization to the DG category
$\Mod\cA,$ where $\cA$ is a small category. Thus we will choose universes $\mathbb{U}\in \mathbb{V}\in\cdots,$
assume that the category $\cA$ is a $\mathbb{U}$\!-small $\mathbb{U}$\!-category and will consider
the DG category $\ModU\cA$ of $\mathbb{U}$\!-small DG $\cA$\-modules. Thus $\ModU\cA$ is a DG $\mathbb{U}$\!-category
and we explain in Appendix that it is DG equivalent to a $\mathbb{V}$\!-small DG category $\ModUs\cA$ of strict $\mathbb{U}$\!-small
DG $\cA$\!-modules. So we may apply the DG localization to $\ModUs\cA$ instead of $\ModU\cA.$

However we decided not to mention explicitely the $\mathbb{U}$\!-smallness issue in the main body of the text. Therefore we add
an appendix which contains all the relevant statements and in particular it has the $\mathbb{U}$\!-small version of Brown representability theorem,
etc.
}
\end{remark}

\section{Enhancements of triangulated categories and formulation of main results}

\begin{defi} Let $\cT$ be a triangulated category. An {\it
enhancement} of $\cT$ is a pair $(\cB , \varepsilon),$ where $\cB$ is a
pretriangulated DG category and $\varepsilon:\Ho(\cB)\stackrel{\sim}{\to} \cT$ is an equivalence
of triangulated categories.
\end{defi}
\begin{defi} The category $\cT$ has a  unique enhancement if it
has one and for any two enhancements $(\cB, \varepsilon)$ and $(\cB', \varepsilon')$
of $\cT$ there exists a quasi-functor $\phi: \cB \to \cB'$
which induces an equivalence $\Ho(\phi):\Ho(\cB)\stackrel{\sim}{\lto} \Ho(\cB').$
In this case the enhancements $(\cB, \varepsilon)$ and $(\cB', \varepsilon')$ are called equivalent.
\end{defi}
\begin{defi}
Enhancements $(\cB, \varepsilon)$ and $(\cB', \varepsilon')$ of $\cT$ are called strong equivalent if there exists a
quasi-functor $\phi:\cB \to \cB'$ such that the functors
$\varepsilon' \cdot\Ho(\phi)$ and $\varepsilon$ are isomorphic.
\end{defi}

Here are some examples of existence of canonical enhancements.

\begin{example}{\rm  Let $\fA$ be an abelian category. Denote by $\cC_{dg}(\fA)$ the
DG category of complexes over $\fA,$ and by $D(\fA)$ the derived category
of $\fA.$ Let $L\subset D(\fA)$ be a localizing subcategory. Consider the full pretriangulated DG subcategory
$\cL\subset \cC_{dg}(\fA)$
such that $\Ob\cL=\Ob L.$ Let $\cC_{dg}(\fA)/\cL$ be the corresponding
DG quotient. Then by Theorem \ref{Drin} above there is a natural equivalence
$F:\Ho(\cC_{dg}(\fA)/\cL)\stackrel{\sim}{\lto} D(\fA)/L$ of triangulated categories.
Hence $(\cC_{dg}(\fA)/\cL, F)$ is a canonical enhancement of
$D(\fA)/L.$}
\end{example}

We also can consider a slight variation.

\begin{example}\label{exdg}{\rm  Let $\cA$ be a DG category.
If $L\subset D(\cA)$ is a localizing subcategory, the quotient
category $D(\cA)/L$ has a canonical enhancement. Namely, let $\cL
\subset \Mod\cA$ be the full DG subcategory which has the
same objects as $L.$ Let $\Mod\cA/\cL$ be the DG quotient
category. Then there is a natural equivalence
$F:\Ho(\Mod\cA/\cL)\stackrel{\sim}{\lto} D(\cA)/L$ of triangulated categories ,
i.e. $(\Mod\cA/\cL ,F)$ is an enhancement of $D(\cA)/L.$
Consider the DG subcategories $\SF(\cA)\subset\cP(\cA)\subset \Mod\cA$ of semi-free and
h-projective DG modules, respectively. Then the natural DG functors
$$
\SF(\cA)/\cL\cap \SF(\cA)\lto \cP(\cA)/\cL\cap \cP(\cA)\lto \Mod\cA/\cL
$$
are
quasi-equivalences, so that $\SF(\cA)/\cL\cap \SF(\cA)$ and $\cP(\cA)/\cL\cap \cP(\cA)$ are other
enhancements of $D(\cA)/L$ that are equivalent to
$\Mod\cA/\cL.$}
\end{example}

For a triangulated category it is natural to ask if it has an
enhancement. Next is the question of
uniqueness.

First, let us prove the following not difficult and very natural proposition.

\begin{prop}\label{forbeg}
Let $\cA$ be a small category which we consider as a DG category
 and $D(\cA)$ and $\Perf(A)=D(\cA)^c$ be the derived category of DG $\cA$\!-modules and
its subcategory of perfect DG modules, respectively. Then each of them has a unique enhancement.
\end{prop}
\begin{proof}
The canonical DG functors $SF(\cA)\hookrightarrow\cP(\cA)\hookrightarrow\Mod\cA/\Ac(\cA)$ are quasi-isomorphisms
and gives equivalent enhancements of the derived category $D(\cA).$

Let $\cD$ be a pretriangulated DG category
and $\epsilon :D(\cA)\stackrel{\sim}{\to} \Ho(\cD)$ be an equivalence of triangulated
categories.  Denote by $\cB
\subset \cD$ the full DG subcategory with the set of objects
$\{\epsilon (h^Y)\}_{Y\in \cA}.$

Denote by $\tau_{\leq 0}\cB$ the DG subcategory with
the same objects as $\cB$ and morphisms
$$\Hom_{\tau_{\leq 0}\cB}(M,N)=\tau_{\leq 0}\Hom_{\cB}(M,N),$$
where $\tau_{\leq 0}$ is the usual truncation of complexes:
\begin{equation}\label{trunc}
\tau_{\leq 0}\{\cdots\lto C^{-1}\lto C^0\stackrel{d^0}{\lto} C^1 \lto \cdots\}=\{\cdots\lto
C^{-1}\lto \Ker (d^0)\lto 0\}.
\end{equation}
We have the obvious diagram of DG
categories and DG functors
$$
\cA=\Ho(\cB) \stackrel{p}{\longleftarrow} \tau_{\leq 0}\cB
\stackrel{i}{\lto} \cB,
$$
which are quasi-equivalences, because $\epsilon :D(\cA)\stackrel{\sim}{\to} \Ho(\cD)$ is an equivalence of triangulated
categories. They induce DG functors
$$
\SF(\cA) \stackrel{p^*}{\longleftarrow} \SF(\tau_{\leq 0}\cB)
\stackrel{i^*}{\lto} \SF(\cB),
$$
that are quasi-equivalences as well.
We also have a DG functor $\varPhi :\cD \to \Mod\cB$ defined by the rule
$\varPhi(X)(B):=\Hom_{\cB}(B,X),$ where $B\in\cB$ and $X\in \cD.$ The DG functor induces a quasi-functor $\phi:\cD\to \SF(\cB)$
(see Remark \ref{first}).
Since, by construction, objects of $\cB$
form a set of compact generators of $\Ho(\cD)\cong D(\cA),$ Proposition \ref{Keller} implies that
$\phi$ is a quasi-equivalence too. Thus, we get the following chain of quasi-equivalences
$$
\SF(\cA) \stackrel{p^*}{\longleftarrow} \SF(\tau_{\leq 0}\cB)
\stackrel{i^*}{\lto} \SF(\cB)\stackrel{\phi}{\longleftarrow}  \cD,
$$
and the enhancement $\cD$ of $D(\cA)$ is equivalent to the standard enhancement $\SF(\cA).$
The case of the subcategory of perfect DG modules can be considered similarly.
\end{proof}

The following theorems are our main results.

\begin{theo}\label{main} Let $\cA$ be a small category which we consider as a DG category
 and $L\subset D(\cA)$ be a
localizing subcategory with the quotient functor $\pi :D(\cA)\to
D(\cA)/L$ that has a right adjoint (Bousfield localization) $\mu.$ Assume that the following conditions hold
\begin{enumerate}
\item[a)]  for every $Y\in \cA$ the object  $\pi (h^Y)\in D(\cA)/L$ is
compact;
\item[b)] for every $Y,Z\in \cA$ we have
$\Hom (\pi (h^Y),\pi (h^Z)[i])=0\quad \text{when}\quad i<0.$
\end{enumerate}

Then the triangulated category $D(\cA)/L$ has a unique
enhancement.
\end{theo}
As we mentioned in Remark \ref{Usmall} in the above theorem we actually work with a small version
of $D(\cA),$ i.e. the category $D(\cA)$ stands for $D_{\mathbb U}(\cA)$ for a chosen universe $\mathbb U$ (see Appendix).

\begin{theo}\label{mainperf}
Let $\cA$ be a small category which we consider as a DG category
and $L\subset D(\cA)$ be a
localizing subcategory that is generated by compact objects $L^c:=L\cap D(\cA)^c.$
Assume that for the quotient functor $\pi :D(\cA)\to
D(\cA)/L$ the following condition holds
\begin{enumerate}
\item[] for every $Y,Z\in \cA$ we have
$\Hom (\pi (h^Y),\pi (h^Z)[i])=0\quad \text{when}\quad i<0.$
\end{enumerate}
Then the triangulated
subcategory  of compact objects $(D(\cA)/L)^c$ has a
unique enhancement.
\end{theo}
Note that under condition that subcategory $L$ is generated by compact objects $L^c:=L\cap D(\cA)^c$
 the quotient
$D(\cA)/L$ has small Hom-sets by Theorem \ref{Neeman} (1) and, hence, the Bousfield localization
functor $\mu: D(\cA)/L\to D(\cA)$ exists by Proposition \ref{adjnew} a).
In addition, in this case the property a) of Theorem \ref{main} is automatically holds by Theorem \ref{Neeman} (2),
i.e. for all $Y\in \cA$ the objects  $\pi (h^Y)\in D(\cA)/L$ are
compact.
We also note that Theorems \ref{main} and \ref{mainperf} have more advanced and precise versions (see  Theorems \ref{fufabig} and \ref{corperf}).

It will be shown in Section \ref{geom} that Theorem \ref{main} implies the following corollary.

\begin{theo} Let $\cC$ be a Grothendieck category. Assume that it has a set of small generators
which are compact objects in the derived category $D(\cC).$
Then the derived category  $D(\cC)$ has a unique enhancement.
\end{theo}

This result can be applied to the category of quasi-coherent sheaves on a quasi-compact and quasi-separated scheme.
Let $X$ be a a quasi-compact and quasi-separated scheme over $k.$ Denote by  $\Qcoh X$ the
abelian categories of  quasi-coherent sheaves on $X.$
We say that the scheme $X$ {\it has enough  locally free sheaves},
if  for any finitely presented sheaf $\cF$ there is an epimorphism
$\E\twoheadrightarrow\cF$ with a locally free
sheaf $\E$ of finite type.

\begin{theo}
Let $X$ be a quasi-compact and separated scheme that has enough locally free sheaves. Then
the derived category of quasi-coherent sheaves $D(\Qcoh X)$  has a unique enhancement.
\end{theo}

For particular case of quasi-projective schemes we obtain

\begin{cor} Let $X$ be a quasi-projective scheme over $k.$ Then the
derived category of quasi-coherent sheaves $D(\Qcoh X)$   has a unique enhancement.
\end{cor}

Denote by $\Perf(X)\subset D(\Qcoh X)$ the full
subcategory of perfect complexes which coincides with the subcategory of compact objects there (see \cite{Ne3, BvB}).
In Section \ref{geom} we also show that the above Theorem \ref{mainperf} implies the
following statement.

\begin{theo} Let $X$ be a quasi-projective scheme over $k.$ Then the
triangulated category of perfect complexes $\Perf(X)=D(\Qcoh X)^c$ has a unique enhancement.
\end{theo}

In Section \ref{cohbound} we introduce a notion of compactly approximated objects in a triangulated category
and we  prove that the triangulated category of compactly
approximated objects $(D(\cA)/L)^{ca}$ has a unique
enhancement (Theorem \ref{unca}). This result allows us to deduce the uniqueness of enhancement for
the bounded derived category of coherent sheaves on a quasi-projective
scheme.

\begin{theo}
The bounded derived  category of coherent sheaves $D^b(\coh X)$ on a quasi-projective scheme $X$ has a unique enhancement.
\end{theo}

In the case of projective varieties using results of \cite{Or, Or2} we can prove a stronger result.

\begin{theo} Let $X$ be a projective scheme over $k$ such that the maximal torsion
subsheaf $T_0(\O_X)\subset \O_X$ of dimension $0$ is trivial. Then
the triangulated categories $D^b(\coh X)$ and $\Perf(X)$ have  strongly
unique enhancements.
\end{theo}

These results on uniqueness of enhancements also allow  to obtain some corollaries on
representation of fully faithful functors between derived categories of (quasi)-coherent sheaves.

\begin{cor}
Let $X$ and $Y$ be quasi-compact separated schemes over a field $k.$ Assume that $X$ has enough locally free sheaves.
Let
$F: D(\Qcoh X)\to D(\Qcoh Y)$ be a fully faithful functor that commutes with direct sums.
Then there is an object $\cE^{\cdot}\in D(\Qcoh(X\times Y))$ such that
the functor $\Phi_{\cE^{\cdot}}$ is
fully faithful and $\Phi_{\cE^{\cdot}}(C^{\cdot})\cong F(C^{\cdot})$ for any $C^{\cdot}\in D(\Qcoh X).$
\end{cor}
\begin{cor} Let $X$ be a quasi-projective scheme and $Y$ be a quasi-compact and separated scheme.
Let $K: \Perf(X)\to D(\Qcoh Y)$ be a fully faithful
functor. Then there is an object $\cE^{\cdot}\in D(\Qcoh(X\times Y))$ such that
\begin{enumerate}
\item the functor $\Phi_{\cE^{\cdot}}|_{\Perf(X)}:\Perf(X)\to D(\Qcoh Y)$ is fully faithful and $\Phi_{\cE^{\cdot}}(P^{\cdot})\cong K(P^{\cdot})$
for any $P^{\cdot}\in\Perf(X);$
\item if $X$ is projective with $T_0(\O_X)=0,$ then $\Phi_{\cE^{\cdot}}|_{\Perf(X)}\cong K;$
\item if $K$ sends $\Perf(X)$ to $\Perf(Y),$ then
the functor $\Phi_{\cE^{\cdot}}: D(\Qcoh X)\to D(\Qcoh Y)$ is fully faithful and
also sends $\Perf(X)$ to $\Perf(Y);$
\item if $Y$ is a noetherian and $K$ sends $\Perf(X)$ to $D^b(\Qcoh Y)_{\coh},$ then the object $\cE^{\cdot}$ is isomorphic to an object of
$D^b(\coh(X\times Y)).$
\end{enumerate}
\end{cor}
\begin{cor}
Let $X$ be a projective scheme with $T_0(\O_X)=0$ and $Y$ be a quasi-compact and separated scheme. Let
$K: D^b(\coh X)\to D(\Qcoh Y)$ be a fully faithful
functor that commutes with homotopy limits.
 Then there is an object $\cE^{\cdot}\in D(\Qcoh(X\times Y))$ such that
$\Phi_{\cE^{\cdot}}|_{D^b(\coh X)}\cong K.$
\end{cor}

\section{Preliminary Lemmas and Propositions}

The following four Sections 3-6 are devoted to proofs of two main Theorems \ref{main} and \ref{mainperf}.
In this section we give some preliminary lemmas and present  the main technical tool for the next sections, which is Proposition \ref{mainprop}.
In Section \ref{cons} we construct a quasi-functor $\widetilde\rho$ (formula (\ref{rho})) which is a central object for all our considerations.
In Sections 5 and 6 we give proofs of our main theorems. We show that we can apply Drinfeld Theorem \ref{Drin} to the quasi-functor
$\widetilde\rho$ and argue that the induced quasi-functor $\rho$ is actually a quasi-equivalence between different enhancements.

Let $\cA$ be a small category which we consider as a DG category.
As above we denote by $\Mod\cA$ the DG category of DG $\cA$\!-modules.
Let us consider the full DG subcategory $\SF(\cA)\subset \Mod\cA$ of semi-free
DG modules.
Since $\cA$ is an ordinary category, any semi-free DG module $P\in\SF(\cA)$ is actually
a complex
$$
P=\{\cdots\lto P^{n-1}\lto P^n \lto \cdots\}
$$
where any $P^n=\bigoplus h^Y$ is a free
$\cA$\!-module $P.$ Of course, not any such complex is a semi-free DG module.
By definition, it is semi-free if it also has a filtration
$0=\Phi_0\subset \Phi_1\subset ...=P$
such that each quotient  $\Phi_{i+1}/\Phi_i$ is a free DG $\cA$\!-module.
On the other hand, any bounded above such complex is semi-free.
We denote by $\SF^{-}(\cA)$ the DG category of bounded above complexes of free $\cA$\!-modules.
and  denote by $\SF_{fg}(\cA)$ the DG category of finitely generated semi-free DG modules,
which is actually DG equivalent to the pretriangulated hull $\cA^{\text{pre-tr}}$ of $\cA.$
We also denote by  $\prfdg(\cA)$ the DG category of perfect DG modules, i.e. the full DG
subcategory of $\SF(\cA)$ consisting of all DG modules which are homotopy
equivalent to a direct summand of a finitely generated semi-free DG module.
Since any cohomologically bounded complex has a bounded above free resolution,
the DG subcategory $\prfdg(\cA)\cap \SF^{-}(\cA)$ is DG equivalent to $\prfdg(\cA).$
So we will consider $\prfdg(\cA)$ as a full DG subcategory of $\SF^{-}(\cA).$

For any semi-free DG module $P$ we can consider the "stupid" truncations of $P$ that by definition are complexes of the form
$$
\sigma_{\leq m}P=\{\cdots \lto P^{m-1}\lto P^m\lto 0\},
$$
$$
\sigma_{\geq n}P=\{0\lto P^{n}\lto P^{n+1}\lto \cdots\}.
$$
We also put
$$
P^{[n,m]}:=\sigma_{\geq n}\sigma_{\leq m}P=\{ 0\lto
P^n\lto\cdots\lto P^m\lto 0\}.
$$

The DG $\cA$\!-modules $\sigma_{\leq m}P,$ $\sigma_{\geq n}P,$
$P^{[n,m]}$ are also semi-free.
For every $n$ there is an exact sequence in $Z^0(\SF(\cA))$
$$
0\lto \sigma_{\geq (n+1)}P\lto P\lto \sigma_{\leq n}P \lto 0
$$
For a fixed $m$ we also have $\sigma_{\leq m}P\cong \colim_n P^{[n,m]}$ in
$Z^0(\SF^-(\cA))$ and hence $\sigma_{\leq m}P\cong \hocolim_n P^{[n,m]}$ in
$\Ho(\SF^{-}(\cA)).$

Let $\cU\subset\SF(\cA)$ be a full pretriangulated DG subcategory that contains
$\cA$ and if $P\in \cU$ then $\cU$ contains all stupid truncations $\sigma_{\leq m}P$ and $\sigma_{\geq  m}P$ as well.
Denote by $h^{\bullet}:\cA\to\Ho(\cU)$ the natural fully faithful functor.

Let $F: \Ho(\cU)\to \T$ be an exact functor to a triangulated category $\T$ that has the following properties

\begin{tabular}{ll}
$(*)$ & \begin{tabular}{ll}
1) & $F$ preserves all direct sums that exist in $\Ho(\cU);$\\
2) & $F(h^Y)$ is compact in $\T$ for every $Y\in \cA;$\\
3) & $\Hom(F(h^Y),F(h^Z)[s])=0$ for every $Y, Z\in \cA$ when $s<0.$
\end{tabular}
\end{tabular}
\begin{remark}{\rm
Since $F$ commutes with all direct sums that exist in $\Ho(\cU)$ then
it also commutes with any homotopy colimits that exists in $\Ho(\cU).$
}
\end{remark}

\begin{lemma}\label{nultrun} Let $\cU$ be either $\SF(\cA)$ or $\SF_{fg}(\cA)$ and let $\cT$ be a triangulated category
and $F:\Ho(\cU)\to \cT$
be an exact functor. Assume that $F$ has the properties
1),2),3) of $(*)$ above. Then for any $Y\in \cA$ and for any semi-free DG $\cA$\!-module $P\in \cU$ we have
$$\Hom (F(h^Y), F(\sigma_{\geq n}P)[i])=0$$
when $i<n.$
\end{lemma}

\begin{proof} The filtration on semi-free DG module $P$  induces a filtration
$\Phi^\prime_j:= \Phi_j\cap \sigma_{\geq n}P,
$
on $\sigma_{\geq n}P$
and each quotient $\Phi^\prime_{j+1}/\Phi^\prime_j$ is isomorphic to a direct sum
of object $h^Z[s]$ with $s\le -n.$ The assumptions 1),2), and 3)
imply that
\begin{multline*}
\Hom (F(h^Y), F (\Phi^\prime_{j+1}/\Phi^\prime_j)[i])=
\Hom (F(h^Y), F (\bigoplus h^Z[s])[i])\cong
\Hom (F(h^Y), \bigoplus F(h^Z[s+i]))\cong\\
\cong \bigoplus \Hom (F(h^Y), F(h^Z)[s+i])=
0.
\end{multline*}
Hence, by induction, $\Hom (F(h^Y),F( \Phi^\prime_{j})[i])=0$ for all $j.$

In the case $\cU=\SF_{fg}(\cA)$ everything is proved, because $\sigma_{\geq n}P\cong\Phi'_j$ for some $j.$

In the case $\cU=\SF(\cA)$ the object $\sigma_{\geq n}P$ is isomorphic to
$\hocolim \Phi'_j$ in $\Ho(\SF(\cA)).$
The functor $F$ preserves all direct sums,
hence it preserves homotopy colimits an we get
$$
F(\sigma_{\geq n}P)\cong\hocolim F(\Phi'_j).
$$
Since the object $F(h^Y)$ is compact, the functor $\Hom(F(h^Y), -)$  commutes with direct sums and carries
 homotopy colimits to colimits of abelian groups (see \cite{Ne} Lemma 1.5). Thus
$$
\Hom (F(h^Y),F(\sigma_{\geq n}P)[i])=\colim\Hom
(F(h^Y),F(\Phi'_j)[i])=0.
$$\end{proof}

\begin{cor}\label{injec} Under the assumptions of Lemma \ref{nultrun}
for every $Y\in \cA$ and every $m\ge0$ we have an injection
$$\Hom (F(h^Y),F(P))\hookrightarrow\Hom (F(h^Y),F(\sigma_{\leq m}P))\cong
\colim_{n}\Hom(F(h^Y),F(P^{[n,m]})),$$ which is a bijection when $m>0.$
\end{cor}

\begin{proof}
The isomorphism is a consequence of the facts
that $F$ preserves homotopy colimits and $F(h^Y)$ is compact.
The injection follows from the exact triangle
$$\sigma_{\geq (m+1)}P\to P\to \sigma_{\leq m}P$$
and the last lemma which gives that
$\Hom (F(h^Y), F(\sigma_{\geq (m+1)}P))=0$ for $m\ge 0.$ If $m>0$ we immediately obtain that the injection is also a bijection.
\end{proof}

\begin{prop}\label{mainprop}
Let $\cU$ be either $\SF^{-}(\cA)$ or $\prfdg(\cA),$ or $\SF_{fg}(\cA)$ and let $\T$ be a triangulated category.
 Let $F_1, F_2:\Ho(\cU)\to \T$ be two exact functors
that satisfy conditions $(*).$ Assume that there is an isomorphism of functors $\theta: F_1\cdot h^\bullet \stackrel{\sim}{\to}
F_2\cdot h^\bullet$  from $\cA$ to $\T.$ Then for  every $P\in \cU$ there exists an isomorphism $\theta_P:
F_1(P) \stackrel{\sim}{\to} F_2(P)$ such that for  any $Y\in \cA$ and every $f\in
\Hom_{\Ho(\cU)}(h^Y[-k], P)$ where $k\in \bbZ$ the diagram
\begin{equation}\label{commute}
\begin{CD}
F_1(h^Y)[-k] @>{F_1(f)}>>  F_1(P)\\
@V{\theta_Y}[-k]VV  @VV{\theta_P}V \\
F_2(h^Y)[-k] @>{F_2(f)}>> F_2(P)
\end{CD}
\end{equation}
commutes in $\T.$
\end{prop}

\begin{proof} We will construct $\theta_P$ in a few steps.

\noindent
{\bf Step 1.}
Denote by $h^{\bullet}(\cA)^{\oplus}\subset \Ho(\cU)$ the full subcategory of $\Ho(\cU)$
which is obtained from $h^{\bullet}(\cA)$ by adding arbitrary (existing in $\Ho(\cU)$)
direct sums of objects of $h^{\bullet}(\cA).$

Since the functors $F_i$ commute with direct sums we can extend our transformation $\theta$ onto
the whole subcategory $h^{\bullet}(\cA)^{\oplus}.$ Indeed, for any $P=\bigoplus h^Y$ the canonical isomorphisms
$\bigoplus F_i(h^Y)\cong F_i(\bigoplus h^Y)$ allow to define an isomorphism
$\theta_P: F_1(\bigoplus h^Y)\to F_2(\bigoplus h^Y)$ as the product of
the canonical morphisms $F_1(h^Y)\stackrel{\theta_Y}{\to} F_2(h^Y)\to F_2(\bigoplus h^Y).$

\noindent{\bf Step 2.}
Now consider a semi-free DG module $P=\{\cdots\to P^{m-1}\to P^m\}$
which  by assumption is bounded above.
For every $n\leq m$ we have the exact
triangle in $\Ho(\cU)$
$$
P^{n-1}[-n] \lto \sigma_{\ge n}P\lto \sigma_{\ge (n-1)} P.
$$
For all $P^k$ we have isomorphisms $\theta^{k}: F_1(P^k)\stackrel{\sim}{\to} F_2(P^k)$ constructed in Step 1.

In Step
2 by descending induction on $n\leq m$ we will construct
isomorphisms $\theta_{\ge n}:F_1(\sigma_{\ge n}P)\stackrel{\sim}{\to} F_2(\sigma_{\ge n}P)$ such
that there is an isomorphism of triangles
$$
\begin{CD}
F_1(P^{n-1})[-n] @>>> F_1(\sigma_{\ge n}P) @>>> F_1(\sigma_{\ge (n-1)}P) @>>> F_1(P^{n-1})[-n+1]\\
@V\theta^{n-1}[-n]VV   @V\theta_{\ge n}VV  @VV\theta_{\ge (n-1)}V  @VV\theta^{n-1}[-n+1]V\\
F_2(P^{n-1})[-n] @>>> F_2(\sigma_{\ge n}P) @>>> F_2(\sigma_{\ge (n-1)}P) @>>> F_2(P^{n-1})[-n+1]
\end{CD}
$$
for all $n\leq m.$

We start with the diagram
$$
\begin{CD}
F_1(P^{m-1}[-m]) @>>> F_1(P^m[-m]) @>>> F_1(\sigma_{\ge (m-1)}P)\\
@V\theta^{{m-1}}[-m]VV @VV\theta^{m}[-m]V  \\
F_2(P^{m-1}[-m]) @>>> F_2(P^m[-m]) @>>> F_2(\sigma_{\ge (m-1)}P)
\end{CD}
$$
where the square is commutative. Hence, there exists an isomorphism $\theta
_{{\ge (m-1)}}:F_1(\sigma_{\ge (m-1)}P)\to F_2(\sigma_{\ge (m-1)}P)$ which completes the
above diagram to a morphism of triangles:
$$
\begin{CD}
F_1(P^{m-1}[-m]) @>>> F_1(P^m[-m]) @>>> F_1(\sigma_{\ge (m-1)}P)@>>> F_1(P^{m-1}[-m+1])\\
@V\theta^{{m-1}}[-m]VV @VV\theta^{m}[-m]V  @VV\theta_{{\ge (m-1)}}V  @VV\theta^{{m-1}}[-m+1]V\\
F_2(P^{m-1}[-m]) @>>> F_2(P^m[-m]) @>>> F_2(\sigma_{\ge (m-1)}P) @>>> F_2(P^{m-1}[-m+1])
\end{CD}
$$
This provides the base of induction.

Assume by induction that we have defined isomorphisms $\theta
_{\ge (n+1)},$ $\theta_{\ge n}$ so that the diagram
\begin{equation}\label{induc}
\begin{CD}
F_1(\sigma_{\ge (n+1)}P) @>>> F_1(\sigma_{\ge n}P) @>>> F_1(P^{n})[-n]\\
@V\theta_{\ge (n+1)}VV  @V\theta_{\ge n}VV  @VV\theta^{n}[-n]V\\
F_2(\sigma_{\ge (n+1)}P) @>>> F_2(\sigma_{\ge n}P) @>>> F_2(P^{n})[-n]
\end{CD}
\end{equation}
is a morphism of triangles. Then we claim that the natural diagram
$$\begin{CD}
F_1(P^{n-1})[-n] @>>> F_1(\sigma_{\ge n} P)\\
@V\theta^{n-1}[-n]VV  @VV\theta_{\ge n}V\\
F_2(P^{n-1})[-n] @>>> F_2(\sigma_{\ge n} P)
\end{CD}
$$
commutes. Indeed, complete it to a diagram
$$\begin{CD}
F_1(P^{n-1})[-n] @>>> F_1(\sigma_{\ge n}P) @>>> F_1(P^n)[-n]\\
@V\theta^{n-1}[-n]VV  @V\theta_{\ge n}VV  @VV\theta^{n}[-n]V\\
F_2(P^{n-1})[-n] @>>> F_2(\sigma_{\ge n}P) @>>> F_2(P^n)[-n]
\end{CD}
$$
where the right square and the outside square commute. By Corollary \ref{injec}
the
natural map
$$
\Hom (F_1(P^{n-1})[-n],F_1(\sigma_{\ge n} P))\lto
\Hom (F_1(P^{n-1})[-n],F_1(P^{n})[-n])
$$
is injective. Using the
compositions with the isomorphisms $\theta_{\ge n}$ and $\theta
^n[-n]$ we conclude that the natural map
$$
\Hom (F_1(P^{n-1})[-n],F_2(\sigma_{\ge n} P))\lto
\Hom (F_1(P^{n-1})[-n],F_2(P^{n})[-n])
$$
is also injective.
Therefore the left square in the above diagram commutes too. Hence,
there is an isomorphism $\theta_{\ge (n-1)}:F_1(\sigma_{\ge (n-1)}P)\to
F_2(\sigma_{\ge (n-1)}P)$ such that the diagram
$$
\begin{CD}
F_1(P^{n-1})[-n] @>>> F_1(\sigma_{\ge n}P) @>>> F_1(\sigma_{\ge (n-1)}P) @>>> F_1(P^{n-1})[-n+1]\\
@V\theta^{n-1}[-n]VV   @V\theta_{\ge n}VV  @VV\theta_{\ge (n-1)}V  @VV\theta^{n-1}[-n+1]V\\
F_2(P^{n-1})[-n] @>>> F_2(\sigma_{\ge n}P) @>>> F_2(\sigma_{\ge (n-1)}P) @>>> F_2(P^{n-1})[-n+1]
\end{CD}
$$
is an isomorphism of triangles. This completes the induction step.

If $\cU=\SF_{fg}(\cA)$ the construction of isomorphisms $\theta_P$ is finished.

\noindent{\bf Step 3.}
Let now consider the case $\cU=\SF^{-}(\cA).$
Since $P$ is bounded above we have a natural isomorphism $\colim_n (\sigma_{\ge n}P)\stackrel{\sim}{\to} P$ in
$Z^0(\SF^-(\cA))$ and hence $P$ is isomorphic to $\hocolim_n (\sigma_{\ge n}P)$ in
$\Ho(\SF^{-}(\cA)).$ The canonical sequence
$$
\bigoplus_{n\le m}(\sigma_{\ge n}P) \lto \bigoplus_{n\le m} (\sigma_{\ge n}P) \lto P
$$
gives an exact triangle in $\Ho(\SF(\cA)).$

Since the functors $F_i$ commute with direct sums we can extend our isomorphisms $\theta_{\ge n}$ on
the direct sum. Indeed, for the object $\bigoplus (\sigma_{\ge n}P)$ the canonical isomorphism
$\bigoplus F_i(\sigma_{\ge n}P)\cong F_i(\bigoplus \sigma_{\ge n}P)$ allows to define an isomorphism
$\theta_{\bigoplus}: F_1(\bigoplus \sigma_{\ge n}P)\to F_2(\bigoplus \sigma_{\ge n} P)$ as the product of
the canonical morphisms $F_1(\sigma_{\ge n} P)\stackrel{\theta_{\ge n}}{\lto} F_2(\sigma_{\ge n}P)\stackrel{\can}{\lto}
F_2(\bigoplus \sigma_{\ge n}P).$
Now there is an isomorphism $\theta_P: F_1(P)\stackrel{\sim}{\to} F_2(P)$ that gives an isomorphism of triangles
$$
\begin{CD}
F_1(\bigoplus_n(\sigma_{\ge n}P)) @>>> F_1(\bigoplus_n (\sigma_{\ge n}P)) @>>> F_1(P)\\
@V\theta_{\bigoplus}VV   @V\theta_{\bigoplus}VV  @VV\theta_{P}V  \\
F_2(\bigoplus_n(\sigma_{\ge n}P)) @>>> F_2(\bigoplus(\sigma_{\ge (n-1)}P)) @>>> F_2(P)
\end{CD}
$$
It is easy to see that the isomorphism $\theta
_{P}:F_1(P)\stackrel{\sim}{\to} F_2(P)$ makes
the following diagrams commutative for each $n$
$$
\begin{CD}
F_1(\sigma_{\ge n}P) @>>> F_1(P)\\
@V\theta_{\ge n}VV @VV\theta_{P}V \\
F_2(\sigma_{\ge n}P) @>>> F_2(P).
\end{CD}
$$

\noindent{\bf Step 4.}
Let us now consider the case $\cU=\prfdg(\cA)\subset\SF^{-}(\cA).$ By definition the object $P\in \prfdg(\cA)$
is a direct summand of a finitely generated object $Q\in \SF_{fg}(\cA)$ in the triangulated category $\Ho(\cU)\subset D(\cA).$
Take a sufficiently negative $n\ll 0.$ The object $\sigma_{\ge n} P$ is finitely generated semi-free and, hence,
the object $\sigma_{\le (n-1)}P$ is a perfect DG module as well. Thus it is homotopy
equivalent to a direct summand of a finitely generated semi-free DG module $R.$ It is easy to see that the object
$\sigma_{\le (n-1)}P$ is also a direct summand of the finitely generated DG module $\sigma_{\le (n-1)}R.$
Let us consider a diagram
\begin{equation}\label{diagperf}
\begin{CD}
F_1(\sigma_{\le (n-1)}P)[-1] @>>> F_1(\sigma_{\ge n}P) @>>> F_1(P) @>>> F_1(\sigma_{\le (n-1)}P)\\
&&  @V{\theta_{\ge n}}V{\wr}V   \\
F_2(\sigma_{\le (n-1)}P)[-1] @>>> F_2(\sigma_{\ge n}P) @>>> F_2(P) @>>> F_2(\sigma_{\le (n-1)}P)
\end{CD}
\end{equation}
where $\theta_{\ge n}$ was constructed in Step 2.
Since $n\ll 0$ and $\sigma_{\le (n-1)}R, Q\in \SF_{fg}(\cA)$ we obtain that
$$
\Hom (F_j(\sigma_{\le (n-1)}R)[i], F_k(Q))=0\quad\text{for all}\quad
i\ge -1,\quad \text{and}\quad j,k=1,2
$$
by Lemma \ref{nultrun}. This implies that for their direct summands $\sigma_{\le (n-1)}P$ and $P$ we also have
\begin{equation}\label{vanish}
\Hom (F_j(\sigma_{\le (n-1)}P)[-1], F_k(P))=0\quad\text{for all}\quad
i\ge -1,\quad \text{and}\quad j,k=1,2.
\end{equation}
 Therefore, there are unique morphisms $\theta_{P}: F_1(P){\to} F_2(P)$ and
 $\theta_{\le (n-1)}: F_1(\sigma_{\le (n-1)}P)\to F_2(\sigma_{\le (n-1)}P)$ that complete
the diagram (\ref{diagperf}) to a morphism of triangles.
Since $\theta_{\ge n}$ is an isomorphism, cones $C(\theta_P)$ and $C(\theta_{\le (n-1)})$ are isomorphic.
Vanishing conditions (\ref{vanish}) implies that
$$
\Hom(F_2(\sigma_{\le (n-1)}P), C(\theta_P))=0\quad\text{and}\quad\Hom(F_1(\sigma_{\le (n-1)}P)[1], C(\theta_P))=0.
$$
Hence, the cone $C(\theta_{\le (n-1)})\cong C(\theta_P)$ is trivial and
$\theta_P$ is an isomorphism.

Finally, since for any $n\ll 0$ the isomorphism $\theta_P$ is unique
it does not depend on $n.$  More precisely, for any $n\ll 0$ and each $k>n$  the right and the left squares in the diagram
$$
\begin{CD}
F_1(\sigma_{\ge k}P) @>>> F_1(\sigma_{\ge n}P) @>>> F_1(P)\\
@V\theta_{\ge k}VV  @VV\theta_{\ge n}V  @VV\theta_PV\\
F_2(\sigma_{\ge k}P) @>>> F_2(\sigma_{\ge n}P) @>>> F_2(P)
\end{CD}
$$
are commutative. Hence the outside square is also commutative for each $k\in\bbZ.$

\noindent{\bf Step 5.}
Now we should check that the diagram (\ref{commute})
$$
\begin{CD}
F_1(h^Y)[-k] @>{F_1(f)}>> F_1(P)\\
@V{\theta_Y}[-k]VV  @VV{\theta_P}V\\
F_2(h^Y)[-k] @>{F_2(f)}>>  F_2(P)
\end{CD}
$$
 commutes for any $Y\in \cA$ and every  $f\in \Hom_{\Ho(\cU)}(h^Y[-k],P).$

The map $f$ is
represented by a morphism of the complexes $f: h^Y[-k]\to P.$ It is decomposed as
$$
h^Y[-k]\to \sigma_{\ge k}P\to P.
$$
In the diagram
$$\begin{CD}
F_1(h^Y)[-k] @>>> F_1(\sigma_{\geq k}P) @>>> F_1(P)\\
@V{\theta_Y}[-k]VV  @VV{\theta_{\geq k}}V  @VV{\theta_P}V\\
F_2(h^Y)[-k] @>>> F_2(\sigma_{\geq k}P) @>>> F_2(P)
\end{CD}
$$
the right square commutes for any $k\in \bbZ$ by the construction of $\theta_P.$
Thus it is sufficient  to prove that the first square commutes too.
The map $h^Y[-k]\to \sigma_{\ge k}P\to P$ is induced by a map $f_k: h^Y\to P^k.$ Consider the diagram
$$
\begin{CD}
F_1(h^Y)[-k] @>>> F_1(\sigma_{\geq k}P) @>>> F_1(P^k)[-k]\\
@V{\theta_Y[-k]}VV  @VV{\theta_{\geq k}}V @VV{\theta^k[-k]}V\\
F_2(h^Y)[-k] @>>> F_2(\sigma_{\geq k}P) @>>> F_2(P^k)[-k]
\end{CD}
$$
 The right square commutes by (\ref{induc}) and the outside square commutes by Step 1.

By Corollary \ref{injec}
the
natural map
$$
\Hom (F_1(h^Y)[-k],F_1(\sigma_{\ge k} P))\lto
\Hom (F_1(h^Y)[-k],F_1(P^{k})[-k])
$$
is injective. Using the
compositions with the isomorphisms $\theta_{\ge k}$ and $\theta
^k[-k]$ we conclude that the natural map
$$
\Hom (F_1(h^Y)[-k],F_2(\sigma_{\ge k} P))\lto
\Hom (F_1(h^Y)[-k],F_2(P^{k})[-k])
$$
is also injective.
Therefore the left square in the above diagram commutes too.
\end{proof}

\section{Preliminary Constructions}\label{cons}

Let $\cA$ be a small category which we consider as a DG category and $L\subset D(\cA)$ be a
localizing subcategory with the quotient functor $\pi :D(\cA)\to
D(\cA)/L.$

As above we denote by $\Mod\cA$ the DG category of DG $\cA$\!-modules.
Consider the DG subcategories $\SF(\cA)\subset\cP(\cA)\subset \Mod\cA$ of semi-free and
h-projective DG modules, respectively. The canonical DG functors $\SF(\cA)\hookrightarrow\cP(\cA)\hookrightarrow\Mod\cA/\Ac(\cA)$ are quasi-equivalences
and give equivalent enhancements of the derived category $D(\cA)$ (see Example \ref{exdg}).
As a consequence, the category $D(\cA)/L$ has
canonical enhancements. Namely, if $\cL \subset \Mod\cA$ is
the full DG subcategory having the same objects as $L$ then
the natural DG functors
$$
\SF(\cA)/\cL\cap \SF(\cA)\lto \cP(\cA)/\cL\cap \cP(\cA)\lto \Mod\cA/\cL
$$
are
quasi-equivalences, i.e. the canonical functors
$$
\Ho(\SF(\cA)/\cL\cap \SF(\cA))\stackrel{\sim}{\lto} \Ho(\cP(\cA)/\cL\cap \cP(\cA))\stackrel{\sim}{\lto} \Ho(\Mod\cA/\cL)\cong D(\cA)/L
$$
are equivalences.

Denote by $\widetilde{\pi}$  the
quotient DG functor $\widetilde{\pi}:\SF(\cA)\to
\SF(\cA)/\cL\cap \SF(\cA).$ We denote the
composition of the functor ${\Ho(h^\bullet)}:\cA {\to} \Ho(\Mod\cA)$ with the localization $\Ho(\Mod\cA)\to D(\cA)$
simply by $h^\bullet.$ Note that this functor $h^\bullet :\cA \to
D(\cA)$ is full and faithful.

Consider an enhancement  of the triangulated category of compact objects $(D(\cA)/L)^c.$
This means that there are given a pretriangulated DG category $\cC$
and an equivalence of triangulated
categories $\epsilon :(D(\cA)/L)^c\stackrel{\sim}{\to} \Ho(\cC).$  Denote by $\cB
\subset \cC$ the full DG subcategory with the set of objects
$\{\epsilon \pi(h^Y)\}_{Y\in \cA}.$
As in formula (\ref{phi}) before Proposition \ref{Keller1} there is the canonical DG functor
$$
\varPsi :\cC \lto \Mod\cB,\quad\text{ where}\quad \varPsi (X)(B)=\Hom_{\cC}(Y,X)\quad\text{ for}\quad B\in\cB, \; X\in\cC.
$$
In composition with DG functor $\Mod\cB\to \Mod\cB/\Ac(\cB)$ it induces (as in Remark \ref{first}) a quasi-functor
$$
\psi:\cC\lto \SF(\cB),\qquad
H^0(\psi) :\Ho(\cC)\stackrel{}{\lto}
\Ho(\SF(\cB))\cong D(\cB).
$$
Since the objects $\{\pi(h^Y)\}_{Y\in \cA}$ are compact in $D(\cA)/L$ we have
the following composition
$$
\cA\stackrel{\pi h^\bullet}{\lto} (D(A)/L)^c \stackrel{\epsilon}{\lto}\Ho(\cC)\stackrel{H^0(\psi)}{\lto}
D(\cB).
$$
By construction it factors through the canonical embedding $\Ho(\cB)\hookrightarrow \D(\cB)$
and we denote by $a$ the corresponding functor $a:\cA\to \Ho(\cB).$

Denote by $\tau_{\leq 0}\cB$ the DG subcategory with
the same objects as $\cB$ and morphisms
$$\Hom_{\tau_{\leq 0}\cB}(M,N)=\tau_{\leq 0}\Hom_{\cB}(M,N),$$
where $\tau_{\leq 0}$ is the usual truncation of complexes as in (\ref{trunc}).

We have the obvious diagram of DG
categories and DG functors
$$
\cA\stackrel{a}{\lto} \Ho(\cB) \stackrel{p}{\longleftarrow} \tau_{\leq 0}\cB
\stackrel{i}{\lto} \cB,
$$
where the functor $p: \tau_{\leq 0}\cB\to \Ho(B) $ is a quasi-equivalence by condition b) in the assumptions of Theorem \ref{main}.
They induce DG functors between semi-free DG modules and we obtain a commutative diagram
$$
\begin{CD}
\cA @>a>> \Ho(\cB) @<p<<  \tau_{\leq 0}\cB @>{i}>> \cB\\
@V{h}VV @V{h}VV @VV{h}V @VV{h}V\\
\SF(\cA)@>a^*>> \SF(\Ho(\cB)) @<{p^*}<< \SF(\tau_{\leq 0}\cB)
@>{i^*}>> \SF(\cB).
\end{CD}
$$
Passing to homotopy categories we get the following commutative diagram
\begin{equation}\label{maindiag}
\begin{CD}
\cA @>a>> \Ho(\cB) @= \Ho(\cB) @= \Ho(\cB)\\
@V{h}VV @V{h}VV @VV{h}V @VV{h}V\\
\Ho(\SF(\cA))@>\Ho(a^*)>> \Ho(\SF(\Ho(\cB))) @<\Ho(p^*)<< \Ho(\SF(\tau_{\leq 0}\cB))
@>\Ho(i^*)>> \Ho(\SF(\cB))\\
@| @| @| @|\\
D(\cA))@>{\bL a^*}>> D(\Ho(\cB)) @<{\bL p^*}<< D(\tau_{\leq 0}\cB)
@>\bL i^*>> D(\cB)
\end{CD}
\end{equation}
The DG-functor $p$ and, hence, DG functor $p^*$ are quasi-equivalences.
This gives us  a quasi-functor $p_*:\SF(\Ho(\cB))\to \SF(\tau_{\leq 0}\cB).$
This also implies that the functor $\Ho(p^*)\cong\bL p^*$ is an equivalence and the right
adjoint functor $p_*: D(\Ho(\cB))\to D(\tau_{\leq 0}\cB)$ is its quasi-inverse.

We denote by $\widetilde\rho$ the quasi-functor that is the composition
\begin{equation}\label{rho}
\widetilde\rho: \SF(\cA)\stackrel{ a^*}{\lto} \SF(\Ho(\cB)) \stackrel{p_*}{\lto} \SF(\tau_{\leq 0}\cB)
\stackrel{i^*}{\lto} \SF(\cB)
\end{equation}
and denote by $F_1$ the induced functor
\begin{equation}\label{F1}
F_1=\Ho(\widetilde\rho): D(\cA)\stackrel{\bL a^*}{\lto} D(\Ho(\cB)) \stackrel{p_*}{\lto} D(\tau_{\leq 0}\cB)
\stackrel{\bL i^*}{\lto} D(\cB).
\end{equation}
The functor $F_1$ evidently preserves direct sums.

\section{Proof of Theorem \ref{main}}\label{Thmain}

Let $\cD$ be a pretriangulated DG category
and $\widetilde\epsilon :D(\cA)/L\stackrel{\sim}{\to} \Ho(\cD)$ be an equivalence of triangulated
categories.  Denote by $\cC
\subset \cD$ the full DG subcategory of compact objects in $\Ho(\cD).$
The equivalence $\widetilde\epsilon$ induces an  equivalence $\epsilon: (D(\cA)/L)^c\stackrel{\sim}{\to}\Ho(\cC).$
Thus we can apply the construction from the previous section.

As above denote by $\cB\subset\cC\subset\cD$ the full DG subcategory with the set of objects
$\{\widetilde\epsilon \pi (h^Y)\}_{Y\in \cA}.$
By construction in previous section formulas (\ref{rho}) and (\ref{F1}) give a quasi-functor
$$
\widetilde\rho: \SF(\cA)\lto\SF(\cB)
$$
and the corresponding functor
$
F_1=\Ho(\widetilde\rho): D(\cA)\to D(\cB).
$

Notice that, by assumption, the objects $\{\widetilde\epsilon \pi (h^Y)\}_{Y\in
\cA}$ are compact and by Remark \ref{remar} they  compactly generate the category $\Ho(\cD).$
Therefore, by Proposition \ref{Keller} the canonical DG functor
$\varPhi :\cD \to \Mod\cB,$ which is defined by formula (\ref{phi})
$
\varPhi (X)(B)=\Hom_{\cD}(B, X)$ for $B\in\cB, \; X\in\cD,$
induces a quasi-functor
$$
\phi:\cD\to \SF(\cB)\quad\text{ such that}\quad
H^0(\phi) :\Ho(\cD)\stackrel{\sim}{\lto}
\Ho(\SF(\cB))\cong D(\cB)
$$
is an equivalence.
Therefore, $\SF(\cB)$ is another enhancement of $D(\cA)/L$ which is equivalent to $\cD.$
Thus it is sufficient to prove that $\SF(\cB)$ is quasi-equivalent to $\SF(\cA)/\cL\cap \SF(\cA).$

Denote by $F_2$ the composition
of $H^0(\phi)$ with the equivalence $\widetilde\epsilon$  and the localization $\pi$
$$
F_2: D(\cA)\stackrel{\pi}{\lto} D(A)/L\stackrel{\widetilde\epsilon}{\lto}\Ho(\cD)\stackrel{H^0(\phi)}{\lto}
D(\cB).
$$

Thus we have two functors $F_1$ and $F_2$ from $D(\cA)$ to $D(\cB)$ and both of them enjoy properties
1), 2), 3) of $(*).$ By construction of $F_2$ and $a,$ the composition of $h^{\bullet}:\cA\to\D(\cA)$ and $F_2$
coincides with $\cA\stackrel{a}{\to}\Ho(\cB)\to D(\cB).$ The commutativity of diagram (\ref{maindiag}) immediately proves  the following
lemma.
\begin{lemma}\label{simple} There exists an isomorphism  $\theta: F_1\cdot h^\bullet \stackrel{\sim}{\to}
F_2 \cdot h^\bullet$ of functors from $\cA$ to $D(\cB).$
\end{lemma}

Now we are ready to prove the following lemma.

\begin{lemma}\label{fact} The quasi-functor $\widetilde\rho:\SF(\cA)\to \SF(\cB)$ factors through the DG quotient
$\SF(\cA)/\cL\cap\SF(\cA).$
\end{lemma}
\begin{proof}
By Drinfeld Theorem \ref{Drin} it is sufficient to show that the corresponding functor $F_1:D(\cA)\to D(\cB)$
factors through the quotient $D(\cA)/L.$

Let $P\in\SF(\cA)$ be a semi-free DG module that belongs to $\cL.$ Consider it as the object of
$D(\cA).$ We have to show that $F_1(P)\cong 0.$ Since $\{F_1(h^Y)\}_{Y\in\cA}$ is a set of compact generators of $D(\cB)$
it is enough to check that
for any $Y\in \cA$ and any $k\in \bbZ$
$$
\Hom(F_1(h^Y)[k], F_1(P))=0
$$
Let us consider a stupid truncation
$
\sigma_{\ge (m+1)}P\to P\to \sigma_{\le m}P,
$
for a some $m>k.$

By Corollary \ref{injec}, we have an isomorphism
$$
\Hom(F_1(h^Y)[k], F_1(P))\cong\Hom(F_1(h^Y)[k], F_1(\sigma_{\le m}P)).
$$
On the other hand, $\sigma_{\le m}P$ belongs to $\SF^{-}(\cA)$
and by Proposition \ref{mainprop} there is an isomorphism $F_1(\sigma_{\le m}P)\cong F_2(\sigma_{\le m}P).$
Now applying Corollary \ref{injec} to the functor
$F_2$ we obtain isomorphisms
\begin{multline*}
\Hom(F_1(h^Y)[k], F_1(P))\cong \Hom(F_1(h^Y)[k], F_1(\sigma_{\le m} P))\cong\Hom(F_2(h^Y)[k], F_2(\sigma_{\le m} P))\cong
\\\cong\Hom(F_2(h^Y)[k], F_2(P))=0,
\end{multline*}
where the last equality holds because $F_2(P)=H^0(\phi)\cdot\widetilde\epsilon\cdot\pi(P)=0.$

Thus, by Theorem \ref{Drin} there
exists a quasi-functor
$$
\rho: \SF(\cA)/\cL \cap \SF(\cA) \lto \SF(\cB)
$$
and an isomorphism of
functors $F_1=\Ho(\widetilde\rho)\cong \Ho(\rho)\cdot \pi.$
\end{proof}

Now the following lemma finishes the proof of the theorem.

\begin{lemma}\label{equiv}
The functor $\Ho(\rho):D(\cA)/L\lto D(\cB)$ is an equivalence.
\end{lemma}
\begin{proof}
Let us  prove that the functor $\Ho(\rho)$ is fully  faithful.

The set $\{\pi (h^Y)\}_{Y\in \cA}$ is a set of compact generators
for the category $D(\cA)/L,$ and the functor $\Ho(\rho)$ preserves
direct sums, since $\pi$ and $H^0(\widetilde\rho)$ do. Thus, it suffices to prove that  the map
$$
\Ho(\rho):\Hom (\pi (h^Y)[k],\pi (h^Z))\lto \Hom ( \Ho(\rho)\pi (h^Y)[k], \Ho(\rho) \pi (h^Z))
$$
is an isomorphism for every $Y,Z \in \cA$
and any $k\in \bbZ.$

Let us fix $Y,Z$ and $k$ as above.
Recall  that the localization functor $\pi$ has a right adjoint
functor $\mu :D(\cA)/L\to D(\cA)$ which is full and faithful, so
that the natural morphism of functors $\id_{D(\cA)/L}\to \pi \mu$ is an
isomorphism.
Let $P\cong\mu \pi (h^Z)\in D(\cA)$ be a semi-free DG $\cA$\!-module. We have $\pi (h^Z)\cong \pi (P)$
and the map
$$
\pi :\Hom_{D(\cA)}(h^Y[k], P)\lto \Hom_{D(\cA)/L}(\pi (h^Y)[k], \pi(P))
$$
is an isomorphism.

Consider the stupid truncation
$\sigma_{\le m}P$
for some $m>k.$
We have a commutative diagram
$$
\begin{CD}
\Hom_{D(\cA)}(h^Y[k], P ) @>{\sim}>>   \Hom_{D(\cA)/L}(\pi(h^Y)[k], \pi(P))\\
@VV{\wr}V  @VV{\wr}V \\
\Hom_{D(\cA)}(h^Y[k], \sigma_{\le m} P) @>{\pi}>> \Hom_{D(\cA)/L}(\pi(h^Y)[k], \pi (\sigma_{\le m} P))
\end{CD}
$$
where the right vertical arrow is an isomorphism by Corollary \ref{injec}. Hence, the lower horizontal arrow is an isomorphism too.

On the other hand, by Lemma \ref{simple} and Proposition \ref{mainprop}  there exist
 isomorphisms $\theta_Y: F_1(h^Y)\stackrel{\sim}{\to} F_2 (h^Y)$
 and
 $\theta_{\le m}: F_1 (\sigma_{\le m} P)\stackrel{\sim}{\to} F_2 (\sigma_{\le m} P)$
 such that the
 diagram

$$
\begin{CD}
\Hom_{D(\cA)}(h^Y[k], \sigma_{\le m} P ) @=   \Hom_{D(\cA)}(h^Y[k], \sigma_{\le m} P)\\
@V{F_1}VV  @VV{F_2}V \\
\Hom_{D(\cB)}(F_1 (h^Y)[k], F_1(\sigma_{\le m} P)) @>{\theta_{\le m} \cdot \theta_Y^{-1}}>>
\Hom_{D(\cB)}(F_2 (h^Y)[k], F_2 (\sigma_{\le m} P))
\end{CD}
$$
commutes.
Since $F_1\cong \Ho(\rho)\cdot \pi$ and $F_2=H^0(\phi)\cdot\widetilde\epsilon\cdot\pi$ we obtain the
following commutative diagram

$$
\begin{CD}
\Hom_{D(\cA)}(h^Y[k], \sigma_{\le m} P) @=  \Hom_{D(\cA)}(h^Y[k], \sigma_{\le m} P)\\
@V{\pi}V{\wr}V   @V{\pi}V{\wr}V \\
\Hom_{D(\cA)/L}(\pi (h^Y)[k],\pi (\sigma_{\le m} P)) @=  \Hom_{D(\cA)/L}(\pi (h^Y)[k], \pi (\sigma_{\le m} P))\\
@V{\Ho(\rho)}VV   @V{H^0(\phi)\cdot\widetilde\epsilon}V{\wr}V \\
\Hom_{D(\cB)}(\Ho(\rho) \pi (h^Y)[k], \Ho(\rho)\pi (\sigma_{\le m} P)) @>{\theta_{\le m} \cdot \theta_Y^{-1}}>>
\Hom_{D(\cB)}( F_2 (h^Y)[k], F_2 (\sigma_{\le m} P))
\end{CD}
$$
In this diagram all arrows, except possibly $\Ho(\rho),$ are
isomorphisms. Hence  $\Ho(\rho)$ is an isomorphism as well. Finally we obtain that in the commutative diagram
$$
\begin{CD}
\Hom_{D(\cA)/L}(\pi (h^Y)[k],\pi (P))@>{\sim}>> \Hom_{D(\cA)/L}(\pi (h^Y)[k],\pi (\sigma_{\le m} P))\\
@V{\Ho(\rho)}VV @V{\Ho(\rho)}V{\wr}V   \\
\Hom_{D(\cB)}(\Ho(\rho) \pi (h^Y)[k], \Ho(\rho)\pi (P))@>{\sim}>> \Hom_{D(\cB)}(\Ho(\rho) \pi (h^Y)[k],
\Ho(\rho)\pi (\sigma_{\le m} P))
\end{CD}
$$
the horizontal arrows are isomorphisms by Corollary \ref{injec} and the right arrow is also an isomorphism as proved above.
Thus we get that the canonical map
$$
\Ho(\rho):\Hom (\pi (h^Y)[k],\pi (h^Z))\lto \Hom ( \Ho(\rho)\pi (h^Y)[k], \Ho(\rho) \pi (h^Z))
$$
is an isomorphism.

Thus we proved
that $\Ho(\rho)$ is fully faithful. Since the image of $\Ho(\rho)$ contains the set of
compact generators $\cB\subset D(\cB)$ and is closed under taking arbitrary direct sums, the functor
$\Ho(\rho)$ is essentially surjective. Therefore, it is an equivalence.
This proves the lemma and the
theorem.
\end{proof}

\begin{theo}\label{fufabig}
Let $\cA$ be a small category which we consider as a DG category
 and $L\subset D(\cA)$ be a
localizing subcategory with the quotient functor $\pi :D(\cA)\to
D(\cA)/L$ that has a right adjoint (Bousfield localization) $\mu.$ Assume that the following conditions hold
\begin{enumerate}
\item[a)]  for every $Y\in \cA$ the object  $\pi (h^Y)\in D(\cA)/L$ is
compact;
\item[b)] for every $Y,Z\in \cA$ we have
$\Hom (\pi (h^Y),\pi (h^Z)[i])=0\quad \text{when}\quad i<0.$
\end{enumerate}
Let  $\cE$ be a DG category and let $F: D(\cA)/L\to \Ho(\cE)$ be a fully faithful functor.
Then there is a quasi-functor $\cF:\SF(\cA)/\cL\cap\SF(\cA)\to\cE$ such that
\begin{enumerate}
\item the functor
$\Ho(\cF):D(\cA)/L\to \Ho(\cE)$ is also fully faithful;
\item $\Ho(\cF)(X)\cong F(X)$ for any $X\in D(\cA)/L.$
\end{enumerate}
\end{theo}
\begin{proof}
(1) Let $\cD\subset\cE$ be a full DG subcategory that consists of all objects in the essential image of $F.$
Then DG category $\cD$ is another enhancement for $D(\cA)/L$
the functor $F$ induces
an equivalence $\widetilde\epsilon :(D(\cA)/L)\stackrel{\sim}{\to} \Ho(\cD)$ between  the triangulated
categories.

As in the proof of Theorem \ref{main} we
denote by $\cB
\subset \cD$ the full DG subcategory with the set of objects
$\{\widetilde\epsilon \pi(h^Y)\}_{Y\in \cA}.$ By Proposition \ref{Keller} a DG functor
$\varPhi :\cD \to \Mod\cB$ induces a quasi-functor  $\phi:\cD\to\SF(\cB)$
such that
$
H^0(\phi) :\Ho(\cD)\stackrel{\sim}{\lto}D(\cB)
$
is an equivalence.

By construction in  Section \ref{cons} (formula (\ref{rho}) we have a quasi-functor
$$
\widetilde\rho: \SF(\cA)\lto\SF(\cB)
$$
and by Lemma \ref{fact}
the quasi-functor $\widetilde\rho$ factors through the DG quotient
$\SF(\cA)/\cL\cap\SF(\cA).$ Hence it induces a quasi-functor
$\rho:\SF(\cA)/\cL\cap\SF(\cA)\to \SF(\cB).$ In Lemma \ref{equiv} we proved that
our quasi-functor $\rho$ is a quasi-equivalence.
We denote by $\cF$ the composition of quasi-functors $\rho, \phi^{-1}$
and the full embedding of $\cD$ to $\cE.$
It is evident that $H^0(\cF)$ is fully faithful by construction.

(2) Lemma \ref{simple} implies  that there is an isomorphism of functors
$\theta: \Ho(\cF)\cdot\pi\cdot h^{\bullet}\stackrel{\sim}{\to} F\cdot\pi\cdot h^{\bullet}$
 from $\cA$ to $H^{0}(\cE).$
By Proposition \ref{mainprop} applied to $\cU=\SF^{-}(\cA)$
there is an isomorphism $\Ho(\cF)\pi(P)\cong F\pi(P)$ for any $P\in \SF^{-}(\cA).$

Denote as above the functor $H^0(\rho)\pi$ and
$H^0(\phi)\widetilde\epsilon\pi$ from $\SF(\cA)$ to $\SF(\cB)$ by $F_1$ and $F_2,$ respectively.
Now let $P\in \SF(\cA)$ be any semi-free DG module.
Consider its truncations
$\sigma _{\leq m}P\in\SF^{-}(\cA).$
The object $P$ is isomorphic to a homotopy limit of the truncations $\sigma_{\le m}P$ in $D(\cA),$ i.e. there is an exact triangle
$$
\begin{CD}
P@>>> \prod\limits_m \sigma_{\le m} P@>>>\prod\limits_m \sigma_{\le m} P.
\end{CD}
$$
We have a natural map $\kappa_P: F_i(P)\to \holim F_i(\sigma_{\le m} P).$ By Lemma \ref{nultrun} for any $Y\in\cA$ and $j\in\bbZ$
we have $\Hom (F_i(h^Y), F_i(\sigma_{\geq m}P)[j])=0$ when $m>j.$ Therefore,
the map
$$
\Hom (F_i(h^Y), F_i(P)[j])\stackrel{\sim}{\lto} \Hom (F_i(h^Y), \holim F_i(\sigma_{\le m}P)[j])
$$
is an isomorphism for any $Y$ and $j.$
Since $F_i(h^Y)$ forms a set of compact generators of $\SF(\cB)$
the map $\kappa_P$ is an isomorphism.

Now we construct an isomorphism $\vartheta_P: F_1(P)\to F_2(P).$
By Proposition \ref{mainprop} there are
isomorphisms $\theta^m:F_1(P^m)\stackrel{\sim}{\to} F_2(P^m)$ for all $m$ and $\theta
_{\leq 0}:F_1(\sigma _{\leq 0}P)\stackrel{\sim}{\to} F_2(\sigma _{\leq 0}P)$ making
the following  square commutative.
$$
\begin{CD}
F_1(P^0) @>>> F_1(\sigma _{\leq 0}P) \\
@V\theta ^0VV  @VV\theta _{\leq 0}V \\
F_2(P^0) @>>> F_2(\sigma _{\leq 0}P)
\end{CD}
$$

For every $m\ge0$ we will construct an
isomorphism $\vartheta _{\le m}:F_1(\sigma _{\leq m}P)\stackrel{\sim}{\to} F_2(\sigma _{\leq
m}P)$ such that the diagram
$$
\begin{CD}
F_1(P^m)[-m]@>>> F_1(\sigma _{\leq m}P) @>>>  F_1(\sigma _{\leq m-1}P)\\
@V\theta^m[-m]VV  @V\vartheta _{\leq m}VV @VV \vartheta _{\leq m-1}V \\
F_2 (P^m)[-m] @>>> F_2(\sigma _{\leq m}P) @>>>  F_2(\sigma _{\leq
m-1}P)
\end{CD}
$$
commutes for $m>0.$ The proof
uses the (ascending) induction on $m$ with the base of induction for
$m=0.$ Assume that the map
$\vartheta _{\le m}:F_1(\sigma _{\leq m}P)\stackrel{\sim}{\to} F_2(\sigma _{\leq
m}P)$
exists for some $m\geq 0$.
We claim that in the natural diagram
$$
\begin{CD}
F_1(P^m)[-m]@>>> F_1(\sigma _{\leq m}P) @>>> F_1(P^{m+1})[-m]\\
@V\theta^m[-m]VV @V\vartheta _{\leq m}VV  @VV\theta ^{m+1}[-m]V\\
F_2(P^m)[-m]@>>> F_2(\sigma _{\leq m}P) @>>> F_2(P^{m+1})[-m]
\end{CD}
$$
the right square commutes. Indeed, the outside square commutes by Proposition \ref{mainprop}
and the left one commutes by induction assumption.
Further, we have that the
map
$$
\Hom (F_1(\sigma _{\leq m}P),F_1(P^{m+1})[-m])\to \Hom
(F_1(P^m)[-m],F_1(P^{m+1})[-m])
$$
is injective, because
$$
\Hom (F_1(\sigma_{\le m-1}P),F_1(P^{m+1})[-m])=\Hom(\hocolim_k F_1( P^{[k, m-1]})[m], F_1(P^m))=0
$$
by property 3) of (*).
Using the
compositions with $\theta^{m+1}$ we conclude that the natural map
$$
\Hom (F_1(\sigma _{\leq m}P),F_2(P^{m+1})[-m])\to \Hom
(F_1(P^m)[-m],F_2(P^{m+1})[-m])
$$
is also injective.
Therefore, the right square in the above diagram commutes too.
Hence we can find $\vartheta _{\le m+1}$ so
that the diagram
$$\begin{CD}
F_1(P^{m+1})[-m-1] @>>> F_1(\sigma _{\leq {m+1}}P) @>>>
F_1(\sigma _{\leq m}P) @>>> F_1(P^{m+1})[-m]\\
@V\theta ^{m+1}[-m-1]VV   @V\vartheta _{\leq m+1}VV  @VV\vartheta _{\leq m}V @VV\theta ^{m+1}[-m]V\\
F_2 (P^{m+1})[-m-1] @>>> F_2(\sigma _{\leq m+1}P) @>>>  F_2(\sigma
_{\leq m}P) @>>> F_2 (P^{m+1})[-m]
\end{CD}
$$
commutes. This provides the induction step.
Since $F_i(P)\cong\holim F_i(\sigma_{\le m}P),$
there exists an isomorphism $\vartheta _P:F_1(P)\stackrel{\sim}{\to} F_2(P)$ that
gives an isomorphism of triangles
$$
\begin{CD}
F_1(P) @>>> \prod\limits_m F_1(\sigma _{\leq m}P)@>>>\prod\limits_m F_1(\sigma _{\leq m}P)\\
@V\vartheta _PVV @V\prod\vartheta _{\leq m}VV @VV\prod\vartheta _{\leq m}V\\
F_2(P) @>>> \prod\limits_m F_2(\sigma _{\leq m}P)@>>>\prod\limits_m F_2(\sigma _{\leq m}P)
\end{CD}
$$
Thus, $H^0(\cF)\pi(P)\cong F\pi(P)$ for any $P\in\SF(\cA).$
\end{proof}

\section{Proof of Theorem \ref{mainperf}}\label{Thmainperf}

As in Section \ref{cons}, let $\cC$ be a pretriangulated DG category
and  $\epsilon :(D(\cA)/L)^c\stackrel{\sim}{\to} \Ho(\cC)$ be an equivalence of  triangulated
categories.  Denote by $\cB
\subset \cC$ the full DG subcategory with the set of objects
$\{\epsilon \pi(h^Y)\}_{Y\in \cA}.$
As in formula (\ref{phi})  there is the canonical DG functor
$\varPsi :\cC \to \Mod\cB$ which induces a quasi-functor $\psi:\cC\to \SF(\cB)$
and the functor
$$
H^0(\psi) :\Ho(\cC)\stackrel{}{\lto}
\Ho(\SF(\cB))\cong D(\cB).
$$
Objects $\{\epsilon \pi(h^Y)\}_{Y\in \cA}$ classically generate the category $(D(\cA)/L)^c$ by Theorem \ref{Neeman}.
Hence by Proposition \ref{Keller1} the functor $H^0(\psi)$ induces an equivalence between
$\Ho(\cC)$ and $\Perf(\cB)=D(\cB)^c.$

Consider    the DG category of perfect DG modules $\prfdg(\cB),$ that by definition is the full DG
subcategory of $\SF(\cB)$ consisting of all DG modules which are homotopy
equivalent to a direct summand of a finitely generated semi-free DG module.
The quasi-functor $\psi$ induces a quasi-functor $\chi: \cC\to \prfdg(\cB)$ that is quasi-equivalence, i.e.
the induced functor
$$
H^0(\chi): H^0(\cC)\stackrel{\sim}{\lto} D(\cB)^c=\Perf(\cB)
$$
is an equivalence. Thus
the DG category $\prfdg(\cB)$ is another enhancement of $(D(\cA)/L)^c$ which is equivalent to $\cC$ via
the quasi-functor $\chi.$
Therefore it is sufficient to prove that $\prfdg(\cB)$ is quasi-equivalent to a natural enhancement coming from the
DG quotient. More precisely, we can consider a full DG subcategory $\cV\subset\SF(\cA)/\cL\cap\SF(\cA)$ which consists
of all objects that belong to the subcategory of compact objects
$(D(\cA)/L)^c.$
Thus we have to show that the DG category $\cV$ is quasi-equivalent to $\prfdg(\cB).$

By construction in  Section \ref{cons} formulas (\ref{rho}) and (\ref{F1}) give a quasi-functor
$$
\widetilde\rho: \SF(\cA)\lto\SF(\cB)
$$
and the induced functor $F_1=\Ho(\widetilde\rho): D(\cA)\to D(\cB).$
The restriction of $F_1$ to the subcategory of compact objects $D(\cA)^c$ induces a functor
$$
G_1:\Perf(\cA)=D(\cA)^c{\lto}
D(\cB)^c=\Perf(\cB).
$$

Denote by $G_2$ the composition
of $H^0(\chi)$ with the equivalence $\epsilon$  and the localization $\pi$
$$
G_2: \Perf(\cA)=D(\cA)^c\stackrel{\pi}{\lto} (D(A)/L)^c\stackrel{\epsilon}{\lto}\Ho(\cC)\stackrel{H^0(\chi)}{\lto}
D(\cB)^c=\Perf(\cB).
$$

Thus we have two functors $G_1$ and $G_2$ from $D(\cA)^c$ to $D(\cB)^c$ and both of them enjoy the properties
1), 2), 3) of $(*).$ By construction of $G_2$ and $a,$ the composition of $h^{\bullet}:\cA\to\D(\cA)^c$ and $G_2$
coincides with $\cA\stackrel{a}{\to}\Ho(\cB)\to D(\cB)^c.$ The commutativity of diagram (\ref{maindiag})
immediately proves  the following
lemma.
\begin{lemma}\label{simple2} There exists an isomorphism  $\theta: G_1\cdot h^\bullet \stackrel{\sim}{\to}
G_2 \cdot h^\bullet$ of functors from $\cA$ to $D(\cB)^c.$
\end{lemma}

As in the proof of Theorem \ref{main}, first,  we show that the following lemma holds.
\begin{lemma}\label{factor}
The quasi-functor $\widetilde\rho:\SF(\cA)\to\SF(\cB)$ factors through the DG quotient
$\SF(\cA)/\cL\cap\SF(\cA).$
\end{lemma}
\begin{proof}
By Drinfeld Theorem \ref{Drin} it is sufficient to show that the induced functor $F_1: D(\cA)\to D(\cB)$
factors through the quotient $D(\cA)/L.$
Let $P\in\prfdg(\cA)$ be a perfect DG module that belongs to $\cL.$
By Lemma \ref{simple2} and Proposition \ref{mainprop} we have
$$
G_1(P)\cong G_2(P)=H^0(\chi)\cdot\epsilon\cdot\pi(P)=0.
$$
By assumption, the subcategory $L$ is generated by objects compact in $D(\cA).$
Since $F_1$ commutes with direct sums
we get that $F_1(P)=0$ for any $P\in \SF(\cA)\cap\cL.$
Hence, the functor $F_1:D(\cA)\to D(\cB)$
factors through the quotient $D(\cA)/L,$
and, by Theorem \ref{Drin}, there
is a quasi-functor
$$
\rho: \SF(\cA)/\cL \cap \SF(\cA) \lto \SF(\cB)
$$
with an isomorphism of
functors $F_1=\Ho(\widetilde\rho)\cong \Ho(\rho)\cdot \pi.$
\end{proof}

Next we would like
to prove the following lemma
\begin{lemma}\label{equivp}
The functor $\Ho(\rho): D(\cA)/L\lto D(\cB)$ is an equivalence.
\end{lemma}
\begin{proof}
 We already know
that $\Ho(\rho)$ is essentially surjective, because the image of $\Ho(\rho)$ contains the set of
compact generators $\cB\subset D(\cB)$ and is closed under taking arbitrary direct sums.
Now we show that it is full and faithful.

As usual, since the set $\{\pi (h^Y)\}_{Y\in \cA}$ is a set of compact generators
for  $D(\cA)/L,$ and the functor $\Ho(\rho)$ preserves
direct sums, it suffices to prove that for any $Y,Z \in \cA,$
and any $k\in \bbZ$ the map
$$
\Ho(\rho):\Hom (\pi (h^Y)[k],\pi (h^Z))\lto \Hom ( \Ho(\rho)\pi (h^Y)[k], \Ho(\rho)\pi(h^Z))\cong \Hom ( G_1 (h^Y)[k], G_1 (h^Z))
$$
is an isomorphism.

By Theorem \ref{Neeman}, since $L$ is generated by compact objects,
the natural functor $D(\cA)^c/L^c\to (D(\cA)/L)^c$ is fully faithful. Hence, by definition of
localization, any morphism $f: \pi(h^Y)[k] \to \pi(h^Z)$ is represented by a pair $(g, s)$ in $D(\cA)^c$ of the form
$$
h^Y[k] \stackrel{g}{\lto} P \stackrel{s}{\longleftarrow} h^Z,
$$
where $P$ is a perfect
DG module and a cone of $s$ belongs to $L^c,$ i.e $\pi(s)$ is an isomorphism and $f=\pi(s)^{-1} \pi(g).$

{\bf Full.}
Consider a morphism $f_1 \in \Hom ( G_1 (h^Y)[k], G_1 (h^Z)).$ The isomorphism $\theta: G_1\cdot h^\bullet \stackrel{\sim}{\to}
G_2 \cdot h^\bullet$ induces a morphism
$f_2=\theta_Z\cdot  f_1\cdot  \theta_Y[k]^{-1}$ from $G_2 (h^Y)[k]$ to $G_2 (h^Z).$

The functor $G_2$ is the composition $H^0(\chi)\cdot\epsilon\cdot\pi$ and $H^0(\chi)\cdot\epsilon:
(D(\cA)/L)^c\stackrel{\sim}{\to} \Ho(\cB)$
is an equivalence.

Denote by $f: \pi(h^Y)[k] \to \pi(h^Z)$ the morphism for which $H^0(\chi)\cdot\epsilon(f)=f_2.$
By Proposition \ref{mainprop} applied to $\cU=\prfdg(\cA)$ there is an isomorphism $\theta_{P}: G_1(P)\stackrel{\sim}{\to} G_2(P)$
such that in the following diagram

\begin{equation}\label{prizm}
\xymatrix{
& G_1(P) \ar@{->}'[d][dd]^(.2){\wr}_(.3){\theta_P}\\
G_1(h^Y)[k]\ar[ur]^{G_1(g)}\ar[rr]^(.4){f_1}\ar[dd]_{\theta_{Y}[k]}^{\wr}&& G_1(h^Z)\ar[ul]_{G_1(s)}\ar[dd]^{\theta_{Z}}_{\wr}\\
& G_2(P)\\
G_2(h^Y)[k]\ar[ur]^{G_2(g)}\ar[rr]^(.4){f_2}&& G_2(h^Z)\ar[ul]_{G_2(s)}
}.
\end{equation}
the three squares are commutative and the lower triangle is also commutative.
Now the following
sequence of equalities
\begin{equation}
\begin{split}
f_1&=\Ho(\rho)(\pi(s)^{-1})\cdot G_1(s)\cdot f_1=\Ho(\rho)(\pi(s)^{-1})\cdot G_1(s)\cdot \theta_Z^{-1}\cdot f_2\cdot \theta_Y[k]=
\\
&=\Ho(\rho)(\pi(s)^{-1})\cdot \theta_P^{-1}\cdot G_2(s)\cdot  f_2\cdot \theta_Y[k]
=\Ho(\rho)(\pi(s)^{-1})\cdot \theta_P^{-1}\cdot G_2(g)\cdot \theta_Y[k]=
\\
&=\Ho(\rho)(\pi(s)^{-1})\cdot G_1(g)=
\Ho(\rho)(\pi(s)^{-1}\pi(g))=\Ho(\rho)(f)
\end{split}
\end{equation}
show us that $f_1$ is in the image of the functor $\Ho(\rho).$

{\bf Faithful.}
Consider a morphism $f: \pi(h^Y)[k]\to \pi(h^Z).$
As above there are a perfect DG $\cA$\!-module $P\in \prfdg(\cA)$ and a pair of morphisms $(g, s)$ in $D(\cA)^c$ of the form
$$
h^Y[k] \stackrel{g}{\lto} P \stackrel{s}{\longleftarrow} h^Z
$$
such that the morphism $\pi(s)$  is an isomorphism and $f=\pi(s)^{-1}\pi(g).$

Denote by $f_1$ and $f_2$ the images of $f$ under the functors $\Ho(\rho)$ and  $H^0(\chi)\cdot\epsilon,$ respectively
Consider again the diagram (\ref{prizm}). Now both triangles are commutative and both back squares are commutative too.
Assume that $f_1=\Ho(\rho)(f)=0.$ Then $G_1(g)=0$ and, as consequence, $G_2(g)=0.$ Thus we get
$$
f_2=H^0(\chi)\epsilon(\pi(s)^{-1}) G_2(s) f_2= H^0(\chi)\epsilon(\pi(s)^{-1}) G_2(g)=0.
$$
But the functor $H^0(\chi)\epsilon$ is an equivalence. We conclude that $f=0.$
This proves faithfulness.
\end{proof}

Thus, our quasi-functor $\rho: \SF(\cA)/\cL \cap \SF(\cA) \lto \SF(\cB)$ is a quasi-equivalence and it induces
a quasi-functor between subcategories of compact objects
$\rho^c: \cV\to \prfdg(\cB),$ where $\cV\subset\SF(\cA)/\cL\cap\SF(\cA)$ is the full DG subcategory which consists
of all compact objects $(D(\cA)/L)^c.$ And $\rho^c$ is a quasi-equivalence as well.
This finishes the proof of Theorem \ref{mainperf}.

More precisely we had proved the following theorem

\begin{theo}\label{corperf}
Let $\cA$ be a small category which we consider as a DG category
and $L\subset D(\cA)$ be a
localizing subcategory that is generated by compact objects $L^c=L\cap D(\cA)^c.$
Assume that for the quotient functor $\pi :D(\cA)\to
D(\cA)/L$ the following condition holds
\begin{enumerate}
\item[] for every $Y,Z\in \cA$ we have
$\Hom (\pi (h^Y),\pi (h^Z)[i])=0\quad \text{when}\quad i<0.$
\end{enumerate}
Let $\cV\subset\SF(\cA)/\cL\cap\SF(\cA)$ be the full DG subcategory which consists
of all compact objects $(D(\cA)/L)^c.$
Let  $\cE$ be another pretriangulated DG category and $N: (D(\cA)/L)^c \to \Ho(\cE)$ be a fully faithful functor.
Then there is a quasi-functor $\cN:\cV\to\cE$ such that
\begin{enumerate}
\item the functor
$\Ho(\cN): (D(\cA)/L)^c\to \Ho(\cE)$ is also fully faithful;
\item there is an isomorphisms of functors
$\theta: \Ho(\cN)\cdot\pi\cdot h^{\bullet}\stackrel{\sim}{\to} N\cdot\pi\cdot h^{\bullet}$
 from $\cA$ to $H^{0}(\cE);$
\item $\Ho(\cN)(X)\cong N(X)$ for any $X\in (D(\cA)/L)^c.$
\end{enumerate}
\end{theo}
\begin{proof}
Let $\cC\subseteq\cE$ be a full DG subcategory that consists of all objects in the essential image of $N.$
Then DG category $\cC$ is another enhancement for $(D(\cA)/L)^c$ and the functor $N$ induces
an equivalence $\epsilon :(D(\cA)/L)^c\stackrel{\sim}{\to} \Ho(\cC)$ between  the triangulated
categories.

As in the proof of Theorem \ref{mainperf} we
denote by $\cB
\subset \cC$ the full DG subcategory with the set of objects
$\{\epsilon \pi(h^Y)\}_{Y\in \cA}.$ By Proposition \ref{Keller1} a DG functor
$\varPsi :\cC \to \Mod\cB$ induces a quasi-functor  $\chi:\cC\to\prfdg(\cB)$
such that
$$
H^0(\chi) :\Ho(\cC)\stackrel{\sim}{\lto}D(\cB)^c=\prfdg(\cB).
$$
is an equivalence.
By construction in  Section \ref{cons} (formula (\ref{rho}) we have a quasi-functor
$$
\widetilde\rho: \SF(\cA)\lto\SF(\cB)
$$
and by Lemma \ref{factor}
the quasi-functor $\widetilde\rho$ factors through the DG quotient
$\SF(\cA)/\cL\cap\SF(\cA).$ Hence it induces a quasi-functor
$\rho:\SF(\cA)/\cL\cap\SF(\cA)\to \SF(\cB).$ In Lemma \ref{equivp} we proved that
our quasi-functor $\rho$ is a quasi-equivalence and, therefore,
it induces
a quasi-equivalence between DG subcategories of compact objects
$\rho^c: \cV\to \prfdg(\cB).$ Now we denote by $\cN$ the composition of quasi-functors $\rho^c, \chi^{-1}$
and the full embedding of $\cC$ to $\cE.$

(1) It is evident that $H^0(\cN)$ is fully faithful by construction.

(2) Lemma \ref{simple2} implies immediately that there is an isomorphisms of functors
$\theta: \Ho(\cN)\cdot\pi\cdot h^{\bullet}\stackrel{\sim}{\to} N\cdot\pi\cdot h^{\bullet}$
 from $\cA$ to $H^{0}(\cE).$

(3) By Proposition \ref{mainprop} applied to $\cU=\prfdg(\cA)$
there is an isomorphism $\Ho(\cN)\pi(P)\cong N\pi(P)$ for any $P\in D(\cA)^c.$
By Theorem \ref{Neeman} (4)  any object $X\in (D(\cA)/L)^c$ is a direct summand of an object $\pi(P),$ where
$P\in\SF_{fg}(\cA).$ Since for any $Y, Z\in\cA$ we have $\Hom(\pi(h^Y), \pi(h^Z)[i])=0,$ when $i<0,$ we can deduce that
$\Hom(\pi(P), \pi(P)[k])=0$ for sufficiently negative $k\ll 0.$ Therefore, we do not have any homomorphisms
from
$H^0(\cN)\pi(P)$ to  $N\pi(P)[k]$ and from $N\pi(P)$ to $H^0(\cN)\pi(P)[k]$ for $k\ll 0,$
because these functors are fully faithful and $H^0(\cN)\pi(P)\cong N\pi(P).$
 This implies that for $X$ as a direct summand of $\pi(P)$ we also have
\begin{equation}\label{vancon}
\Hom(H^0(\cN)(X), N(X)[k])=0,\quad \Hom(N(X), H^0(\cN)(X)[k])=0\quad\text{for}\quad k\ll 0.
\end{equation}
Consider an object $X\oplus X[2m+1]$ for sufficiently large $m\gg 0.$
Its class in the Grothendieck group $K_0((D(\cA)/L)^c)$ is equal to $0.$
By Lemma 2.2 of \cite{Tom} it belongs to any full subcategory whose idempotent completion
is $(D(\cA)/L)^c.$ Hence, $X\oplus X[2m+1]\cong\pi(Q)$
for some $Q\in\Perf(\cA)=D(\cA)^c.$ We know that there is an isomorphism
$H^0(\cN)\pi(Q)\stackrel{\sim}{\to}N\pi(Q).$ The vanishing conditions (\ref{vancon})
implies that such isomorphism induces an isomorphism $H^0(\cN)(X)\stackrel{\sim}{\to} N(X).$
\end{proof}


\section{Applications to commutative and non-commutative geometry}\label{geom}

Let $\cA$ be a small category. As above we can consider $\cA$ as DG category and let $\Mod\cA$ be the DG category of DG $\cA$\!-modules.
We also can consider the abelian category of right $\cA$\!-modules which we denote by $\umod\cA.$
The DG category of all complexes $C^{dg}(\umod\cA)$ is exactly the DG category $\Mod\cA$
and  the derived category $D(\umod\cA)$ of the abelian category $\umod\cA$ is exactly the derived category $D(\cA)$ introduced above.

Let $\cS\subset \umod\cA$ be a Serre (or dense) subcategory. By definition, this means that for any exact sequence
$$
0\lto M'\lto M\lto M''\lto 0
$$
$M$ is in $\cS$ if and only if both $M'$ and $M''$ are in $\cS.$

\begin{defi}
In this case we can define a quotient category $\QMod_{\cS}(\cA):=\umod\cA/\cS$ as a category with the same objects as $\umod\cA$
and
$$
\Hom_{\QMod_{\cS}(\cA)}(M, N):= \colim \Hom_{\umod\cA}(M', N/N'),
$$
where $M'$ and $N'$ are submodules of $M$ and $N$ such that $M/M'$ and $N'$ are in $\cS.$
\end{defi}

We will assume that the subcategory $\cS$ is {\it localizing}, i.e. it is a Serre category that is closed
under direct sums in $\umod\cA.$
(Hence it is closed under direct limits as well.)
Denote by $\varPi$ the canonical quotient functor from $\umod\cA$ to $\QMod_{\cS}(\cA).$ It is exact, preserve direct sums and,
since $\umod\cA$ is an AB5-category with a set of generators and with enough injective objects, the quotient category $\QMod_{\cS}(\cA)$
also has these properties. Moreover, in this case the quotient functor $\varPi$ has a right adjoint functor $\varOmega$
which is called a section functor.
The functor $\varOmega$ is left exact and  it is full and faithful, i.e. the natural morphism $\varPi\varOmega\to \id_{\QMod}$
is an isomorphism.
All these facts are standard theory of localization for abelian categories and can be found in \cite{Ga, Po}.

The functor $\varPi$ is exact and it induces the functor between the derived categories which we also denote by
$$
\varPi: D(\cA)=D(\umod\cA)\to D(\QMod_{\cS}(\cA)).
$$
Since $\varPi$ respect direct sums, by  Theorem 4.1 \cite{Ne3}, it has a right adjoint functor
$$
\bR\varOmega: D(\QMod_{\cS}(\cA))\to D(\umod\cA)=D(\cA).
$$

We are interesting only the case when $\varPi(h^Y)$ are compact in $D(\QMod_{\cS}(\cA))$ for all $Y\in\cA.$
This property is equivalent to the condition that $\bR\varOmega$ preserves direct sums.
Indeed, in the following commutative diagram

$$
\begin{CD}
\bigoplus\limits_i \Hom(\varPi (h^Y), X_i)@>{\sim}>> \bigoplus\limits_i \Hom(h^Y, \bR\varOmega(X_i))\\
@V{can}VV @V{can}VV   \\
\Hom(\varPi (h^Y), \bigoplus\limits_i X_i)@>{\sim}>>  \Hom(h^Y, \bR\varOmega(\bigoplus\limits_i X_i))
\end{CD}
$$
the right arrow is an isomorphism if and only if the left arrow is an isomorphism.

Now for any injective object $I\in \QMod_{\cS}(\cA)$ we have
$\varPi\bR\varOmega(I)\cong \varPi\varOmega(I)\cong I.$ Hence the functor $\bR\varOmega$ is fully faithful on the subcategory
of bounded below complexes $D^{+}(\QMod_{\cS}(\cA)).$ If now $\bR\varOmega$ preserves direct sums then it is fully faithful on the whole
derived category $D(\QMod_{\cS}(\cA)).$

Furthermore, for any object $M^{\cdot}\in D(\cA)$ we have an exact triangle of the form
$$
 N^{\cdot}\lto M^{\cdot}\lto \bR\varOmega\varPi(M^{\cdot})
$$
and $\varPi(N^{\cdot})\cong 0.$ This implies that cohomologies of $N^{\cdot}$ belongs to $\cS.$
Denote by $L_{\cS}$ the full triangulated subcategory of $D(\cA)$ that consists of all objects cohomologies of which belongs to $\cS.$
It is a localizing triangulated subcategory, because it is closed under taking direct sums. Since $\varPi(N^{\cdot})=0$
when $N^{\cdot}\in L_{\cS},$
the quotient functor
$\varPi: D(\cA)\to D(\QMod_{\cS}(\cA))$ factors through  the projection $\pi: D(\cA)\to D(\cA)/L_{\cS}.$ Moreover, it is evident now
that the functor $\bR\varOmega$ in composition with $\pi$ establish an equivalence between $D(\QMod(\cA))$ and $D(\cA)/L_{\cS},$
because both  these categories are equivalent to the  right orthogonal $L_{\cS}^{\perp}$ in $D(\cA).$
Thus we obtain
\begin{lemma}\label{quotab} Let $\cA$ be a small category and $\cS\subset \umod\cA$ be a localizing subcategory.
Let $\varPi:\umod\cA\to \QMod_{\cS}(\cA)$ be the canonical functor to the quotient category  and let $L_{\cS}\subset D(\umod\cA)$ be
the full triangulated subcategory of $D(\cA)$ that consists of all objects cohomologies of which belongs to $\cS.$
Assume that the objects $\varPi(h^Y)$ are compact in $D(\QMod_{\cS}(\cA))$ for all $Y\in\cA.$
Then the functor $\varPi$ induces an equivalence of triangulated categories
$$
\varPi': D(\umod\cA)/L_{\cS}\stackrel{\sim}{\lto} D(\QMod_{\cS}(\cA)).
$$
\end{lemma}
Now Theorem \ref{main} together with this lemma implies the following proposition.
\begin{prop}\label{qcat} Let $\cA$ be a small category and $\cS\subset \umod\cA$ be a localizing subcategory.
 Then the derived category of the quotient abelian category $D(\QMod_{\cS}(\cA))$ has a unique enhancement if
the objects $\varPi(h^Y)\in D(\QMod_{\cS}(\cA))$ are compact for all $Y\in\cA.$
\end{prop}

As we mentioned above
the quotient category $\QMod_{\cS}(\cA)$
is a Grothendieck category, i.e. it is
an abelian AB5-category with a set of generators.
A well-known theorem of Gabriel and Popescu (see, for example, \cite{Po})
states essentially  that any Grothendieck category $\cC$ is equivalent to a quotient
category of the category $\umod\Lambda$ of right modules
over the endomorphism  ring $\Lambda$ of a  generator $U\in\cC.$
There  are also  Gabriel-Popescu type theorems for a set of generators.

\begin{theo}\label{GP}{\rm (\cite{Me, Igl})} Let $\cC$ be a Grothendieck category. Assume that $\cA$ is a full subcategory of $\cC$
such that $\Ob\cA$ form a set of generators of $\cC.$
Then there exist a localizing subcategory $\cN \subset
\umod\cA$ and an equivalence of categories
$\cC \stackrel{\sim}{\to} \QMod _{\cN}(\cA).$
\end{theo}
Combining this theorem with Proposition \ref{qcat} we obtain

\begin{theo}\label{Grothcat} Let $\cC$ be a Grothendieck category. Assume that it has a set of small generators
which are compact objects in the derived category $D(\cC).$
Then the derived category  $D(\cC)$ has a unique enhancement.
\end{theo}

Now we can apply this result to the category of quasi-coherent sheaves on a quasi-compact and separated scheme.
We say that a quasi-compact and separated scheme $X$ {\it has enough  locally free sheaves},
if  for any finitely presented sheaf $\cF$ there is an epimorphism
$\E\twoheadrightarrow\cF$ with a locally free
sheaf $\E$ of finite type. In this case, the set of all locally free sheaves of finite rank forms a set
of generators of the abelian category of quasi-coherent sheaves $\Qcoh X.$ This follows from a fact that in this case
 every sheaf in $\Qcoh X$
is a filtering colimit of finitely presented $\O_X$\!-modules (see \cite{EGA1} 6.9.12).
Moreover,  any locally free
sheaf  of finite rank is a compact object in $D(\Qcoh X),$ because the functor of global sections commutes with direct sums
for a quasi-compact and separated scheme (see \cite{Ne3} Lemma 1.4, Ex.1.10). It is also need to mention that
for a quasi-compact and separated scheme $X$ the category
$D(\Qcoh X)$ is equivalent to
the category $D(X)_{\Qcoh}$ of complexes of $\O_X$\!-modules with quasi-coherent
cohomology (\cite{BN} Cor.5.5).
Thus, Theorem \ref{Grothcat} implies immediately.

\begin{cor}\label{qcqse}
Let $X$ be a quasi-compact and separated scheme that has enough locally free sheaves. Then
the derived category of quasi-coherent sheaves $D(\Qcoh X)$  has a unique enhancement.
\end{cor}
\begin{remark} {\rm The statement is also true for a quasi-compact and semi-separated scheme in definition of \cite{TT} B.7, because
the proofs of Corollary 5.5 in \cite{BN} and Lemma 1.4 in \cite{Ne3} can be applied for a semi-separated scheme directly. }
\end{remark}

For a quasi-projective scheme by Serre theorem we have a precise description of the category of quasi-coherent sheaves
as a quotient category.
Let $X$ be a quasi-projective scheme. Then it is an open subscheme of a projective scheme $\overline{X}\subset\bbP^n.$
Denote by $A$ the following $\bbZ$\!-graded algebra
$$
A=\bigoplus_n H^0( \overline{X}, \cO_{\overline{X}}(n)).
$$
We can  consider the abelian category  of all graded $A$\!-modules $\Gr(A).$
This category has a Serre subcategory of torsions modules $\Tors(A).$
 Recall that a module $M$ is called torsion if for
any element $x\in M$ one has $x A_{\ge p}=0$ for some $p.$ Denote by
$\QGr(A)$ the quotient category $\Gr(A)/\Tors(A).$

With the graded algebra $A$ one can associate  a
$\bbZ$\!-category $\cA,$ objects of which are $\bbZ$ and morphisms
$\Hom_{\cA}(i, j)=A^{j-i}$ so that the composition in
$\cA$ comes from the multiplication in $A.$
It is clear that the
categories $\Gr(A),$ $\Tors(A),$ and $\QGr(A)$ are equivalent  to
$\umod\cA,$ $\Tors(\cA),$ and $\QMod(\cA)$ respectively.

The well-known Serre theorem gives us an equivalence between the category of quasi-coherent sheaves $\Qcoh\overline{X}$ on $\overline{X}$ and
the quotient category $\QGr(A)=\QMod(\cA)$ (see \cite{Se, EGA2}).
On the other hand, the category of quasi-coherent sheaves $\Qcoh X$ on $X$ is equivalent to the quotient
of the category $\Qcoh\overline{X}$ by the subcategory of quasi-coherent sheaves with support on the complement
$\overline{X}\backslash X$ (\cite{Ga}).
Therefore, the category $\Qcoh X$ is equivalent to the quotient of the category $\Gr(A)=\umod\cA$ by the
localizing subcategory of $I$\!-torsion modules
$\Tors_{I}(A),$ where $I$ is a homogenous ideal such that the support of the subscheme $\Proj A/I\subset \overline{X}$
is exactly $\overline{X}\backslash X.$
More precisely, a graded $A$\!-module $M$ is called $I$\!-torsion if
$$
M\cong\varGamma_I(M)=\colim_n \Hom_{A}(A/I^n, M).
$$

In addition, the free modules $h^Y$ map to the corresponding line bundles $\O(i)|_{X}$ which are compact objects in
$D(\Qcoh X).$

\begin{cor}\label{qcqproj}
The derived category of quasi-coherent sheaves $D(\Qcoh X)$ on a quasi-projective scheme $X$ has a unique enhancement.
\end{cor}

Now let us consider the subcategory of compact objects $D(\Qcoh X)^c.$ It is well-known that the subcategory of compact objects
on any quasi-compact and quasi-separated scheme
coincides with the subcategory of perfect complexes (see \cite{Ne3, BvB}).

\begin{theo}\label{perfc}
The triangulated category of perfect complexes $\Perf(X)$ on a quasi-projective scheme $X$ has a unique enhancement.
\end{theo}
\begin{proof}
To applying Theorem \ref{mainperf} for the subcategory of compact objects $\Perf(X)=D(\Qcoh X)^c$
we need to show that the corresponding localizing subcategory $D_{I}(\Gr(A))$ of all complexes
cohomologies of which are $I$\!-torsion modules
is compactly generated.

The
functor $\varGamma_I: \Gr(A)\to \Tors_I(A)$ has a right-derived functor $\bR\varGamma_I :
D(\Gr(A))\to D(\Tors_I(A))$ via h-injective resolutions, i.e. complexes of modules $\cI$ such that
$\Hom(\cN, \cI)=0$ in the homotopy category of graded $A$\!-modules
for any acyclic complex $\cN$ (see \cite{Sp} Th.C, or \cite{KS} Ch.14 for details).

It is known that the canonical functor $i: D(\Tors_I(A))\to D(\Gr(A))$
is fully faithful and realizes an equivalence of
$D(\Tors_{I}(A))$ with the full subcategory $D_{I}(\Gr(A)).$
(It is proved for noetherian rings, for example, in \cite{L}
Cor 3.2.1.)
To prove this fact it is sufficient to show that for
any $C^{\cdot}\in D_{I}(\Gr(A))$ the natural map
$i\bR\varGamma_I (C^{\cdot}) \to C^{\cdot}$ is an isomorphism.
Since the functor
$\bR\varGamma_I$ is bounded for noetherian schemes (\cite{L}, Cor. 3.1.4), by usual "way
out" argument (\cite{Ha}, \S 7)  it is sufficient to check the
isomorphism $i\bR\varGamma_I (M) \to M$ only for $I$\!-torsion modules $M\in
\Tors_I(A).$ But for $I$\!-torsion module $M$ we have $\bR\varGamma_I M\cong
\varGamma_I M,$ because  $\bR\varGamma_I M\cong M\otimes \cK(t_1)\otimes\cdots\cK(t_m),$ where
$(t_1,\dots, t_m)$ is a sequence in $A,$ generating the ideal $I$ and $\cK(t_i):=\{A\to A_{t_i}\}$
(Prop. 3.1.2 \cite{L}).
Now the corollary follows from the lemma

\begin{lemma}\label{perfcg} The subcategory $D_{I}(\Gr(A))\subset D(\Gr(A))$ is compactly generated.
\end{lemma}
\begin{proof}
By definition it is sufficient to show that for any $C^{\cdot}\in D_{I}(\Gr(A))$
there is a perfect complex $P^{\cdot}$ from $D_{I}(\Gr(A))$ and a non-zero morphism
from $P^{\cdot}$ to $C^{\cdot}.$

Since there is an equivalence
$D(\Tors_I(A))\stackrel{\sim}{\to} D_{I}(\Gr(A))$ we can assume that the object
$C^{\cdot}$ is a  complex of $I$\!-torsion modules.
Consider a nontrivial cohomology of $C^{\cdot}.$ Assume for simplicity that it is 0th cohomology $H^{0}(C^{\cdot})\ne 0.$
Let us consider a non trivial map from $A(k)\to H^{0}(C^{\cdot})$ for some $k\in \bbZ$ and
lift it to a nontrivial map $f: A(k)\to Z^0(C^{\cdot})=\Ker d_0.$ Thus the corresponding morphism
$\widetilde{f}: A(k)\to Z^0(C^{\cdot})\to C^{\cdot}$
is non trivial in the derived category,
and the image of $f$ is an $I$\!-torsion module. Let us cover the kernel $\Ker f$ by a free module $T=\bigoplus_{i=1}^{p} A(k_i),$
and consider the Koszul complex
$$
\cK^{\cdot}:=\{ 0\lto\det(T(-k))(k)\lto\cdots\lto(\Lambda^2 T(-k))(k)\lto T\lto A(k)\lto 0.\}
$$
The cohomologies of the Koszul complex are $I$\!-torsion modules, hence $\cK^{\cdot}\in D_{I}(\Gr(A)).$
On the other hand it is perfect. Finally the map $\widetilde{f}$ can be factorized as
$
 A(k)\to  \cK^{\cdot}\to \im f \hookrightarrow Z^{0}(C^{\cdot})\to C^{\cdot}.
$
Since $\widetilde{f}$ is not trivial, the morphism $\cK^{\cdot}\to  C^{\cdot}$ is also non trivial.
Therefore, the subcategory $D_{I}(\Gr(A))$ is compactly generated.
\end{proof}
Thus,  applying Theorem \ref{mainperf} we obtain the required statement.
\end{proof}

\section{Bounded derived categories of coherent sheaves}\label{cohbound}

Let as above $\cA$ be a small category which we consider as DG category.
Let $D(\cA)$ be the derived category of DG $\cA$\!-modules and $L$ be a localizing subcategory
in $D(\cA)$ which is generated by objects compact in $D(\cA).$
Consider again the quotient functor $\pi: D(\cA)\to D(\cA)/L.$ By Neeman's theorem \ref{Neeman}
the objects  $\pi (h^Y)\in D(\cA)/L$ are
compact for all $Y,Z\in \cA.$
As above we assume that the following condition holds
\begin{enumerate}
\item[] for every $Y,Z\in \cA$ we have
$\Hom (\pi (h^Y),\pi (h^Z)[i])=0\quad \text{when}\quad i<0.$
\end{enumerate}

In this section we are going to talk
about bounded derived categories of coherent sheaves. To work with these categories we introduce  a
notion of a triangulated subcategory of bounded and coherent objects for a general triangulated category.

\begin{defi}\label{bcoh} Let $\cT$ be a triangulated category that admits arbitrary direct sums.
Let $S\subset \cT^{c}$ be a set of compact generators of the category $\cT.$
We say that an object $M\in \cT$ is compactly approximated if
\begin{itemize}
\item[a)] there is $m\in \bbZ$ such  that for any $Y\in S$
we have $\Hom ( Y, M[i])=0$  when $i<m;$
\item[b)] for any $k\in \bbZ$ there is a morphism $\varphi_k: P_k\to M$ from a compact object $P_k\in \cT^c$
such that for every  $Y\in S$ the canonical map
$$
\Hom(Y, P_k[i])\lto \Hom(Y, M[i])
$$
is an isomorphism when $i\ge k.$
\end{itemize} We denote by $\cT^{ca}$ the full subcategory of compactly approximated objects.
\end{defi}
\begin{remark}{\rm
This definition depends on a set of compact generators $S.$ In our applications this set is fixed by a construction and gives
a usual bounded derived category of coherent sheaves as we see below (see Proposition \ref{cohcapp}).}
\end{remark}
\begin{remark}\label{Pcomp}{\rm
By Remark \ref{twodef} the set $S$ classically generates the category of compact objects $\cT^c.$
Hence any compact object $P$ is a direct summand of a finite extension of finite direct sums
of objects of the form $Y[i]$ where $Y\in S.$ This implies that if we fix  a
compactly approximated object $M$ and a compact object $P$  then for any $m\in \bbZ$
there is an $s\in\bbZ$ such that for any $k\le s$  the morphism
$$
\Hom(P, P_k[i])\lto \Hom(P, M[i])
$$
induced by $\varphi_k: P_k\to M$
is an isomorphism when $i\ge m.$
}
\end{remark}
This remark implies the following lemma.

\begin{lemma}\label{hocolcoh} The property b) of Definition \ref{bcoh} holds if and only if $M$ is isomorphic to
$\hocolim P_k,$ where $P_0\to P_{-1}\to \cdots$ is a sequence of morphisms of compact objects
such that for every $Y\in S$ the canonical map $\Hom(Y, P_k[i])\stackrel{\sim}{\to} \Hom(Y, P_s[i])$
is an isomorphism when $i\ge k\ge s.$
\end{lemma}
\begin{proof}
$\Leftarrow$ If $M\cong\hocolim P_k$ then the morphisms $\varphi_k: P_k\to M$ satisfy property
b) of Definition \ref{bcoh}, because $\Hom(Y, P_k[i])\cong\Hom(Y, M[i])$ for any $i\ge k.$

$\Rightarrow$ By Remark \ref{Pcomp} for the map $\varphi_0: P_0\to M$ we can find $k_1\ll 0$
and a map $u_0: P_0\to P_{k_1}$ such that $\varphi_{k_1}u_0=\varphi_0.$ Moreover, there are isomorphisms
$$
\Hom(Y, P_0[i])\stackrel{\sim}{\lto}\Hom(Y, P_{k_1}[i])\stackrel{\sim}{\lto}\Hom(Y, M[i])
$$
for any $i\ge 0.$ Taking $\varphi_{k_1}:P_{k_1}\to M$ we can find $k_2\ll k_1$ and a map $u_1: P_{k_1}\to P_{k_2}$
such that $\varphi_{k_2}u_1=\varphi_{k_1}.$ Repeating this procedure we get
 a sequence of morphisms $\{u_i: P_{k_i}\to P_{k_{i+1}}\}_{i\in\bbN}$ that satisfies condition
of the lemma. Now we can take  $M'=\hocolim P_{k_i}.$ The maps $\varphi_{k_i}$ induce a map $\varphi: M'\to M.$
It is easy to see that $\Hom(Y, M'[i])\cong \Hom(Y, M[i])$ under $\varphi$ for any $i\in\bbZ$ and any $Y\in S\subset\cT^{c}.$
Since $S$ is a set of compact generators the map $\varphi$ is an isomorphism.
\end{proof}

\begin{lemma}\label{subsetcomp}
In the notation as above, the  compactly approximated objects $\T^{ca}$ form a triangulated subcategory.
In addition, if all objects $Z\in S$ satisfy property a) of Definition \ref{bcoh} (with $Z=M$) the subcategory $\cT^{ca}$ contains
the subcategory of compact objects $\cT^c.$
\end{lemma}
\begin{proof}
It is evident that any shift $M[i]$ of a compactly approximated object is compactly approximated too.
It is also easy to see that a cone $C(f)$ of a map $f:M\to N$ of two compactly approximated objects satisfies
property a) of Definition \ref{bcoh}.

Now we need to show that
the cone $C(f)$ satisfies property b) of this definition.
Let us fix $k\in\bbZ$ and consider a map $\varphi_k: P_k\to M$ from Definition \ref{bcoh}.
Now take a sufficiently negative $m\ll k$ and
a morphism $\psi_m: Q_m\to N$ from a compact object $Q_m\in \cT^c$
as in Definition \ref{bcoh} such that
the canonical map
$$
\Hom(P_k, Q_m[i])\lto \Hom(P_k, N[i])
$$
is an isomorphism when $i\ge m.$ It exists by Remark \ref{Pcomp}.
The morphism $f: M\to N$ induces a morphism from $P_k$ to $N$ that can be uniquely lifted to a map
$f': P_k\to Q_m.$
Now morphisms $\varphi_k$ and $\psi_m$ induce a map $\chi_k$ from the compact object $C(f')$ to the object $C(f)$
$$
\begin{CD}
P_k @>{f'}>> Q_m @>>> C(f') @>>> P_k[1]\\
@V{\varphi_k}VV   @V{\psi_m}VV  @VV{\chi_k}V  @VV{\varphi_k[1]}V\\
M @>f>> N @>>> C(f) @>>> M[1]
\end{CD}
$$
The 5-Lemma gives us isomorphisms
$$
\Hom(Y, C(f')[i])\stackrel{\sim}{\lto} \Hom(Y, C(f)[i])
$$
for any $Y\in S$ and $i\ge k.$ Hence, the property b) of Definition \ref{bcoh} holds for the object $C(f)$ too.

If a compact object satisfies the property a) then it is compactly approximated because it satisfies the property b) with
$\varphi_k$ be the identity morphism. Therefore, all objects $Y\in S$ belong to $\cT^{ca}.$ Since $S$
classically generates the subcategory
of compact objects,
the subcategory $\cT^{ca}$ contains $\cT^c,$ because the property a) obviously extends to direct summands.
\end{proof}

Let $\cA$ be a small category which we consider as DG category.
Let $D(\cA)$ be the derived category of DG $\cA$\!-modules and $L$ be a localizing subcategory
in $D(\cA),$ generated by compact objects $L^c=D(\cA)^c\cap L.$
Consider the quotient functor $\pi: D(\cA)\to D(\cA)/L.$ By Theorem \ref{Neeman}
the objects  $\pi (h^Y)\in D(\cA)/L$ are
compact for all $Y,Z\in \cA.$ We assume that the following condition holds
\begin{enumerate}
\item[] for every $Y,Z\in \cA$ we have
$\Hom (\pi (h^Y),\pi (h^Z)[i])=0\quad \text{when}\quad i<0.$
\end{enumerate}

The main aim of this section is to show that  the triangulated subcategory of compactly approximated objects $(D(\cA)/L)^{ca}$
with respect to
the set of compact generators $\pi(h^Y),\; Y\in\cA$ also has a unique enhancement.

Let $\cD$ be an enhancement of this category, i.e. $\D$ is a pretriangulated DG category
and $\bar\epsilon :(D(\cA)/L)^{ca}\stackrel{\sim}{\to} \Ho(\cD)$ be an equivalence of triangulated
categories.

Denote by $\cC
\subseteq \cD$ the full DG subcategory which consists of all  objects that belongs to $\bar\epsilon((D(\cA)/L)^c).$
(By Lemma \ref{subsetcomp} $(D(\cA)/L)^c\subseteq (D(\cA)/L)^{ca}.$)
The equivalence $\bar\epsilon$ induces an  equivalence $\epsilon: (D(\cA)/L)^c\stackrel{\sim}{\to}\Ho(\cC).$
Thus we can apply the construction from Section \ref{cons}.

As in Section \ref{cons} denote by $\cB\subset\cC\subseteq\cD$ the full DG subcategory with the set of objects
$\{\epsilon \pi (h^Y)\}_{Y\in \cA}.$
By the construction in Section \ref{cons} formulas (\ref{rho}) and (\ref{F1}) give a quasi-functor
$$
\widetilde\rho: \SF(\cA)\lto\SF(\cB)
$$
and the induced functor $F_1=\Ho(\widetilde\rho): D(\cA)\lto D(\cB).$

By Lemma \ref{factor} the quasi-functor $\widetilde\rho:\SF(\cA)\to \SF(\cB)$ can be factored through the DG quotient
$\SF(\cA)/\cL\cap\SF(\cA).$ And we get a quasi-functor
$$
\rho: \SF(\cA)/\cL \cap \SF(\cA) \lto \SF(\cB)
$$
such that
the induced functor $\Ho(\rho): D(\cA)/L\lto D(\cB)$ is an equivalence by Lemma \ref{equivp}.

Denote by $\cW\subset\SF(\cA)/\cL\cap\SF(\cA)$ the full DG subcategory which consists
of all objects that are compactly approximated in  $D(\cA)/L,$ i.e that belong to $(D(\cA)/L)^{ca}.$
Denote by $\cW'\subset\SF(\cB)$ the essential image of $\cW$ under the quasi-functor $\rho.$
It is evident that the DG subcategory $\cW'\subset \SF(\cB)$ consists of all compactly approximated objects
$D(\cB)^{ca}$ (compactly approximated with respect to $\cB$). As a consequence, we obtain
a quasi-equivalence
\begin{equation}\label{comaprho}
\rho^{ca}: \cW\to \cW'
\end{equation}
between DG categories $\cW$ and $\cW'$ that are natural enhancements of
$(D(\cA)/L)^{ca}$ and $D(\cB)^{ca}$ respectively.

Now we consider the canonical DG functor
$\varPhi :\cD \to \Mod\cB$ defined by the rule
$$
\varPhi(X)(B)=\Hom_{\cD}(B, X), \quad\text{where}\quad B\in\cB, X\in\cD.
$$
In composition with DG quotient functor $\Mod\cB\to \Mod\cB/\Ac(\cB)$
it induces a quasi-functor
$$
\phi :\cD\lto\SF(\cB).
$$
It remains to show that $\phi$ induces a quasi-equivalence between $\cD$ and $\cW'.$
By Proposition \ref{Keller1} the functor $H^0(\phi)$ realizes an equivalence between
subcategories of compact objects $H^0(\cC)\cong (D(\cA)/L)^c$ and $D(\cB)^c.$
Moreover, by construction, we have isomorphisms
\begin{equation}\label{newis}
\Hom_{H^0(\cD)}(B, X)\stackrel{\sim}{\lto} \Hom_{D(\cB)}(H^0(\phi)B, H^0(\phi)X)
\end{equation}
for any $B\in\cB$ and $X\in\cD.$

Let us denote by $Q$ the composition of functors
\begin{equation}
Q: (D(\cA)/L)^{ca}\stackrel{\bar\epsilon}{\lto}H^0(\cD)\stackrel{H^0(\phi)}{\lto} D(\cB)\stackrel{H^{0}(\rho)^{-1}}\lto D(\cA)/L.
\end{equation}

Now we prove the following lemma.
\begin{lemma}\label{fufa} The functor $Q$ induces a functor $
Q' : (D(\cA)/L)^{ca}\to (D(\cA)/L)^{ca},
$ which is fully faithful.
\end{lemma}
\begin{proof}
As we know that $Q(\pi(h^Y))\cong \pi(h^Y)$ for each $Y\in\cA$ and $Q$ gives an autoequivalence of the subcategory
of compact objects $(D(\cA)/L)^c.$ Moreover, the isomorphism (\ref{newis}) gives an isomorphism
$$
\Hom(\pi(h^Y), X)\stackrel{\sim}{\lto} \Hom(Q(\pi(h^Y)), Q(X))
$$
for every $Y\in\cA$ and every $X\in (D(\cA)/L)^{ca}.$
Since the subcategory $(D(\cA)/L)^c$ is the smallest triangulated subcategory which contains $\pi(h^Y)$
and which is closed under direct summands
we have the same isomorphism
$$
\Hom(P, X)\stackrel{\sim}{\lto} \Hom(Q(P), Q(X))
$$
for any compact object $P\in (D(\cA)/L)^c.$ This immediately implies that any object $Q(X)$ is compactly approximated
for each $X\in (D(\cA)/L)^{ca}.$ Hence, we obtain a functor
$$
Q': (D(\cA)/L)^{ca}\to (D(\cA)/L)^{ca}.
$$

Let us show that $Q$ (and consequently $Q'$) is fully faithful.
Let $X$ and $X'$ be two objects of $(D(\cA)/L)^{ca}.$ Let $m$ be an integer such that
$
\Hom(\pi(h^Y), X'[i])=0,
$
for all $Y\in\cA$ when $i < m.$

Fix  some $k\ll m$ and
consider a map $\varphi_k: P_k\to X$ as in Definition \ref{bcoh} from a compact object $P_k$
such that for every $\pi(h^Y)$ the canonical map
$$
\Hom(\pi(h^Y), P_k[i])\lto \Hom(\pi(h^Y), X[i])
$$
is an isomorphism when $i\ge k.$
Denote by $C_k$ a cone of $\varphi_k.$
Consider the right adjoint to $\pi$ functor $\mu: D(\cA)/L\to D(\cA).$ For any object $M\in D(\cA)/L$ we have
$$
\Hom( h^Y, \mu(M)[i])\cong \Hom(\pi(h^Y), M[i]).
$$
Hence the cohomologies $H^i(\mu(X'))$ and $H^i(\mu Q(X'))$ are trivial when $i< m.$

By the same reason the cohomologies $H^j(\mu(C_k))$ and $H^j(\mu Q(C_k))$ are trivial when $j\ge k.$
We know that the functor $\mu$ is fully faithful. Since $k\ll m$ we obtain
\begin{equation}\label{zero}
\Hom(C_k, X')=\Hom(\mu(C_k), \mu(X'))=0,\quad \Hom(C_k[-1], X')=\Hom(\mu(C_k)[-1], \mu(X'))=0.
\end{equation}
This implies that the canonical map
$$
\Hom(X, X')\lto\Hom(P_k, X')
$$
is an isomorphism. For $Q(C_k)$ and $Q(X')$ there is a similar vanishing as in (\ref{zero}) and we get an isomorphism
$$
\Hom(Q(X), Q(X'))\stackrel{\sim}{\lto}\Hom(Q(P_k), Q(X')).
$$
Thus we have a commutative diagram
$$
\begin{CD}
\Hom (X, X' ) @> Q>>   \Hom(Q(X), Q(X'))\\
@VV{\wr}V  @VV{\wr}V \\
\Hom( P_k, X) @>{\sim}>> \Hom(Q(P_k), Q(X'))
\end{CD}
$$
where three arrows are isomorphisms. Hence the upper arrow is also an isomorphism.
This implies that the functor $Q$ is fully faithful.
\end{proof}

Finally, we have to show that the corresponding functor
$Q'$ is essentially surjective.

\begin{lemma}\label{essurj} In the notation as above
the functor
$Q': (D(\cA)/L)^{ca}\to (D(\cA)/L)^{ca}$ is essentially surjective.
\end{lemma}
\begin{proof}
Let $X\in (D(\cA)/L)^{ca}$ be a compactly approximated object. Suppose that it is bounded by $m\in\bbZ$ as in Definition \ref{bcoh} a) and
$X$ is isomorphic to $\hocolim P_k,$ where
$$
P_0\stackrel{u_0}{\lto} P_{-1}\stackrel{u_1}{\lto}\cdots\lto P_{-n}\stackrel{u_n}{\lto} P_{-n-1}\lto\cdots
$$
is a sequence of morphisms of compact objects
such that for every $Y\in S$ the canonical map $\Hom(Y, P_k[i])\stackrel{\sim}{\to} \Hom(Y, P_s[i])$
is an isomorphism when $i\ge k\ge s$ as in Lemma \ref{hocolcoh}.

The functor $Q'$ induces an autoequivalence on the subcategory of compact objects. Hence there are a sequence of compact object
$\{u'_i: P'_{-i}\to P'_{-i-1}\}_{i\in\bbN}$ and isomorphisms
$t_i:P_{-i}\stackrel{\sim}{\to}Q'(P'_{-i})$ such that $Q'(u'_{i})t_i= t_{i+1}u_i.$
Consider an object $X'\cong \hocolim P'_k$ and denote by $\varphi'_k: P_k\to X'$ the respective morphisms.
 The object $X'$ is bounded by the same  $m\in\bbZ$ as $X.$
By Lemma \ref{hocolcoh} it is compactly approximated.
Now we have to prove that $Q'(X')\cong X.$

Since $X\cong\hocolim P_k\cong\hocolim Q'(P'_k)$ we have a morphism $t_X: X\to Q'(X')$
which is induced by $Q'(\varphi'_k).$
For any $k\ll 0$ and $i\ge k$ there is a commutative diagram
$$
\begin{CD}
\Hom(\pi(h^Y), P_k[i]) @>{\sim}>>\Hom(\pi(h^Y), X[i]) \\
@VV{\wr}V  @VVV  \\
\Hom(Q'\pi(h^Y), Q'(P'_k)[i]) @>{\sim}>>\Hom(Q'\pi(h^Y), Q'(X')[i]).
\end{CD}
$$
This implies that the right vertical arrow is an isomorphism for any $i.$ Hence the morphism $t_X: X\to Q'(X')$ is an isomorphism
since the objects $\pi(h^Y)$ form a set of compact generators.
\end{proof}

 Thus  we obtain the following theorem.

\begin{theo}\label{unca}
Let $\cA$ be a small category which we consider as a DG category
 and $L\subset D(\cA)$ be a
localizing subcategory, which is generated by compact objects $L^c=D(\cA)^c\cap L.$
Assume that for the quotient functor $\pi :D(\cA)\to
D(\cA)/L$ the following condition holds
\begin{enumerate}
\item[] for every $Y,Z\in \cA$ we have
$\Hom (\pi (h^Y),\pi (h^Z)[i])=0\quad \text{when}\quad i<0.$
\end{enumerate}
Then the category of compactly approximated objects $(D(\cA)/L)^{ca}$ has a unique
enhancement.
\end{theo}
\begin{proof}
Since the image of the functor $Q$ belongs to the subcategory of compactly approximated objects  the quasi-functor
$\phi:\cD \to \SF(\cB)$ induces a quasi-functor
$
\phi^{ca}:\cD\lto \cW',
$
 where  $\cW'\subset\SF(\cB)$ as above is the DG subcategory which consists
of all compactly approximated object $D(\cB)^{ca}.$
By Lemmas \ref{fufa} and \ref{essurj} the functor $Q'$ is an equivalence. Hence the quasi-functor $\phi^{ca}$
is a quasi-equivalence as well.
On the other hand, we showed above that there is a quasi-equivalence
$\rho^{ca}: \cW\to \cW',$ where $\cW\subset\SF(\cA)/\cL\cap\SF(\cA)$ is the full DG subcategory, which consists of all
compactly approximated  objects in $D(\cA)/L.$ The composition
\begin{equation}\label{quasf}
(\phi^{ca})^{-1}\cdot\rho^{ca}: \cW\to \cD
\end{equation}
gives a quasi-equivalence between these two different enhancements of $(D(\cA)/L)^{ca}.$
\end{proof}

Let $X$ be a noetherian scheme. As above we say that
$X$ {\it has enough  locally free sheaves}
if for any coherent sheaf $\cF$
there is a locally free sheaf of finite type $\E$ and an epimorphism $\E\twoheadrightarrow\cF.$
For example, any quasi-projective scheme satisfies these conditions.

\begin{prop}\label{cohcapp}
Let $X$ be a noetherian scheme that has enough locally free sheaves. Let $S=\{\E_i\}_{i\in I}$
be a set of locally free sheaves of finite type such that for any coherent sheaf $\cF$
there is an epimorphism from a finite direct sum $\bigoplus_{j=1}^{n}\E_{i_j}$ to $\cF.$
Then the set $S$ is a set of compact generators of $D(\Qcoh X)$ and
an object $M\in D(\Qcoh X)$ is compactly approximated (with respect to $S$) if and only if it is a cohomologically
bounded complex with coherent cohomologies.
Thus the triangulated subcategory of compactly approximated objects $D(\Qcoh X)^{ca}$ is equivalent
to  the bounded category of coherent sheaves $D^b(\coh X)\cong D^b(\Qcoh X)_{\coh}.$
\end{prop}

First of all, it is known that for any noetherian scheme $X$ the natural functor from the bounded derived
category of coherent sheaves $D^b(\coh X)$ to $D(\Qcoh X)$ is fully faithful and
establishes an equivalence of $D^b(\coh X)$ with the subcategory $D^b(\Qcoh X)_{\coh}$
of cohomologically bounded complexes with coherent cohomologies.
To prove Proposition \ref{cohcapp} we will need the following lemmas.

\begin{lemma}\label{fd}{\rm (\cite{TT}, B.11, B.8)}
Let $X$ be a noetherian scheme. Then there is an integer $N$ such that for all $k\ge N$
 and all quasi-coherent sheaves $\cF$ we have
 $\Ext^k(\E, \cF)=0,$ where $\E$ is a locally free sheaf of finite type.
\end{lemma}

\begin{lemma}{\rm(\cite{TT}) 2.3.1 e), 2.2.8)}\label{cov}
Let $X$ be a scheme as in Proposition \ref{cohcapp}. Then for any  $M\in D^b(\Qcoh X)_{\coh}$
there is a bounded above complex of locally free sheaves of finite
type $P^{\cdot}$ with a quasi-isomorphism
$P^{\cdot}\stackrel{\sim}{\to}M$.
\end{lemma}

\begin{lemma}\label{nezero}
Let a scheme $X$ and a set of locally free sheaves of finite type $S=\{\E_i\}_{i\in I}$ be as
in Proposition \ref{cohcapp}. Let $M\in D(\Qcoh X)$ be a complex of quasi-coherent sheaves
and  $H^j(M)\ne 0$ for some $j\in\bbZ$ then there is an sheaf $\E_i\in S$ such that
$\Hom(\E_i, M[j])\ne 0.$
\end{lemma}
\begin{proof}Let us consider the stupid truncation $\sigma_{\ge j}M.$ We have an epimorphism
$H^j(\sigma_{\ge j}M)\twoheadrightarrow H^j(M).$ Any quasi-coherent sheaf on noetherian scheme
is a direct limit of its coherent subsheaves.
Therefore we can find a coherent subsheaf $\cF\subset H^j(\sigma_{\ge j}M)$ such that the compositon map to
$H^j(M)$ is nontrivial. By assumption any coherent sheaf can be covered by a direct sum of $\E_i\in S.$
Hence, we can find a morphism $\E_i\to H^j(\sigma_{\ge j}M)$ such that the composition with the map
to $H^j(M)$ is nontrivial. This map induces a map
$$
\E_i\to H^j(\sigma_{\ge j}M)\to \sigma_{\ge j}M[j]\to M[j]
$$
which is nontrivial, because it is nontrivial on the cohomologies.
\end{proof}
\noindent
{\it Proof of Proposition \ref{cohcapp}}
$\Leftarrow$
Any (cohomologically) bounded complex satisfies the property a) of Definition \ref{bcoh}.
By Lemma \ref{cov} for any $M\in D^b(\Qcoh X)_{\coh}$ there is a
bounded above resolution of locally free sheaves of finite type $P^{\cdot}\stackrel{\sim}{\to} M.$
To construct an approximation $\varphi_k: P_k\to M$ we can consider a stupid truncation
$\sigma_{\ge l}P^{\cdot}$ for $l\ll (k-N),$ where $N$ is an integer from Lemma \ref{fd}.
The object $\sigma_{\ge l}P^{\cdot}$ is a perfect complex and the canonical map $\varphi_k: \sigma_{\ge l}P^{\cdot}\to M$
satisfies the property b). Indeed, the cone $C_k$ of this map $\varphi_k$ is a cohomologically bounded complex such that
$H^j(C_k)$
are trivial when $j>l.$ Lemma \ref{fd} implies that $\Hom(\E_i, C_k[m])=0$ for all $\E_i\in S$ when $m\ge k.$
Therefore, any bounded complex of coherent sheaves is compactly approximated.

$\Rightarrow$
Let $M\in D(\Qcoh X)$ be a compactly approximated object. By Lemma \ref{nezero} the property a) implies that
$M$ is cohomologically bounded below. On the other hand, the property b) gives us that $M$ is cohomologically bounded above.
Indeed, by Lemma \ref{nezero} a cone of a map $\varphi_k: P_k\to M$ is bounded above and the object $P_k$
is cohomologically bounded as perfect complex.
Moreover, Lemma \ref{nezero} implies that all cohomologies $H^m(C_k)$ of the cone $C_k$ of the map $\varphi_k$ are trivial
when $m> k.$ Hence, for sufficiently negative $k$ we have that nontrivial cohomologies of $M$ are isomorphic to the cohomologies
of $P_k.$ Therefore, they are coherent sheaves and $M$ is cohomologically bounded complex with coherent cohomologies.
\hfill$\Box$

\begin{theo}\label{bdcmain}
The bounded derived  category of coherent sheaves $D^b(\coh X)$ on a quasi-projective scheme $X$ has a unique enhancement.
\end{theo}
\begin{proof}
Let $X$ be a quasi-projective scheme. Then it is an open subscheme of a projective scheme $\overline{X}\subset\bbP^n.$
Denote by $A$ the following $\bbZ$\!-graded algebra
$$
A=\bigoplus_n H^0( \overline{X}, \cO_{\overline{X}}(n)).
$$

With the graded algebra $A$ one can associate  a
$\bbZ$\!-category $\cA,$ objects of which are $\bbZ$ and
$\Hom_{\cA}(i, j)=A^{j-i}$ so that the composition in
$\cA$ comes from the multiplication in $A.$

It was shown in the previous Section \ref{geom} that the category $\Qcoh X$ is equivalent to a quotient of the category
$\Gr(A)=\umod\cA$ by the
localizing subcategory of $I$\!-torsion sheaves
$\Tors_{I}(\cA),$ where $I$ is a homogenous ideal such that the support of the subscheme $\Proj A/I\subset \overline{X}$
is exactly $\overline{X}\backslash X.$ Moreover, by Lemma \ref{quotab} we also know that the derived category
$D(\Qcoh X)$ is equivalent to the quotient of $D(\cA)$ by  the localizing subcategoy $D_{I}(\cA)\subset D(\cA)$ that
consists of all objects cohomologies of which belongs to $\Tors_{I}(\cA).$
In addition, the free modules $h^Y$ map to the corresponding line bundles $\O(i)|_{X}$ which are compact objects in
$D(\Qcoh X).$

By Lemma \ref{perfcg}
the subcategory $D_{I}(\cA)$ is compactly generated. Hence,
Theorem \ref{unca} implies that the bounded derived category of coherent sheaves $D^b(\coh X)$ which is equivalent to the subcategory
of compactly approximated objects in $D(\Qcoh X)$ has a unique enhancement.
\end{proof}

\section{Strong uniqueness and fully faithful functors}

In this section we prove a strong uniqueness for bounded derived categories of coherent sheaves and categories of perfect complexes
on projective schemes.
We remind the notion of an ample sequence in abelian category introduced in \cite{Or}.

\begin{defi}\label{defamp} Let  ${\fA}$ be a $k$\!-linear  abelian category.
Let $\{ P_i \}_{i\in\bbZ}$ be a sequence of objects of ${\fA}.$
We say that this sequence is  ample if for every object $C\in {\fA}$
there exists $N$  such that for all $i<N$ the following conditions hold:

\begin{enumerate}
\item[a)] there is an epimorphism $P_i^{\oplus n_i} \twoheadrightarrow C$ for some $n_i\in\bbN;$
\item[b)] $\Ext^j (P_i, C)=0$ for any $j\ne 0;$
\item[c)] $\Hom( C, P_i)=0$.
\end{enumerate}
\end{defi}

\begin{prop}\label{ampcoh}
Let $X\subseteq\bbP^N$ be a projective scheme such that the maximal torsion
subsheaf $T_0(\O_X)\subset \O_X$ of dimension $0$ is trivial.  Then the
 sequence $\{ \O_X(i) \}_{i\in\bbZ}$ is ample in the abelian
 category of coherent sheaves $\coh X.$
\end{prop}
\begin{proof}
It is a classical result of Serre \cite{Se} that  for any coherent sheaf $\cG$ on a projective scheme $X$
and for sufficiently large $m\gg 0$ the sheaf $\cG(m)$ is generated by a finite number of global sections and
$H^j(X, \cG(m))=0$ for $j>0.$ This implies a) and b) of Definition \ref{defamp}.

Since any coherent sheaf $\cG$ on $X$ is covered by a $\O(k)^{\oplus n_k}$ for some $k\in\bbZ,$
it is sufficient to show that $H^0(X, \O_X(m))=$ for $m\ll 0.$ But for $m\ll 0$
$$
H^0(X, \O_X(m))=H^0(\bbP^N, \O_X(m))=H^0(\bbP^N, \EXT^N_{\bbP^N}(\O_X, \omega_{\bbP^N})(-m))=0,
$$
because by local duality the sheaf $\EXT^N_{\bbP^N}(\O_X, \omega_{\bbP^N})$ is trivial when $T_0(\O_X)$ is trivial.
\end{proof}

Let ${\fA}$ be an abelian category with an ample sequence
$\{ P_i\}.$ Denote by $D^b(\fA)$  the bounded derived category of ${\fA}.$ Let us
consider the full subcategory $j: {\cP}\hookrightarrow D^b(\fA)$ such that
$\Ob {\cP}:=\{  P_i\; |\; i\in \bbZ \}$.
The following proposition is proved in \cite{Or, Or2}
\begin{prop}\label{ext}
Let $F : D^b(\fA)\stackrel{\sim}{\to} D^b(\fA)$ be an autoequivalence.
Suppose there exists an isomorphism of functors
$\theta_{\cP} : j\stackrel{\sim}{\to}F\cdot j,$ where $j$  is  the natural embedding
of ${\cP}:=\{  P_i\; |\; i\in \bbZ \}$ to $D^b(\fA).$
 Then it can be extended to an isomorphism
$\id\stackrel{\sim}{\to}F$ on the whole $D^b(\fA)$.
\end{prop}

We can extend this proposition to the case of exact categories. Let
$\fE$ be an exact category. Assume that it is a full exact
subcategory of an abelian category $\fA,$ i.e. $\fE\subset\fA$ is
closed under extensions in $\fA.$ We also assume that an additional
property holds:

\begin{tabular}{ll}
(EPI)&\quad a map $f$ in $\fE$ is an admissible epimorphism if and only if it is an epimorphism in $\fA.$
\end{tabular}

Now we define an ample sequence in $\fE.$

\begin{defi}\label{ampl} Let  ${\fE}$ be a $k$\!-linear  exact category.
Let $\{ P_i \}_{i\in\bbZ}$ be a sequence of objects of $\fE.$
We say that the sequence $\{ P_i \}_{i\in\bbZ}$ is  ample in $\fE$
if there is an exact embedding $\fE\subset \fA$ in an abelian category $\fA$
such that the condition (EPI) holds and $\{ P_i \}_{i\in\bbZ}$ is ample in $\fA.$
\end{defi}

Starting with an exact category $\fE$ we can construct a derived category $D^*(\fE).$
A complex $E^{\cdot}$ is  called acyclic if for any $n$ the differential
$d_n: E_n\to E_{n+1}$ factors as $E_n\stackrel{p_n}{\to} Z_n\stackrel{i_n}{\to} E_{n+1},$
where $p_n$ is a cokernel for $d_{n-1}$ and an admissible epimorphism and
$i_n$ is a kernel for $d_{n+1}$ and an admissible monomorphism.
We define a derived category $D^{*}(\fE)$ as a quotient of homotopy category
$H^*(\fE)$ by the triangulated subcategory of acyclic complexes.

\begin{remark}{\label{exactful}\rm
If $\fE\subset\fA$ is an exact subcategory of an abelian category $\fA$ such that the condition
(EPI) holds and for any $C\in\fA$ there is an epimorphism $E\twoheadrightarrow C$ from $E\in \fE,$
then  the canonical functor $D^{-}(\fE)\to D^{-}(\fA)$ is an equivalence and hence
the functor $D^{b}(\fE)\to D^{b}(\fA)$ is  fully faithful (\cite{Ke2}).
This follows from the fact that for any bounded above complex $C^{\cdot}$ over $\fA$ with cohomologies
from $\fE$ there is a quasi-isomorphism $E^{\cdot}\to C^{\cdot},$ where $E^{\cdot}$ is a bounded above complex over $\fE.$ }
\end{remark}

Proposition \ref{ext} can be generalized  and extended to the case of exact categories.

\begin{prop}\label{exten}
Let $\fE$ be an exact category possessing an ample sequence
$\{ P_i\}.$ Let $F : D^b(\fE)\stackrel{\sim}{\to} D^b(\fE)$ be an autoequivalence.
Suppose there is an isomorphism of functors
$\theta_{\cP} : j\stackrel{\sim}{\to}F\cdot j,$ where $\Ob {\cP}:=\{  P_i\; |\; i\in \bbZ \}$ and
$j : {\cP}\hookrightarrow D^b(\fE)$ is  a natural full embedding.
Then it can be extended to an isomorphism
$\id\stackrel{\sim}{\to}F$ on the whole $D^b(\fE)$.
\end{prop}
\begin{proof}
A proof of this proposition is essentially the same as the proof of Proposition \ref{ext} and it is given
in Appendix \ref{prB}.
\end{proof}
\begin{remark}\label{amperf}{\rm
Our main example of an exact category is the category of locally free sheaves of finite type $\Loc X$ on
a projective scheme $X.$ In this case the exact embedding $\Loc X\subseteq\coh X$ satisfies condition (EPI) and
the bounded derived category $D^b(\Loc X)$ is equivalent to the category
of perfect complexes $\Perf X\subseteq D^b(\coh X).$
If the maximal torsion subsheaf $T_0(\O_X)\subset \O_X$ of dimension 0  is trivial then the
sequence $\{\cO_{X}(i)\}_{i\in \bbZ}$ is an ample sequence in $\Loc X$ in according to Definition
\ref{ampl}.}
\end{remark}

\begin{theo}\label{streq} Let $\fE$ be an exact category with an ample sequence
$\{ P_i\}_{i\in \bbZ}$ and $j:\cP\hookrightarrow D^b(\fE)$
be a full subcategory with $\Ob {\cP}:=\{  P_i\; |\; i\in \bbZ \}.$
Assume that there is an equivalence $u: D^b(\fE)\stackrel{\sim}{\to} (D(\cP)/L)^c$ (or with $(D(\cP)/L)^{ca}$),
where $L\subset D(\cP)$ is a localizing subcategory that is generated by compact objects $L^c=L\cap D(\cP)^c,$ such that
$u\cdot j\cong \pi\cdot h^{\bullet}$ on $\cP.$
Then the category $D^b(\fE)$ has a strongly unique enhancement.
\end{theo}
\begin{proof}
(1)
Let  $\cE$ be a pretriangulated DG category and $N: (D(\cP)/L)^c \stackrel{\sim}{\to} \Ho(\cE)$ be an equivalence.
By Theorem \ref{corperf} there is a quasi-equivalence $\cN:\cV\to\cE,$ where
$\cV\subset\SF(\cP)/\cL\cap\SF(\cP)$ is the full DG subcategory which consists
of all compact objects $(D(\cP)/L)^c.$
Moreover, we know that there is an isomorphisms of functors
$\theta: \Ho(\cN)\cdot\pi\cdot h^{\bullet}\stackrel{\sim}{\to} N\cdot\pi\cdot h^{\bullet}$
 from $\cP$ to $H^{0}(\cE).$

Consider the composition $F=u^{-1}\Ho(\cN)^{-1} N u$ from $D^b(\fE)$ to itself. There is an isomorphism
$j\stackrel{\sim}{\to} F\cdot j$ on the subcategory $\cP.$
Hence we can apply Proposition \ref{exten} and obtain that the functor
$F$ is isomorphic to the identity functor on the whole $D^b(\fE).$ Therefore, the functors $\Ho(\cN)$ and $N$ are isomorphic.
Thus, any equivalence $N: (D(\cP)/L)^c \stackrel{\sim}{\to} \Ho(\cE)$ can be lifted to a quasi-equivalence $\cN,$
and the category $D^b(\fE)\cong (D(\cP)/L)^c$ has strongly unique enhancement.

(2)
Denote by $\cW\subset\SF(\cP)/\cL\cap\SF(\cP)$ the full DG subcategory which consists
of all objects that are compactly approximated in  $D(\cP)/L,$ i.e that belong to $(D(\cP)/L)^{ca}.$

Let  $\cD$ be a pretriangulated DG category and $\bar\epsilon: (D(\cP)/L)^{ca} \stackrel{\sim}{\to} \Ho(\cD)$ be an equivalence.
In the proof of Theorem \ref{unca} we constructed a quasi-equivalence
$(\phi^{ca})^{-1}\cdot\rho^{ca}: \cW\to \cD.$ Moreover, by construction,
the functors $H^0((\phi^{ca})^{-1}\cdot\rho^{ca})\pi h^{\bullet}$ and  $\bar\epsilon \pi h^{\bullet}$ from $\cP$ to $H^0(\cD)$
are isomorphic.
Thus, for the composition $F= u^{-1} H^0((\rho^{ca})^{-1}\cdot \phi^{ca})\cdot\bar\epsilon u$
we have an isomorphism
$j\stackrel{\sim}{\to} F\cdot j$ on the subcategory $\cP.$
Applying Proposition \ref{exten}  we obtain that the functor
$F$ is isomorphic to the identity functor on the whole $D^b(\fE).$
Therefore, the functors $\Ho((\phi^{ca})^{-1}\cdot\rho^{ca})$ and $\bar\epsilon$ are isomorphic.
Thus any equivalence $\bar\epsilon: (D(\cP)/L)^{ca} \stackrel{\sim}{\to} \Ho(\cD)$
can be lifted to a quasi-equivalence.
\end{proof}
This theorem immediately implies the following corollary.
\begin{theo}\label{strun}
Let $X$ be a projective scheme over $k$ such that the maximal torsion
subsheaf $T_0(\O_X)\subset \O_X$ of dimension $0$ is trivial. Then
the triangulated categories $\Perf(X)$ and $D^b(\coh X)$ have  strongly
unique enhancements.
\end{theo}
\begin{proof}
Let $X$ be a projective scheme.
Denote by $A$ the graded algebra
$
A=\bigoplus_n H^0( X, \cO_{X}(n)).
$
To the graded algebra $A$ we can attach a
$\bbZ$\!-category $\cA,$ objects of which are $\bbZ$ and
$\Hom_{\cA}(i, j)=A^{j-i}$ with a natural composition law.

It was explained in Section \ref{geom} that the category $\Qcoh X$ is equivalent to a quotient of the category
$\umod\cA=\Gr(A)$ by the
localizing subcategory of torsion sheaves
$\Tors(\cA).$ Moreover, by Lemma \ref{quotab} we also know that the derived category
$D(\Qcoh X)$ is equivalent to the quotient  $D(\cA)/D_{\tors}(\cA)$ where $D_{\tors}(\cA)\subset D(\cA)$
is a localizing subcategory  that
consists of all objects cohomologies of which belongs to $\Tors(\cA).$
In addition, the free modules $h^Y$ map to the corresponding line bundles $\O_X(i)$ which are compact objects in
$D(\Qcoh X).$
Lemma \ref{perfcg} gives us that
the subcategory $D_{\tors}(\cA)$ is compactly generated.
The category $\Perf(X)=D^b(\Loc X)$ (Remark \ref{amperf})  is equivalent to the category
$(D(\cA)/D_{\tors}(\cA))^c$ and  by Proposition \ref{cohcapp} the bounded derived category of coherent sheaves
$D^b(\coh X)$ is equivalent to the triangulated category of compactly approximated objects $(D(\cA)/D_{\tors}(\cA))^{ca}\cong D(\Qcoh X)^{ca}.$
These equivalences are isomorphic to the identity on the full subcategory
$\cP:=\{\O_X(i)\}_{i\in \bbZ}.$ By Remark \ref{amperf} (resp. Prop. \ref{ampcoh}) the sequence $\cP:=\{\O_X(i)\}_{i\in \bbZ}$ is ample in
$\Loc X$ (resp.$\coh X$) on a projective scheme with $T_0(\O_X)=0.$ Hence we can apply Theorem \ref{streq} and obtain that
$\Perf(X)$ (resp.$D^b(\coh X)$) has strongly unique enhancement.
\end{proof}

Let $X$ be a quasi-compact and quasi-separated scheme. Define a DG enhancement $\cd_{dg}(\Qcoh X)$ of the derived category
$D(\Qcoh X)$
as the quotient $\cC_{dg}(\Qcoh X)/\Ac_{ dg}(\Qcoh X),$ where $\cC_{dg}(\Qcoh X)$ is the DG category of
unbounded complexes over $\Qcoh X$ and $\Ac_{dg}(\Qcoh X)$ is the DG subcategory of unbounded acyclic complexes.
Thus the DG category $\cd_{dg}(\Qcoh X)$ is an enhancement of the derived category $D(\Qcoh X).$
There is a theorem of Bertrand To\"{e}n which says that the functors between DG derived categories
are represented by objects on the product.

\begin{theo}[\cite{To} Th.8.9]\label{Toen} Let $X$ and $Y$ be quasi-compact and
separated schemes over a field $k.$
Then we have a canonical quasi-equivalence
$$
\cd_{dg}(\Qcoh (X\times Y))\stackrel{\sim}{\lto} \RHOM_c(\cd_{dg}(\Qcoh X),\cd_{dg}(\Qcoh Y)),
$$
where $\RHOM_c$ denotes the DG category formed by
the direct sums preserving quasi-functors (we say that a quasi-functor preserves direct sums if its homotopy functor does).
\end{theo}

\begin{remark}{\rm
In other words this theorem tells us that  for any  quasi-functor $\cF$
from $\cd_{dg}(\Qcoh X)$ to $\cd_{dg}(\Qcoh Y)$ that preserves direct sums the functor $H^0(\cF)$ can be represented
by an object on the product $X\times Y,$ i.e. it is isomorphic to a functor of the form
$$
\Phi_{\cE^{\cdot}}(-):=\bR p_{2*}(\cE^{\cdot}\stackrel{\bL}{\otimes}  p^{*}_{1}(-)),
$$ where $\cE^{\cdot}$ is a complex of quasi-coherent sheaves
on the product $X\times Y,$ and this representation is unique up to isomorphism in $D(\Qcoh (X\times Y))$
(see \cite{To} Cor.8.12).}
\end{remark}

Combining Theorem \ref{Toen} with the uniqueness of enhancement for the derived category of quasi-coherent sheaves
on a quasi-compact and quasi-separated scheme with enough locally free sheaves
we obtain the following corollary.
\begin{cor}
Let $X$ and $Y$ be quasi-compact separated schemes over a field $k.$ Assume that $X$ has enough locally free sheaves.
Let
$F: D(\Qcoh X)\to D(\Qcoh Y)$ be a fully faithful functor that commutes with direct sums.
Then there is an object $\cE^{\cdot}\in D(\Qcoh(X\times Y))$ such that
the functor $\Phi_{\cE^{\cdot}}$ is
fully faithful and $\Phi_{\cE^{\cdot}}(C^{\cdot})\cong F(C^{\cdot})$ for any $C^{\cdot}\in D(\Qcoh X).$
\end{cor}
\begin{proof}
We know that by Corollary \ref{qcqse} $D(\Qcoh X)$ has a unique enhancement. Moreover, by Theorem \ref{GP} and Lemma \ref{quotab}
the category $D(\Qcoh X)$ can be represented as a quotient $D(\cA)/L$ where $\cA$ is a small category formed
by a set of locally free sheaves of finite type and $L$ is a localizing subcategory in $D(\cA).$
Thus, the enhancement $\cd_{dg}(\Qcoh X)$ is equivalent to the enhancement  $\Mod\cA/\cL.$
Under the quotient functor all objects of the form $h^{Y}$ goes to locally free sheaves of finite type and, hence, they are
compact in $D(\Qcoh X).$
Now we can apply Theorem \ref{fufabig} and obtain a quasi-functor $\cF$ from
$\cd_{dg}(\Qcoh X)$ to $\cd_{dg}(\Qcoh Y)$ such that $H^0(\cF)$  is fully faithful and
$H^0(\cF)(C^{\cdot})\cong F(C^{\cdot})$ for any $C^{\cdot}\in D(\Qcoh X).$

By construction, the functor
$H^0(\cF)$ realizes an equivalence of $D(\Qcoh X)$ with a subcategory in $D(\Qcoh Y)$
that is the essential image of the functor $F.$ The inclusion of this subcategory in $D(\Qcoh Y)$ commutes with
direct sums by assumption. Therefore, the functor $H^0(\cF)$ commutes with direct sums and by Theorem \ref{Toen}
it has the form $\Phi_{\cE^{\cdot}}$ for some $\cE^{\cdot}\in D(\Qcoh(X\times Y)).$
\end{proof}
Note that under assumptions in the corollary above the functor $F$ has a right adjoint functor by Brown
representability theorem (see \cite{NeBook} Thm. 8.4.4.).

\begin{cor}\label{repr} Let $X$ be a quasi-projective scheme and $Y$ be a quasi-compact and
separated scheme. Let $K: \Perf(X)\to D(\Qcoh Y)$ be a fully faithful
functor. Then there is an object $\cE^{\cdot}\in D(\Qcoh(X\times Y))$ such that
\begin{enumerate}
\item the functor $\Phi_{\cE^{\cdot}}|_{\Perf(X)}:\Perf(X)\to D(\Qcoh Y)$ is fully faithful and $\Phi_{\cE^{\cdot}}(P^{\cdot})\cong K(P^{\cdot})$
for any $P^{\cdot}\in\Perf(X);$
\item if $X$ is projective with $T_0(\O_X)=0,$ then $\Phi_{\cE^{\cdot}}|_{\Perf(X)}\cong K;$
\item if $K$ sends $\Perf(X)$ to $\Perf(Y)$ then
the functor $\Phi_{\cE^{\cdot}}: D(\Qcoh X)\to D(\Qcoh Y)$ is fully faithful and
also sends $\Perf(X)$ to $\Perf(Y);$
\item if $Y$ is noetherian and  $K$ sends $\Perf(X)$ to $D^b(\Qcoh Y)_{\coh},$ then the object $\cE^{\cdot}$ is isomorphic to an object of
$D^b(\coh(X\times Y)).$
\end{enumerate}
\end{cor}
\begin{proof}
Since $X$ is a quasi-projective scheme it is an open subscheme of a projective scheme $\overline{X}\subset\bbP^n.$
We have a graded $A$  algebra
$
A=\bigoplus_n H^0( \overline{X}, \cO_{\overline{X}}(n)).
$
and  a corresponding
$\bbZ$\!-category $\cA,$ objects of which are $\bbZ$ and morphisms
$\Hom_{\cA}(i, j)=A^{j-i}.$

As it was explained in Section \ref{geom} the category $D(\Qcoh X)$ is equivalent to the quotient
$D(\cA)/D_{I}(\cA),$ where  $D_{I}(\cA)$ is a
localizing subcategory of complexes with $I$\!-torsion cohomologies
(here $I$ is an ideal such that the support of the subscheme $\Proj A/I\subset \overline{X}$
is exactly $\overline{X}\backslash X.$)
In addition, the free modules $h^Y$ map to the corresponding line bundles $\O(i)|_{X}$ which are compact objects in
$D(\Qcoh X),$ and $\Perf(X)\cong (D(\cA)/D_{I}(\cA))^c.$

As we know (see Proposition \ref{Keller}) the DG categories $\cd_{dg}(\Qcoh X)$ and $\cd_{dg}(\Qcoh X)$ are
quasi-equivalent to the DG categories $\SF(\prfdg(X))$ and $\SF(\prfdg(Y)),$ respectively, and we denote by $\phi_X$ and $\phi_Y$ the corresponding
 quasi-functors.
For shortness, denote by $\cC$ the DG category $\prfdg(Y)$ and
by $\cC'$ the full DG subcategory in $\SF(\cC)\cong\cd_{dg}(\Qcoh Y)$
which consists of all objects in the essential image
of $H^0(\phi_Y)K.$ The functor $K$ induces  an equivalence
$N:\Perf(X)\stackrel{\sim}{\to}H^0(\cC').$

By Theorem \ref{corperf} there is a quasi-functor $\cN:\prfdg(X)\to\cC'$ which is a quasi-equivalence.
 It induces a quasi-equivalence $\cN^*:\SF(\prfdg(X))\to\SF(\cC').$

Let $\cD$ be a DG subcategory $\cD\subset\SF(\cC)$ that contains $\cC$ and $\cC'.$
We denote by $\cJ:\cC'\hookrightarrow\cD\subset\SF(\cC)$ and $\cI:\cC\hookrightarrow\cD\subset\SF(\cC)$ the respective DG full embeddings.
We have the  extension DG functor $\cJ^*:\SF(\cC')\to\SF(\cD)$ and the restriction DG functor $\cI_{*}:\Mod\cD\to\Mod\cC,$ which
 induces a quasi-functor $\iota_{*}:\SF(\cD)\to\SF(\cC).$
Consider the composition of quasi-functors $\iota_*\cdot\cJ^*\cdot\cN^*:\SF(\prfdg(X))\to \SF(\prfdg(Y)).$ The functors
$H^0(\iota_*), H^0(\cJ^*),$ and $H^0(\cN^*)$
evidently commute with direct sums. (Note that the right adjoint to the quasi-functor $\iota_*\cJ^*$ is the quasi-functor
$\phi:\SF(\cC)\to \SF(\cC')$ which is induced by DG functor
$\varPhi:\SF(\cC)\to \Mod\cC'$ given by the standard formula
$\varPhi(M)(C')=\Hom_{\SF(\cC)}(C', M)$ for $C'\in\cC', \; M\in\SF(\cC).$)

Thus, considering the compositions of quasi-functors $\phi_Y^{-1}\cdot\iota_*\cdot\cJ^*\cdot\cN^*\cdot\phi_X$ we obtain a quasi-functor $\F$ from
$\cd_{dg}(\Qcoh X)$ to $\cd_{dg}(\Qcoh Y)$ that completes the following commutative diagram
$$
\xymatrix{
\cd_{dg}(\Qcoh X) \ar[rrr]^{\F} \ar[d]_{\phi_X}^{\wr} &&&\cd_{dg}(\Qcoh Y) \ar[d]^{\wr}_{\phi_Y}\\
\SF(\prfdg(X)) \ar[r]^-{\cN^*}_-{\sim} &\SF(\cC') \ar[r]^-{\cJ^*}&\SF(\cD)\ar[r]^-{\iota_*}&
\SF(\cC)=\SF(\prfdg(Y))
}
$$
The quasi-functor $\F$ commutes with direct sums and, hence, by Theorem \ref{Toen} the functor $H^0(\cF)$
is isomorphic to $\Phi_{\cE^{\cdot}}$ with $\cE^{\cdot}\in D(\Qcoh(X\times Y)).$

(1) The restriction of the quasi-functor $\iota_{*}\cJ^*$ on $\cC'$ is isomorphic to the inclusion $\cC'\to\SF(\cC).$
This implies that the restriction $\Phi_{\cE^{\cdot}}|_{\Perf(X)}$ is fully faithful.
Moreover, by property (3) of Theorem \ref{corperf}   there is an isomorphism
$\Ho(\cN)(P^{\cdot})\cong N(P^{\cdot})$
 for any $P^{\cdot}\in\Perf(X).$   Hence, $\Phi_{\cE^{\cdot}}(P^{\cdot})\cong K(P^{\cdot}).$

(3) If $K$ sends $\Perf(X)$ to $\Perf(Y),$ then we can take $\cD=\cC.$ In this case $\iota_*$
is the identity, but $H^0(\cJ^*)$ is fully faithful by Proposition \ref{Keller2}. Therefore, $H^0(\F)$ is fully faithful
and sends $\Perf(X)$ to $\Perf(Y).$

(2) Let $X$ be projective. We know that
by property (2) of Theorem \ref{corperf}   there is an isomorphism of functors
$\theta: \Ho(\cN)\cdot\pi\cdot h^{\bullet}\stackrel{\sim}{\to} N\cdot\pi\cdot h^{\bullet}$
 from $\cA$ to $H^{0}(\cC').$

Consider the composition $G=\Ho(\cN)^{-1}\cdot N$ from $\Perf(X)$ to itself. We have a natural transformation
from $j\stackrel{\sim}{\to} G\cdot j$ on the subcategory $h^{\bullet}(\cA)$ which coincides with the full subcategory
$\cP:=\{\O_X(i)\}_{i\in \bbZ}.$ Since $X$ is projective with $T_0(\O_X)=0$ the sequence $\{\O_X(i)\}_{i\in \bbZ}$ is ample in $\Loc X.$
Hence we can apply the Proposition \ref{exten} to $\Perf(X)$ and obtain that the functor
$G$ is isomorphic to the identity functor on the whole $\Perf(X).$ Therefore, the functors $\Ho(\cN)$ and $N$ are isomorphic,
i.e. the functor $K$ is isomorphic to the functor $\Phi_{\cE^{\cdot}}|_{\Perf(X)}.$

(4) Finally, we have to argue that the object $\cE^{\cdot}$ is isomorphic to an object from $D^{b}(\coh (X\times Y))$
if $Y$ is noetherian and the functor $\Phi_{\cE^{\cdot}}$ sends $\Perf(X)$ to $D^b(\Qcoh Y)_{\coh}.$
It is well-known that the canonical functor $D^b(\coh Z)\to D(\Qcoh Z)$ is fully faithful and establishes an equivalence with
the subcategory  $D^b(\Qcoh Z)_{\coh}$ of cohomologically bounded complexes with coherent cohomologies.
Consider an inclusion $i: X\hookrightarrow{\mathbb P}^N.$ The composition of the inverse image functor $\bL i^*$ and $\Phi_{\cE^{\cdot}}$ gives
a functor from $D(\Qcoh {\mathbb P}^{N})$ to $D(\Qcoh Y)$ which is represented by the object $\bR (i, \id_{Y})_{*}\cE^{\cdot}$
and sends $\Perf({\bbP}^{N})$ to $D^b(\Qcoh Y)_{\coh}.$
It is sufficient to check that the object $\bR (i, \id_{Y})_{*}\cE^{\cdot}$ belongs to $D^b(\Qcoh(\bbP^N\times Y))_{\coh}.$
The category $D(\Qcoh(\bbP^N\times Y))$ has a semi-orthogonal decomposition
of the form $\langle \O(-N)\boxtimes~D(\Qcoh Y), \dots, \O\boxtimes D(\Qcoh Y)\rangle.$
The components of $\bR (i, \id_{Y})_{*}\cE^{\cdot}$ with respect to this decomposition are isomorphic to
$\O(-p)\boxtimes\Phi_{\cE^{\cdot}} \bL i^*(\Omega^{p}_{\bbP^N}[p]).$

There is a similar decomposition
$\langle \O(-N)\boxtimes D^b(\coh Y), \dots, \O\boxtimes D^b(\coh Y)\rangle$
for the bounded category of coherent sheaves $D^b(\coh(\bbP^N\times Y)).$ Since the objects
$\Phi_{\cE^{\cdot}} \bL i^*(\Omega^{p}_{\bbP^N}[p])$ belong to $D^b(\Qcoh Y)_{\coh},$ the object
$\bR (i, \id_{Y})_{*}\cE^{\cdot}$ is isomorphic to an object of $D^b(\coh(\bbP^N\times~Y)).$
Therefore, the object $\cE^{\cdot}$ is also isomorphic to an object of $D^b(\coh(X\times Y)).$
\end{proof}

Let us consider the bounded derived category of coherent sheaves $D^b(\coh X)$ on a quasi-projective scheme $X.$
Consider an exact functor $F$ from $D^b(\coh X)$ to a triangulated category $\cT$ that admits arbitrary direct sums.
We say that the functor $F$ {\it commutes with homotopy colimits} if for any $M^{\cdot}\in D^b(\coh X),$
which is isomorphic in $D(\Qcoh X)$ to $\hocolim_i P_i^{\cdot}$ of perfect complexes $P_i^{\cdot},$  there is an isomorphism
 $F(M)\stackrel{\sim}{\to}\hocolim_i F(P_i^{\cdot})$ in
$\cT$ that commutes with canonical morphisms from $F(P_i)$ for each $i.$
Note that our functor $F$ is defined only on $D^b(\coh X)\cong D^b(\Qcoh X)_{\coh}\subset D(\Qcoh X).$

\begin{cor}\label{dbcoh}
Let $X$ be a projective scheme with $T_0(\O_X)=0$ and $Y$ be a quasi-compact and separated scheme. Let
$K: D^b(\coh X)\to D(\Qcoh Y)$ be a fully faithful
functor that commutes with homotopy colimits.
 Then there is an object $\cE^{\cdot}\in D(\Qcoh(X\times Y))$ such that
$\Phi_{\cE^{\cdot}}|_{D^b(\coh X)}\cong K.$
\end{cor}
\begin{proof}
We can consider the restriction of the functor $K$ on the subcategory of perfect complexes
$\Perf(X)\subset D^b(\Qcoh X)_{\coh}\cong D^b(\coh X).$
By Corollary \ref{repr} there is an object $\cE^{\cdot}\in D(\Qcoh (X\times Y))$
and an isomorphism of functors $\theta:\Phi_{\cE^{\cdot}}|_{\Perf(X)}\cong K|_{\Perf(X)}.$
For any object $M^{\cdot}\in D^b(\coh X)$ there is a quasi-isomorphism $P^{\cdot}\stackrel{\sim}{\to} M^{\cdot}$ where
$P^{\cdot}$ is  a bounded above complex of locally free sheaves of finite
type (see Lemma \ref{cov}). Hence, the object $M^{\cdot}$ is isomorphic to
$\hocolim_k \;\sigma_{\ge k}P^{\cdot}$ in $D(\Qcoh X).$
By assumption, there is an isomorphism $\hocolim_k K(\sigma_{\ge k} P^{\cdot})\stackrel{\sim}{\to} K(M^{\cdot})$
in $D(\Qcoh Y).$ On the other hand, the functor $\Phi_{\cE^{\cdot}}$ commutes with direct sums and homotopy colimits. Therefore, there is an isomorphism
$\theta_M: \Phi_{\cE^{\cdot}}(M^{\cdot})\stackrel{\sim}{\to} K(M^{\cdot})$ that makes the following square

$$
\begin{CD}
\Phi_{\cE^{\cdot}}(\sigma_{\ge k}P^{\cdot}) @>>> \Phi_{\cE^{\cdot}}(M^{\cdot})\\
@V\theta_{\ge k}V{\wr}V @V\theta_{M}V{\wr}V \\
K(\sigma_{\ge k} P^{\cdot}) @>>> K(M^{\cdot}).
\end{CD}
$$
commutative for any $k.$

Now we have to check that the restriction of $\Phi_{\cE^{\cdot}}$ on $D^b(\coh X)$ is fully faithful.
Let $Q$ belong to $\Perf(X).$ Any morphism $f:Q\to M^{\cdot}$ factors through $\sigma_{\ge k} P^{\cdot}$ for $k\ll 0.$
Since the right and left squares in the following diagram

$$
\begin{CD}
\Phi_{\cE^{\cdot}}(Q)@>>>\Phi_{\cE^{\cdot}}(\sigma_{\ge k}P^{\cdot}) @>>> \Phi_{\cE^{\cdot}}(M^{\cdot})\\
@V{\theta_{Q}}V{\wr}V @V\theta_{\ge k}V{\wr}V @V\theta_{M}V{\wr}V \\
K(Q)@>>> K(\sigma_{\ge k} P^{\cdot}) @>>> K(M^{\cdot}).
\end{CD}
$$
commute, the outside square is also commutes. Therefore, $\theta_M\cdot\Phi_{\cE^{\cdot}}(f)= K(f)\cdot\theta_{Q}$
for any $f:Q\to M^{\cdot}.$
If now $\Phi_{\cE^{\cdot}}(f)=0,$ then $K(f)=0$ and $f=0,$ because $K$ is fully faithful.
On the other hand, for any $g:\Phi_{\cE^{\cdot}}(Q)\to \Phi_{\cE^{\cdot}}(M^{\cdot})$  we can consider
$g'=\theta_{M}\cdot g\cdot\theta_{Q}^{-1}$ and take $f:Q\to M^{\cdot}$ such that $K(f)=g'.$
It is obvious now that $\Phi_{\cE^{\cdot}}(f)=g.$ Thus, there are isomorphisms
$$
\Hom(Q, M^{\cdot})\stackrel{\sim}{\lto}\Hom(\Phi_{\cE^{\cdot}}(Q), \Phi_{\cE^{\cdot}}(M^{\cdot}))
$$
for any $Q\in \Perf(X)$ and every $M^{\cdot}\in D^b(\coh X).$

The objects $C^{\cdot}\in D(\Qcoh X)$ for which the natural map
$$
\Hom(C^{\cdot}, M^{\cdot}[m])\stackrel{\sim}{\lto}\Hom(\Phi_{\cE^{\cdot}}(C^{\cdot}), \Phi_{\cE^{\cdot}}(M^{\cdot})[m])
$$
is bijective for all $m\in\bbZ$ form a triangulated subcategory in $D(\Qcoh X)$ which contains $\Perf(X)$ and is closed under direct sums
(since $\Phi_{\cE^{\cdot}}$ commutes with direct sums). Therefore, this subcategory coincides with the whole
$D(\Qcoh X).$ Thus, the functor $\Phi_{\cE^{\cdot}}|_{D^b(\coh X)}$ is fully faithful and has the same essential image as $K.$
Now we can apply Proposition \ref{exten} and extend the isomorphism
$\theta:\Phi_{\cE^{\cdot}}|_{\Perf(X)}\cong K|_{\Perf(X)}$ to an isomorphism of the functors
$\Phi_{\cE^{\cdot}}|_{D^b(\coh X)}\cong K.$
\end{proof}
\begin{remark}{\rm
Notice that in the proof of the corollary we used only the fact that $K$ commutes with homotopy colimits
of a special form $\hocolim \sigma_{\ge k} P$ and this implies that
$K\cong\Phi_{\cE^{\cdot}}|_{D^b(\coh X)}.$ Thus, as a result, it commutes with all homotopy colimits.
}
\end{remark}
\begin{example}\label{exlast}{\rm
We say that the functor $K: D^b(\coh X)\to D(\Qcoh Y)$ is bounded
above by $n\in \bbZ$ if $H^i(K(\cF))=0$ for all coherent sheaves $\cF\in\coh X$ when $i>n.$
If the functor $K$ is bounded above then it commutes with homotopy colimits. Indeed,
for any $M^{\cdot}\in D^b(\coh X)$ and a locally free resolution $P^{\cdot}\to M^{\cdot}$
we have that the map $K(\sigma_{\ge k} P^{\cdot})\to K(M^{\cdot}),$ when $k\ll 0,$
is an isomorphism
on cohomologies $H^i$ for all $i>k+n.$  Therefore, $\hocolim_k K(\sigma_{\ge k} P^{\cdot})\cong K(M^{\cdot}).$
By remark above $K$ commutes with homotopy colimits.
}
\end{example}
\begin{cor}\label{lastcor} Let $X$ be a projective scheme with $T_0(\O_X)=0$ and $Y$ be a noetherian scheme. Let
$K:D^b(\coh X)\to D^b(\coh Y)$ be a fully faithful functor that has
a right adjoint $K^{!}:D^b(\coh Y)\to D^b(\coh X).$ Then
there is an object $\cE^{\cdot}\in D^b(\coh(X\times Y))$ such that
$\Phi_{\cE^{\cdot}}|_{D^b(\coh X)}\cong K.$
\end{cor}
\begin{proof}
To apply Corollary \ref{dbcoh} we need to check that the functor $K$ commutes with homotopy colimits.
Following the example above it is sufficient to show that the functor $K$ is bounded above.
First, we easily  see that $K$ is bounded above on the sequence of line bundles $\{\cO(j)\}_{j\in \bbZ},$
because by Beilinson's theorem for any $j>0$ (resp. $j<-N$) we have a left (resp. right) resolution of $\cO_X(j)$ of the form
$$
V_N\otimes \cO_X(-N)\lto\cdots\lto V_1\otimes \cO_X(-1)\lto V_0\otimes \cO_X
$$
where $X\hookrightarrow\bbP^N$ is a closed embedding and $V_s$ is $H^0(\bbP^N, \Omega^s(s+j))$
(resp. $H^N(\bbP^N, \Omega^s(s+j))$) (see, for example, \cite{OSS} Th.3.1.4).

 Denote by $n$ an integer such that
$H^i(K(\cO(j))=0$ for all $j$ when $i>n.$
If  $\cF\in\coh X$ is a coherent sheaf that has a highest nontrivial cohomology $H^m(K(\cF))=\cG\ne 0$ with $m>n$ then
we can construct a sequence of coherent sheaves $\{\cF_p\in\coh X\}_{p\ge 0}$ such that
$H^{m+p}(K(\cF_p))=\cG.$ The construction goes by induction with base $\cF_0=\cF$ and $\cF_{p+1}$ is the kernel
of a some epimorphism from $\cO(-k_p)^{\oplus n_p}$ to $\cF_p.$ It is evident that such epimorphism exists
when $k_p\ll 0$ and the highest cohomology of $K(\cF_{p+1})$ has a number $m+p+1$ and it is isomorphic to
the $H^{m+p}(K(\cF_p))\cong \cG.$

Assume that $K$ has a right adjoint functor $K^!:D^b(\coh Y)\to D^b(\coh X).$ Take the object $K^!(\cG).$
By construction above we have that
$$
\Hom(\cF_p, K^!(\cG)[-m-p])\cong\Hom(K(\cF_p), \cG[-m-p])\ne 0
$$
for all $p\ge 0.$ This gives a contradiction with the fact that $K^!(\cG)$ is a bounded complex.
Therefore, $K$ is bounded above and by Example \ref{exlast} and Corollary \ref{dbcoh} it is isomorphic to
the functor $\Phi_{\cE^{\cdot}}$ for some $\cE^{\cdot}\in D(\Qcoh (X\times Y)).$
Since the functor $K$ sends $\Perf(X)$ to
$D^b(\coh Y),$ Corollary \ref{repr} (4) implies that $\cE$ is isomorphic to an object of $D^b(\coh(X\times Y)).$
\end{proof}
\begin{remark}{\rm
The last statement  also holds under conditions that $Y$ is projective and $K$ has a partially left adjoint
$K^*:\Perf(Y)\to \Perf(X).$ In this case for sufficiently large $s\gg 0$ there is a nontrivial morphism
from $\cO_{Y}(-s)$ to $\cG$ that can be lifted to a nontrivial morphism from
$\cO_{Y}(-s)$ to $K(\cF_p[m+p]).$ Therefore,
$$
\Hom(K^*(\cO_Y(-s)), \cF_p[m+p])\cong \Hom(\cO_Y(-s), K(\cF_p)[m+p])\ne 0
$$
This gives a contradiction with the fact that $K^*(\cO_Y(-s))$ is a perfect complex.
And, hence, $K$ is bounded above as well.
}
\end{remark}

\appendix
\section
{Small $\mathbb{U}$\!-cocomplete categories}

Fix universes $\mathbb{U} \in \mathbb{V},$ so that $\mathbb{U}$
contains an infinite set \cite{SGA4}.
A set $X$ is called a {\it $\mathbb{U}$\!-set} (resp. {\it $\mathbb{U}$\!-small}) if
$X\in \mathbb{U}$ (resp. $X$ is {\sf isomorphic} to an element of
$\mathbb{U}$). Similarly for groups, vector spaces, etc.

We
call $\cC$ a {\it $\mathbb{U}$\!-category} if for each $A,B \in \cC$ the set
$\Hom (A,B)$ is $\mathbb{U}$\!-small.
We call a $\mathbb{U}$\!-category (or a DG $\mathbb{U}$\!-category) $\cC$ {\it $\mathbb{U}$\!-small}, if
the collection of objects of $\cC$ is a  $\mathbb{U}$\!-small set.
Also a $\mathbb{U}$\!-category is called
{\it essentially $\mathbb{U}$\!-small} if isomorphism classes of its objects
form a $\mathbb{U}$\!-small set.
Denote by $\cA b _{\mathbb{U}}$ the category of $\mathbb{U}$\!-small
abelian groups.

A triangulated $\mathbb{U}$\!-category is called
{\it $\mathbb{U}$\!-cocomplete} if it has all $\mathbb{U}$\!-small direct
sums, i.e. it has a direct sum of any collection of its objects
which is indexed by a $\mathbb{U}$\!-small set.

We would like to translate some well known fact about cocomplete
triangulated categories into the language of triangulated
$\mathbb{U}$\!-categories which are $\mathbb{U}$\!-cocomplete.

Fix a triangulated $\mathbb{U}$\!-category $\cT$ which is
$\mathbb{U}$\!-cocomplete. Recall that a set $S\subset \cT$ {\it
generates} $\cT$ if whenever $\Hom (A,X[n])=0$ for all $A\in S$ and
all $n\in \mathbb{Z},$ then $X=0.$ For a subset $S\subset \cT$ denote by
$\langle S\rangle ^{\mathbb{U}}$ the smallest strictly full
triangulated subcategory which contains $S$ and is
$\mathbb{U}$\!-cocomplete.

\begin{theo}\label{Brown} (Brown representability for
$\mathbb{U}$\!-categories) Let $\cT$ be a triangulated
$\mathbb{U}$\!-category $\cT$ which is $\mathbb{U}$\!-cocomplete.
Suppose that $\cT$ is compactly generated by a $\mathbb{U}$\!-small set $S\subset \cT ^c.$

\begin{enumerate}
\item[a)] Let $H: \T^{\op}\to \cA b_{\mathbb{U}}$ be a cohomological functor
which takes $\mathbb{U}$\!-small coproducts to products. Then $H$ is
representable;
\item[b)] $\cT =\langle S\rangle ^{\mathbb{U}}.$
\end{enumerate}
\end{theo}

\begin{proof} The (simultaneous) proof is the same as
in the book of A.~Neeman  \cite{NeBook} Theorem 8.3.3.
\end{proof}

\begin{lemma}\label{A2} Let $\cT$ be a triangulated $\mathbb{U}$\!-category.
Assume that $\cS\subset \cT$ is an essentially $\mathbb{U}$\!-small
triangulated subcategory. Then the Verdier quotient $\cT /\cS$ is a
$\mathbb{U}$\!-category.
\end{lemma}

\begin{proof} The categories $\cT$ and $\cT /\cS$ have
the same objects. Fix objects $A$ and $B.$ A morphism $f$ between
$A$ and $B$ in $\cT /\cS$ is represented by a diagram
$A\stackrel{s}{\leftarrow} C\stackrel{g}{\to}B,$ where the cone $D$
of the morphism $s$ is in $\cS.$ Up to isomorphism we have a
$\mathbb{U}$\!-small set of choices for a diagram $A\to D,$ where
$D\in \cC.$ Thus up to isomorphism we have a $\mathbb{U}$\!-small set
of choices for a diagram $A\stackrel{s}{\leftarrow} C,$ such that
the cone of $s$ is in $\cS.$ Thus up to isomorphism there is a
$\mathbb{U}$\!-small set of choices for a diagram
$A\stackrel{s}{\leftarrow} C\stackrel{g}{\to}B$ as above.
\end{proof}

\begin{lemma} Let $\cT$ be a triangulated category which is
$\mathbb{U}$\!-cocomplete. Assume that $\cS \subset \cT$ is a full
triangulated subcategory which is closed under $\mathbb{U}$\!-small
direct sums in $\cT.$ Then the quotient category $\cT /\cS$ is also
$\mathbb{U}$\!-cocomplete and the functor $\pi :\cT \to \cT
/\cS$ preserves $\mathbb{U}$\!-small direct sums.
\end{lemma}

\begin{proof} Same as the proof of Lemma 3.2.10 in \cite{NeBook}.
\end{proof}

\begin{remark}{\rm Note that in the last lemma we do not know if the
quotient $\cT /\cS$ is a $\mathbb{U}$\!-category, or even if morphisms
between two objects in $\cT /\cS$ form a set.}
\end{remark}

\begin{prop} Let $\cT$ be a triangulated $\mathbb{U}$\!-category which
is $\mathbb{U}$\!-cocomplete. Let  $\cS \subset \cT$ be a full
triangulated subcategory which is closed under $\mathbb{U}$\!-small
direct sums in $\cT.$ Assume that $\cS$ is generated by a
$\mathbb{U}$\!-small set of objects in $\cS \cap \cT ^c.$ Then

\begin{enumerate}
\item[a)] the localization functor $\pi: \cT \to \cT /\cS$ has a right
adjoint functor $\mu$ which is full and faithful;
\item[b)] $\mu$ preserves arbitrary direct sums;
\item[c)] $\cT /\cS$ is a $\mathbb{U}$\!-category.
\end{enumerate}
\end{prop}

\begin{proof} d) is an immediate consequence of a). The proof of a),
b) and c) is the same as in \cite{Ne} Lemma 1.7,
Proposition 1.9 and \cite{NeBook} Lemma 9.1.7.
\end{proof}

\begin{theo} Let $\cT$ be a triangulated $\mathbb{U}$\!-category which
is $\mathbb{U}$\!-cocomplete and such that the category $\cT ^c$ is
essentially $\mathbb{U}$\!-small. Let  $\cS \subset \cT$ be a full
triangulated subcategory which is closed under $\mathbb{U}$\!-small
direct sums in $\cT.$ Assume that $\cS$ is generated by a
$\mathbb{U}$\!-small set of objects in $\cS \cap \cT ^c.$
 Let $\cC =\cT /\cS$ be the quotient category. Then

\begin{enumerate}
\item $\cC$ is generated by a $\mathbb{U}$\!-small subset of its compact
objects;
\item $\cT ^c$ maps to $\cC ^c$ under the quotient functor;
\item the induced functor $\cT^c/\cS ^c \to \cC ^c$ is full and
faithful;
\item $\cC^c$ is the idempotent completion of $\cT ^c/\cS ^c.$
\end{enumerate}
\end{theo}

\begin{proof} The same as the proof of Theorem 2.1 in \cite{Ne}.
\end{proof}

Let $\cA$ be a $\mathbb{U}$\!-small DG $\mathbb{U}$\!-category. Denote
by $\ModU\cA\subset \Mod\cA$ the full
DG subcategory consisting of all $\mathbb{U}$\!-small DG $\cA$\!-modules,
i.e. DG modules $M$ such that $M(Y)$ is a $\mathbb{U}$\!-small vector
space for each $Y\in \cA.$ It is a pretriangulated category. Put
$$\SF_{\mathbb{U}}(\cA)=
\SF(\cA)\cap \ModU\cA,\quad \Ac _{\mathbb{U}}(\cA)=\Ac (\cA)\cap \ModU\cA.
$$
Let $D_{\mathbb{U}}(\cA)=
H^0(\ModU\cA)/H^0(\Ac _{\mathbb{U}}(\cA))$ be the
corresponding derived category. The natural functor $H^0(\SF_{\mathbb{U}}(\cA))
\to D_{\mathbb{U}}(\cA)$ is an equivalence of
triangulated $\mathbb{U}$\!-categories.
(This is because every
$\mathbb{U}$\!-small DG $\cA$\!-module is quasi-isomorphic to a
$\mathbb{U}$\!-small semi-free DG module. The proof is the same as in
\cite{Ke}, Thm. 3.1 b).)

Note that triangulated
categories $H^0(\ModU\cA),\; H^0(\SF_{\mathbb{U}}(\cA)),\; H^0(\Ac _{\mathbb{U}}(\cA)),$ and
$D_{\mathbb{U}}(\cA)$ are $\mathbb{U}$\!-cocomplete and
since $\cA$ is a DG $\mathbb{U}$\!-category the image of the Yoneda DG
functor $h^\bullet :\cA \to \Mod\cA$ lies in $\ModU\cA.$

Note also that the DG category $\ModU\cA$ is DG
equivalent to a DG category which is $\mathbb{V}$\!-small.
Indeed,
call $M\in \ModU\cA$ {\it strict} if $M(Y)$ is
a $\mathbb{U}$\!-set (and not just a $\mathbb{U}$\!-small set). Denote
by $\ModUs\cA$ the full DG subcategory
consisting of strict DG modules. Clearly every DG module is DG
isomorphic to a strict one. The DG category $\ModUs\cA$ is obviously $\mathbb{V}$\!-small.

Denote $\SF ^{\str}_{\mathbb{U}}(\cA)=\SF _{\mathbb{U}}(\cA)\cap
\ModUs\cA$ and similarly for $\Ac^{\str}_{\mathbb{U}}(\cA).$ The  inclusions
$$
\ModUs\cA \subset \ModU\cA,\quad \SF ^{\str}_{\mathbb{U}}(\cA)\subset \SF_{\mathbb{U}}(\cA), \quad \Ac ^{\str}_{\mathbb{U}}(\cA)\subset \Ac
_{\mathbb{U}}(\cA)
$$
are quasi-equivalences of pre-triangulated
categories. Hence
$$
H^0(\ModUs\cA)/H^0(\Ac ^{\str}_{\mathbb{U}}(\cA))\simeq
D_{\mathbb{U}}(\cA)
$$
and the natural functor $H^0(\SF^{\str}_{\mathbb{U}}(\cA))\to D_{\mathbb{U}}(\cA)$ is an
equivalence. Thus $\SF ^{\str}_{\mathbb{U}}(\cA)$ is a
$\mathbb{V}$\!-small enhancement of $D_{\mathbb{U}}(\cA).$ In
particular, the $\mathbb{U}$\!-category $D_{\mathbb{U}}(\cA)$ is
essentially $\mathbb{V}$\!-small.

Let $L\subset D_{\mathbb{U}}(\cA)$ be a strictly full triangulated
subcategory. Then $L$ is essentially $\mathbb{V}$\!-small. Hence the
quotient $D_{\mathbb{U}}(\cA)/L$ is a $\mathbb{V}$\!-category by Lemma \ref{A2} above.

Let $\cL \subset \SF ^{\str}_{\mathbb{U}}(\cA)$ be the full DG
subcategory of objects which map to $L$ under the equivalence
$H^0(\SF ^{\str}_{\mathbb{U}}(\cA))\to D_{\mathbb{U}}(\cA).$ Then
$\cL$ is a pretriangulated category, so that $H^0(\cL)\simeq L.$
Hence also $D_{\mathbb{U}}(\cA)/L\simeq H^0(\SF^{\str}_{\mathbb{U}}(\cA))/H^0(\cL).$
Since the DG category $\SF^{\str}_{\mathbb{U}}(\cA)$ is $\mathbb{V}$\!-small the Drinfeld DG
quotient $\SF ^{\str}_{\mathbb{U}}(\cA)/\cL$ is defined, and by
Theorem \ref{Drin}
$$
H^0(\SF ^{\str}_{\mathbb{U}}(\cA)/\cL)\simeq
H^0(\SF^{\str}_{\mathbb{U}}(\cA))/H^0(\cL).
$$
Therefore
$
\SF^{\str}_{\mathbb{U}}(\cA)/\cL$ is a canonical enhancement of
$D_{\mathbb{U}}(\cA)/L.$

\begin{lemma}\label{aplem} Let $\cA$ be a $\mathbb{U}$\!-small DG $\mathbb{U}$\!-category. Then
the derived category $D_{\mathbb{U}}(\cA)$ coincides with its full
subcategory $\langle h^\bullet (\cA)\rangle ^{\mathbb{U}}.$
\end{lemma}

\begin{proof} Clearly the $\mathbb{U}$\!-small set $h^\bullet(\cA)$ of
compact objects in $D_{\mathbb{U}}(\cA)$ generates the
$\mathbb{U}$\!-cocomplete triangulated $\mathbb{U}$\!-category
$D_{\mathbb{U}}(\cA).$ So it remains to apply part b) of Theorem \ref{Brown}.
\end{proof}

\begin{prop} Let $\cC$ be a pretriangulated DG
$\mathbb{U}$\!-category.  Assume that the triangulated $\mathbb{U}$\!-category  $H^0(\cC)$ is $\mathbb{U}$\!-cocomplete and is generated by
a $\mathbb{U}$\!-small set of compact objects $\cA \subset
H^0(\cC)^c.$ Consider $\cA $ as a full DG subcategory in $\cC.$ The DG functor $\varPhi :\cC \to \ModU\cA$
defined as
$$\varPhi (X)(Y)=\Hom _{\cB}(Y,X)$$
for $X\in \cC$ and $Y\in \cA$ induces a quasi-functor $\phi:\cC\to\SF(\cA)$ such that the  functor
$H^0(\phi):H^0(\cB)\to D_{\mathbb{U}}(\cA)$ is an equivalence of
categories.
\end{prop}

\begin{proof} First one shows (as in the proof of Proposition \ref{adjnew}) that
$H^0(\phi)$ preserves direct sums.

By Brown representability theorem part b) one knows that
$H^0(\cB)=\langle \cA \rangle ^{\mathbb{U}}.$ Hence the same proof
as in Proposition \ref{Keller}  shows that $H^0(\phi)$ is full and faithful.

On the other hand, the essential image of $H^0(\phi)$ is a full trianglated
subcategory of $D_{\mathbb{U}}(\cA)$ closed under $\mathbb{U}$\!-small
direct sums and containing the $\mathbb{U}$\!-small set
$h^\bullet(\cA)$ of compact generators. Thus $H^0(\phi)$ is
essentially surjective by Lemma \ref{aplem} above.
\end{proof}

The following proposition shows that things don't change much when
we pass to a larger universe or even consider the whole category of
DG modules.

\begin{prop} Let $\cA$ be a $\mathbb{U}$\!-small DG
$\mathbb{U}$\!-category. Then
\begin{enumerate}
\item The category $D_{\mathbb{U}}(\cA)$ is a full subcategory of $D(\cA).$
\item  Let $L\subset D(\cA)$ be a localizing subcategory generated by
objects which are in $L_{\mathbb{U}}:=L\cap D_{\mathbb{U}}(\cA).$
(For example, $L$ may be compactly generated.) Then the natural
functor
$
D_{\mathbb{U}}(\cA)/L_{\mathbb{U}}\to D(\cA)/L
$
is full and faithful.
\end{enumerate}
\end{prop}

\begin{proof}
1) Consider the diagram
$M\stackrel{f}{\leftarrow} P\stackrel{s}{\to} N$
in $H^0(\Mod\cA),$ where $M,N\in H^0(\ModU\cA)$ and $s$ is a quasi-isomorphism. It
suffices to prove that there exists $Q\in H^0(\ModU\cA)$ and a morphism $g:Q\to P,$ such that
$s\cdot g$ is a quasi-isomorphism. Take $Q\in \SF _{\mathbb{U}}(\cA)$
and a quasi-isomorphism $t:Q\to N.$ Then there exists $g:Q\to P,$
such that $t=s\cdot g.$ This proves 1).

2) Following \cite{NeBook} we denote by
$D(\cA)^{(\mathbb{U})}\subset D(\cA)$ the full triangulated subcategory
of all objects $K$ such that for any collection $\{X_{\lambda}, \lambda \in \Lambda\}$
 of objects of
$D(\cA)$  any map $K\to \bigoplus_{\lambda\in \Lambda} X_{\lambda}$
factors through a direct subsum of cardinality strictly less than $\mathbb{U}.$
(These are called
$\mathbb{U}$\!-small in \cite{NeBook}, but we already use this term in a
different way). Also denote by
$$\{D(\cA)^{(\mathbb{U})}\}_{\mathbb{U}}\subset D(\cA)^{(\mathbb{U})}$$
the full triangulated subcategory of $\mathbb{U}$\!-perfect objects (see Def. 3.3.1 and 4.2.2 in \cite{NeBook}).
This category $\{D(\cA)^{(\mathbb{U})}\}_{\mathbb{U}}$ is denoted by
$D(\cA)^{\mathbb{U}}$ and its objects are called
$\mathbb{U}$\!-compact objects of $D(\cA).$ (The categories
$D(\cA)^{(\mathbb{U})}$ and $D(\cA)^{\mathbb{U}}$ are indeed
triangulated, because $\mathbb{U}$ is an infinite cardinal (Lemmas
4.1.4 and Corollary 3.3.12 in \cite{NeBook}). Since $\mathbb{U}$ is a
regular cardinal these categories are also $\mathbb{U}$\!-cocomplete
(Lemma 4.1.5 and Corollary 3.3.14 in \cite{NeBook}). It follows that
$D_{\mathbb{U}}(\cA)\subseteq D(\cA)^{\mathbb{U}}.$ Moreover, it is proved in Lemma 4.4.5 \cite{NeBook}
that in fact $D_{\mathbb{U}}(\cA)= D(\cA)^{\mathbb{U}}.$

Notice that $L_{\mathbb{U}}=\langle L_{\mathbb{U}}\rangle
^{\mathbb{U}},$ since $L$ is cocomplete and $D_{\mathbb{U}}(\cA)$ is
$\mathbb{U}$\!-cocomplete. Now Corollary 4.4.1 in \cite{NeBook} asserts that the functor
$D_{\mathbb{U}}(\cA)/L_{\mathbb{U}}\to D(\cA)/L$
is full and faithful.
\end{proof}

\section{Proof of Proposition \ref{exten}}\label{prB}

In this appendix we present a proof of Proposition 9.6. It is
essentially the same as the proof of Proposition 9.3 that is given
in \cite{Or, Or2}. We will directly follow \cite[3.4.6]{Or2} and
will use the notation of that proof.

Let $\fE$ be an exact category. Assume that it is a full exact
subcategory of an abelian category $\fA,$ i.e. $\fE\subset\fA$ is
closed under extensions in $\fA.$ We also assume that additional
property (EPI) holds:

\begin{tabular}{ll}
(EPI)&\quad a map $f$ in $\fE$ is an admissible epimorphism if and only if it is an epimorphism in $\fA.$
\end{tabular}
As in Definition \ref{ampl} we say that
the sequence $\{ P_i \; |\; i\in\bbZ\}$ of objects in $\fE$ is  ample in $\fE$
if it is ample in $\fA$ as in Definition \ref{defamp}.

Let us consider derived categories $D^*(\fE).$
Since  $\fE\subset\fA$ is an exact subcategory of an abelian category $\fA$ such that the condition
(EPI) holds and for any $C\in\fA$ there is an epimorphism $E\twoheadrightarrow C$ from $E\in \fE,$
then  the canonical functor $D^{-}(\fE)\to D^{-}(\fA)$ is an equivalence and
the functor $D^{b}(\fE)\to D^{b}(\fA)$ is  fully faithful \cite{Ke2}.

\begin{prop}
Let $\fE$ be an exact category with an ample sequence
$\{ P_i\; |\; i\in \bbZ_{\le 0} \}.$ Let $j : {\cP}\hookrightarrow D^b(\fE)$ be the natural embedding
of the full subcategory $\cP$ with objects $\Ob {\cP}:=\{  P_i\; |\; i\in \bbZ_{\le 0} \}.$
Let $F : D^b(\fE)\stackrel{\sim}{\to} D^b(\fE)$ be an autoequivalence.
Suppose that there is an isomorphism of functors
$f : j\stackrel{\sim}{\to}F|_{\cP}.$
Then $f$ can be extended to an isomorphism
$\id\stackrel{\sim}{\to}F$ on the whole category $D^b(\fE)$.
\end{prop}
\begin{proof}
A proof of this proposition is essentially the same as the proof of Proposition \ref{ext}.
We consider the canonical functor from $D^b(\fE)\to D^b(\fA).$ It is fully faithful, because any object $C\in \fA$
can be covered by a direct sum of $P_i$ that belongs to $\fE.$ This means that we can work with
$D^b(\fE)$ as a full triangulated subcategory of $D^b(\fA)$ and, in particular, we can talk about
cohomologies of a complex from $D^b(\fE)$ as objects of $\fA.$

First, since $F$ commutes with finite direct sums, the transformation $f$ extends componentwise to all finite direct sums of objects
of the category $\cP.$ Note that an object $X\in D^b(\fE)$ is isomorphic to an object of $\fA$ if and only if
$\Hom^j(P_i, X)=0$ for all $j\ne 0$ when $i\ll 0.$ It follows that in this case the object $F(X)$ is also isomorphic to an object of $\cA,$
because
$$
\Hom^j(P_i, F(X))\cong\Hom^j(F(P_i), F(X))\cong \Hom^j(P_i, X)=0
$$
for $j\ne 0$ when $i\ll 0$.

\noindent {\it Step 1.} In Step 1 we construct an isomorphism $f_X:
X\stackrel{\sim}{\to} F(X)$ for all $X\in D^b(\fE)$ that are
isomorphic to an object of $\fA.$ Let $X$ be such an object then
$X\cong H^0(X).$ We fix a morphism $v: P_i^{\oplus k}\to X$ such
that the canonical map $P_i^{\oplus k}\to H^0(X)$ is surjective in
the abelian category $\fA.$ After that the proof of Step 1 is the
same as the proof of Step 1 in \cite[3.4.6]{Or2}.

\noindent
{\it Step 2.} Now we show that $f_X$ does not depend on the choice of the morphism $v: P_i^{\oplus k}\to X.$
The proof of Step 2 is the same as the proof of Step 2 in \cite[3.4.6]{Or2}.

\noindent {\it Step 3.} In Step 3 we check that the morphisms $f_X$
define a natural transformation of functors on the subcategory of
$D^b(\fE)$ consisting of all objects that are isomorphic to objects
of the abelian category $\fA.$ That is, for any morphism of such
objects $\varphi: X\to Y,$ we have to prove that
$f_Y\cdot\varphi=F(\varphi)\cdot f_X.$ The proof of this step is
word for word as the proof of Step 3 in \cite[3.4.6]{Or2}.

\noindent {\it Step 4.} We constructed transformations $f_X: X\to
F(X)$ for all $X\in D^b(\fE)$ that are isomorphic to objects of the
abelian category $\fA.$ Now we define $f_{X[n]} : X[n]\to
F(X[n])\cong F(X)[n]$ for any  such $X$ by evident formula
$f_{X[n]}=f_X [n].$ We need to show that these transformations
commute with any $u\in\Hom(X, Y[k])$ for all $k>0.$ For abelian
category we used the fact that any element $u\in\Hom(X, Y[k])$ can
be represented as a composition $u = u_k \cdots u_1$ of some
elements $u_i\in\Hom(Z_{i-1}, Z_{i}[1])$ with $Z_0=X,\; Z_k=Y.$ That
allowed us to reduce the problem to the case $u\in\Hom(X, Y[1]).$ In
the case of $D^b(\fE)$ we should correct the argument.

Let $X$ and $Y$ be objects of $D^b(\fE)$ that are isomorphic to
objects of the abelian category $\fA.$ A fiber of a morphism $u\in
\Hom (X, Y[k])$ when $k\ge 2$ can be represented by a complex over
$\fA$ of the form $ C^{\cdot}: C^{-k+1}\to\cdots\to C^0 $ which has
only two nontrivial comomologies $H^0(C^{\cdot})\cong X$ and
$H^{-k+1}(C^{\cdot})\cong Y.$ By Remark \ref{exactful} there are a
complex $E^{\cdot}: \{\cdots\stackrel{d^{-2}}{\lto}
E^{-1}\stackrel{d^{-1}}{\lto} E^0\}$ over $\fE$ and a
quasi-isomorphism $E^{\cdot}\to C^{\cdot}.$ Consider the usual
truncation of the form $\tau_{\ge -k+1}E^{\cdot}: \{ \coker
d^{-k}\to E^{-k+2}\to\cdots\to E^0\}$ and the induced
quasi-isomorphism $\tau_{\ge -k+1}E^{\cdot}\to C^{\cdot}.$ It is
easy to see that the object $\im d^{-1}$ as a fiber of the map
$E^0\to X$ is isomorphic to an object of $D^b(\fE).$ Similarly, each
$\im d^{-i}=\Ker d^{-i+1}$ for $2\le i\le k-1$ is also isomorphic to
an object of $D^b(\fE).$ And, finally, $\coker d^{-k}$ as an
extension of $\im d^{-k+1}$  and $Y$ is also isomorphic to an object
of $D^b(\fE).$ Now we put $Z_0=X,\; Z_k=Y$ and $Z_i=\im d^{-i}$ for
$i=1,\dots, k-1$ with corresponding $u_i\in \Hom(Z_{i-1},
Z_{i}[1]).$ Thus, the morphism $u$ can be represented as a
composition $u = u_k \cdots u_1.$ Therefore, it is sufficient to
verify that $f_{X[n]}$ commutes with elements $u\in \Hom(X, Y[1]).$

The rest of the proof is the same as the corresponding part of the proof of Step 4 in \cite[3.4.6.]{Or2}.

\noindent
{\it Step 5.}
We carry out the final part of the proof  by induction on the length of
the interval to which the non-trivial cohomology of the object belongs. For this consider the full subcategory
$j_n: \cD_n\hookrightarrow D^b(\fE)$ consisting of objects with non-trivial cohomologies in some interval of lenght $n$
(the interval is not fixed).
We now prove that there is a unique extension of the natural transformation $f$ to a natural
functorial isomorphism $f_n: j_n\to F|_{\cD_n}.$ We have already proved this above for $n=1,$ as the basis of the induction.

Now to prove the induction step, suppose that the assertion is
already proved for some $n=a\ge 1.$ Let $X$ be an object of
$\cD_{a+1}.$ In Step 5 we construct an isomorphism $f_X: X\to F(X)$
for all $X\in \cD_{a+1}.$ The proof is the same as the proof of Step
5 in \cite[3.4.6]{Or2}.

\noindent {\it Step 6.} Now we have to prove that the isomorphism
$f_X$ does not depend on the choices made in the construction in
Step 5. The proof is word for word as the proof of Step 6 in
\cite[3.4.6]{Or2}.

\noindent
{\it Step 7.}
We constructed an isomorphisms $f_X: X\to F(X)$ for all $X\in \cD_{a+1}.$ It remains to show that
this extension of $f_a$ is  a natural transformation from $j_{a+1}$ to $F|_{\cD_{a+1}}.$
Thus, we have to check that for any $\varphi: X\to Y$ with $X$ and $Y$ in $\cD_{a+1}$ we obtain a commutative diagram
$$
\begin{CD}
X @>{\varphi}>> Y\\
@V{f_X}VV       @VV{f_Y }V\\
F(X) @>{F(\varphi)}>> F(Y).
\end{CD}
$$
We reduce this problem to the case in which both objects $X$ and $Y$ belong to $\cD_a.$
The proof is the same as the proof of Step 7 in \cite[3.4.6]{Or2}.
%
\end{proof}

\end{document}